\documentclass[11pt]{article}

\usepackage{amsmath, amsthm}
\usepackage{amsmath, amsfonts}
\usepackage{amsmath, amssymb}
\usepackage{amsmath}
\usepackage[all]{xy}

\newtheorem{theorem}{Theorem}[subsection]
\newtheorem{definition}[theorem]{Definition}
\newtheorem{definition-lemma}[theorem]{Definition/Lemma}
\newtheorem{definition-explanation}[theorem]{Definition/Explanation}
\newtheorem{explanation-definition}[theorem]{Explanation/Definition}
\newtheorem{definition-fact}[theorem]{Definition/Fact}
\newtheorem{lemma}[theorem]{Lemma}
\newtheorem{lemma-definition}[theorem]{Lemma/Definition}
\newtheorem{proposition}[theorem]{Proposition}

\newtheorem{remark}[theorem]{\it Remark}

\newtheorem{assumption}[theorem]{\it Assumption}

\newtheorem{example}[theorem]{Example}
\newtheorem{example-definition}[theorem]{Example/Definition}

\newtheorem{observation}[theorem]{Observation}

\newtheorem{definition-prototype}[theorem]{Definition-Prototype}
\newtheorem{question}[theorem]{Question}

\numberwithin{equation}{subsection}

\newtheorem{sdefinition-lemma}[stheorem]{Definition/Lemma}
\newtheorem{sdefinition-explanation}[stheorem]{Definition/Explanation}
\newtheorem{sexplanation-definition}[stheorem]{Explanation/Definition}
\newtheorem{sdefinition-fact}[stheorem]{Definition/Fact}

\newtheorem{slemma-definition}[stheorem]{Lemma/Definition}

\newtheorem{sexample-definition}[stheorem]{Example/Definition}

\newtheorem{sdefinition-prototype}[stheorem]{Definition-Prototype}

\newtheorem{sstheorem}{Theorem}[subsubsection]

\newtheorem{ssdefinition-lemma}[sstheorem]{Definition/Lemma}
\newtheorem{ssdefinition-explanation}[sstheorem]{Definition/Explanation}
\newtheorem{ssexplanation-definition}[sstheorem]{Explanation/Definition}

\newtheorem{sslemma-definition}[sstheorem]{Lemma/Definition}

\newtheorem{ssconjecture}[sstheorem]{Conjecture}

\newtheorem{ssexample-definition}[sstheorem]{Example/Definition}

\newtheorem{ssdefinition-prototype}[sstheorem]{Definition-Prototype}
\newtheorem{ssquestion}[sstheorem]{Question}

\newcommand{\Arg}{\mbox{\it Arg}\,}

\newcommand{\scriptsizecan}{\mbox{\scriptsize\it can}\,}

\newcommand{\Endsheaf}{\mbox{\it ${\cal E}\!$nd}\,}

\newcommand{\GL}{\mbox{\it GL}}

\newcommand{\Graph}{\mbox{\it Graph}\,}

\newcommand{\Homsheaf}{\mbox{\it ${\cal H}$om}\,}

\newcommand{\Id}{\mbox{\it Id}}

\newcommand{\Imaginary}{\mbox{\it Im}\,}
\newcommand{\Image}{\mbox{\it Im}\,}

\newcommand{\indexx}{\mbox{\it index}\,}

\newcommand{\MStiny}{\mbox{\tiny\it MS}}

\newcommand{\Quot}{\mbox{\it Quot}}

\newcommand{\Real}{\mbox{\it Re}\,}
\newcommand{\Rep}{\mbox{\it Rep}\,}

 \newcommand{\scriptsizesLag}{\mbox{\scriptsize\it sLag}\,}

\newcommand{\Span}{\mbox{\it Span}\,}

\newcommand{\SU}{\mbox{\it SU}}

\newcommand{\Sym}{\mbox{\it Sym}}

\newcommand{\WPscriptsize}{\mbox{\scriptsize \it WP}\,}

\newcommand{\complexscriptsize}{\mbox{\scriptsize\it complex}}

\newcommand{\dimm}{\mbox{\it dim}\,}

\newcommand{\length}{\mbox{\it length}\,}

\newcommand{\pr}{\mbox{\it pr}}
 \newcommand{\prscriptsize}{\mbox{\scriptsize\it pr}}

\newcommand{\singularscriptsize}{\mbox{\scriptsize\it singular}\,}
\newcommand{\smoothscriptsize}{\mbox{\scriptsize\it smooth}\,}

\newcommand{\vol}{\mbox{\it vol}\,}

\begin{document}

\enlargethispage{24cm}

\begin{titlepage}

$ $

\vspace{-1.5cm} 

\noindent\hspace{-1cm}
\parbox{6cm}{\small September 2010}\
   \hspace{8cm}\
   \parbox[t]{5cm}{yymm.nnnn [math.SG]\\ D(7): A, D3, sa-flow.}

\vspace{1.2cm}

\centerline{\large\bf
 D-branes of A-type, their deformations, and Morse cobordism}
\vspace{1ex}
\centerline{\large\bf
 of A-branes on Calabi-Yau 3-folds under a split attractor flow:}
\vspace{1ex}
\centerline{\large\bf
 Donaldson/Alexander-Hilden-Lozano-Montesinos-Thurston/}
\vspace{1ex}
\centerline{\large\bf
  Hurwitz/Denef-Joyce meeting Polchinski-Grothendieck}

\bigskip
\bigskip

\centerline{\large
  Chien-Hao Liu
  \hspace{1ex} and \hspace{1ex}
  Shing-Tung Yau
}

\bigskip

\begin{quotation}
\centerline{\bf Abstract}

\vspace{0.3cm}

\baselineskip 12pt  
{\small
 In [L-Y5] (D(6): arXiv:1003.1178 [math.SG])
  we introduced the notion of
   Azumaya $C^{\infty}$-manifolds with a fundamental module
   and morphisms therefrom to a complex manifold.
 In the current sequel,
 we use this notion to give a prototypical definition of
  supersymmetric D-branes of A-type (i.e. A-branes)
   -- in an appropriate region of the Wilson's theory-space
      of string theory --
  as special Lagrangian morphisms
   with a unitary, minimally flat connection-with-singularity.
 This merges Donaldson's picture of special Lagrangian submanifolds
   (1999) and the Polchinski-Grothendieck Ansatz for D-branes
   in a Calabi-Yau space
  (Sec.~2.1).
 The Higgsing/un-Higgsing and the large- vs.\ small-brane wrapping
  of A-branes in string theory can be achieved
  via deformations of such morphisms
  (Sec.~2.2 and Sec.~2.3).
 {For} the case of Calabi-Yau $3$-folds,
  classical results of
    Alexander (1920), Hilden (1974) and Montesinos (1976),
    Thurston (1982), and Hilden-Lozano-Montesinos (1983)
   on $3$-manifolds branched-covering $S^3$
  implies that
   any embedded special Lagrangian submanifold
    with a complex vector bundle with a unitary flat connection
    on a Calabi-Yau $3$-fold
   is the image of a special Lagrangian morphism from
     an Azumaya $3$-sphere with a fundamental module,
     with a unitary minimally flat connection.
 This suggests a genus-like expansion of the path-integral of D3-branes
   in type IIB string theory compactified on Calabi-Yau $3$-folds
  that resembles the genus expansion of the path-integral of strings
  (Sec.~2.4.2).
 Similarly,
  for the path-integral of D2-branes and M2-branes respectively.
 In Sec.~3,
  we use the technical results of Joyce (2002-2003)
    on desingularizations of special Lagrangian submanifolds
    with conical singularities
   to explain
  how supersymmetric D3-branes thus defined can be driven
   and re-assemble under a reverse split attractor flow
   at a point on the wall of marginal stability
   in Type IIB superstring theory compactified
      on varying Calabi-Yau $3$-folds, studied by Denef (2001).
 This last section is to be read alongside the works
  [De3] (arXiv:hep-th/0107152) of Denef  and
  [Joy3: V] (arXiv:math.DG/0303272) of Joyce.
 {To} cover the basic type of deformations of morphisms
  from Azumaya spaces in this note and its sequel,
 we discuss in Sec.~1
  Morse cobordisms of manifolds and
  their promotion to Morse cobordisms
   of Azumaya manifolds with a fundamental module, and
   of morphisms from Azumaya manifolds to complex manifolds.
 The notion of cone of special Lagrangian cycles
  in a Calabi-Yau manifold is brought out in Sec.~2.4.1
  for further study.
 A summary of the needed results from Joyce is given in the appendix.
} 
\end{quotation}

\bigskip

\baselineskip 12pt
{\footnotesize
} 

\bigskip

\noindent {\small MSC number 2010:
 53C38, 14A22; 15A54, 81T30, 81T75.
} 

\bigskip

\baselineskip 10pt
{\scriptsize
\noindent{\bf Acknowledgements.}
We thank
 Andrew Strominger, Cumrun Vafa
  for consultations on A-branes and some literature guide;
 Miranda Chih-Ning Cheng for lecture and discussions.
C.-H.L.\ thanks in addition
 Frederik Denef
  for discussions and literature guide on [De3];
 Mark Gross, Yng-Ing Lee, Wenxuan Lu
  for discussions related to special Lagrangian cycles and A-branes;
 Jiun-Cheng Chen, Pei-Ming Ho, Sean Keel, Hui-Wen Lin, Chin-Lung Wang
  for discussions on technical issues beyond;
 Yu-jong Tzeng
  for topic course, fall 2010, and discussions;
 Daniel Freed, Shinobu Hosono, Sen Hu for conversations;
 and Ling-Miao Chou for moral support.
 The project is supported by NSF grants DMS-9803347 and DMS-0074329.
} 

\end{titlepage}

\newpage
\enlargethispage{25cm}

\begin{titlepage}

$ $

\vspace{1em}

\centerline{\small\it
 Chien-Hao Liu
  dedicates this work to his early teachers (time-ordered)}
\centerline{\small\it }
\centerline{\small\it
 Ms.\ Chun-Cha Tsei,}
\centerline{\small\it
 Ms.\ Hisayo Ohashi,}
\centerline{\small\it
 Ms.\ Sui-fang Wu,}
\centerline{\small\it }
\centerline{\small\it\hspace{.6ex}
 Mr.\ Kwun-Zhu Lan,$^{\,\ast}$}
\centerline{\small\it
 Mr.\ Shao-Chia Wu,}
\centerline{\small\it }
\centerline{\small\it
 Mr.\ \& Mrs.\ Randall \& Margaret Dicus,}
\centerline{\small\it
 Ms.\ Karolyn Rice,}
\centerline{\small\it\hspace{1.2ex}
 Hr.\ Horst R.\ Gebhardt,$^{\ast\ast}$}
\centerline{\small\it }
\centerline{\small\it
 Ms.\ Ann L.\ Willman,}
\centerline{\small\it\hspace{1.8ex}
 Mr.\ Jenn-Sheng Wang,$^{\ast\ast\ast}$}
\centerline{\small\it\hspace{1.2ex}
 Ms.\ Betty Beckon,$^{\ast\ast}$}
\centerline{\small\it\hspace{1.8ex}
 Principal Mao-Ting Duan,$^{\ast\ast\ast}$}
\centerline{\small\it }
\centerline{\small\it\hspace{1.2ex}
 Fr.\ Ingrid Bj$\o$rlykskar,$^{\ast\ast}$}
\centerline{\small\it\hspace{1.2ex}
 M.\ Andre Deltour, S.J.,$^{\ast\ast}$}
\centerline{\small\it
 Ms.\ Virginia Taylor,}
\centerline{\small\it }
\centerline{\small\it
 Ms.\ Tsei-Chu Chiang,}
\centerline{\small\it
 Mr.\ Min-Nan Huang, }
\centerline{\small\it
 Ms.\ Maryline Campbell,}
\centerline{\small\it }
\centerline{\small\it
  for shaping him collectively and giving him a view of the world
  before college years.}

\vspace{1em}

{\scriptsize
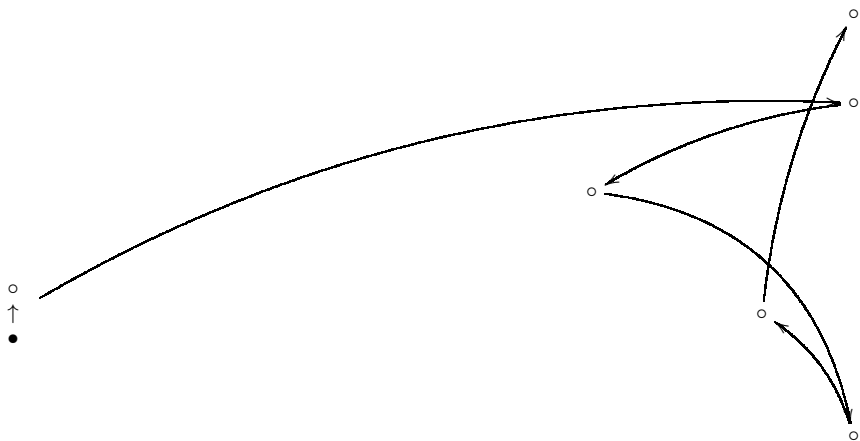
 $$
 \xymatrix{ &&&&&&& &&& **[r]\circ \\
            &&&&&&& &&& **[r]\circ \ar @/_/[llld]     \\
            &&&&&&& **[l] \circ \ar @/^3em/ [rrrdd]   \\
   \mbox{$\begin{array}{c}\circ\\ \uparrow\\ \bullet \end{array}$}
      \ar @/^3em/[rrrrrrrrrruu]
            &&&&&&& & & **[l] \circ \ar @/^/ [ruuu] & \\
            &&&&&&& & & & **[r]\circ \ar @/_/[lu]
 }
 $$
} 

\bigskip

\baselineskip 11pt
{\footnotesize
\noindent
$^{\ast}$(From C.-H.L.)
 Special tribute to {\it Mr.\ Lan},
  who led me to the beauty of mathematics.
 His constant reminder in classes:
  ``{\it
    If you cannot solve a problem,
    then there must be something basic
       you haven't yet really understood,}"
  still rings in my memory.
 $^{\ast\ast}$And
  to these missionaries,
   who spent decades of the best of their lives to serve in a place
    far away from their home.
 They set noble examples  and
  taught me the most difficult lesson on the meaning of life.
 $^{\ast\ast\ast}$And to
  {\it Mr.~Wang} and {\it Principal Duan}
   for granting me a special year to pursue and answer a call
    that influences me forever.
 It's their tolerance to a student and
   exceptional decision against all the rules and odds
  that turned an almost high-school dropout to a Princeton graduate
   several years later.
} 

\end{titlepage}

\newpage
$ $

\vspace{-3em}

\centerline{\sc A-Branes, Deformations, and
                Morse Cobordism under Split Attractor Flow}

\vspace{2em}


\begin{flushleft}
{\Large\bf 0. Introduction and outline.}
\end{flushleft}
In [L-Y5] (D(6)), we
 \begin{itemize}
  \item[$\cdot$]
   \makebox[8em][l]{([L-Y5: Sec.~2.4])}
   reviewed --
    based on the previous parts
     [L-Y1] (D(1)), [L-L-Y-S] (D(2)), [L-Y2] (D(3)),
     [L-Y3] (D(4)), and [L-Y4] (D(5))
    of the project --
   how D-branes in string theory can be re-understood
    via the following {\it ansatz}
    and a string/D-brane-theory-oriented, Grothendieck-motivated notion
    of {\it morphisms} from Azumaya-type noncommuatative spaces
     whose local functions rings are modeled in the ansatz:
  \end{itemize}
  {\bf Polchinski-Grothendieck Ansatz
  [Azumaya-type noncommutativity on D-brane].} {\it
  The world-volume of a D-brane carries a noncommutative structure
  locally associated to a function ring of the form $M_r(R)$,
  where $r\in {\Bbb Z}_{\ge 1}$ and
    $M_r(R)$ is the $r\times r$ matrix ring over $R$.$\,$\footnote{{\it
      Mathematical and String-Theoretical Remark.}
      (Partial history and
       conceptual overview from earlier parts of the project.)
      It should be noted that
      this ansatz already appeared explicitly in the work
       of {\it Pei-Ming Ho} and {\it Yong-Shi Wu}, [H-W] (1996),
        arXiv:hep-th/9611233,
       as part of their definition of `{\it D-branes as quantum space}'
       ([H-W: Sec.~5]).
      The significance of their setting seems overlooked
       by the stringy community and also by us in the brewing years.
      Ten years later, in early spring 2007,
       with a better understanding of Grothendieck's work
         on modern algebraic geometry,
       an intense re-thinking/understanding
         Polchinski's TASI 1996 lectures and textbook [Po] (1998),
       and numerous mixed/entangled influences from string theorists
        (cf.~[L-Y1: footnote~1 and footnote~19] (D(1)))
        and their works
        (cf.~[L-Y1: references] (D(1)) and [L-Y5: references] (D(6))),
      we unexpectedly re-picked up this thread and realized
        the possible variations and potentially far-reaching consequences
        of this ansatz.
      There are numerous themes yet to be understood along the line.

      In retrospect, the main obstruction
       on the mathematical side
        to reveal the power of this ansatz is the lack of
        a ``correct" notion of ``morphisms" from spaces
        with such a structure sheaf.
      The standard noncommutative geometers' approach,
       either by assigning a topological space to such type of algebras
       or by studying their category of modules via Morita equivalence,
       obscures the richness of this geometry
       and block its ability to describe D-branes correctly.
      On the string-theory side, the reason of its being overlooked
       as a fundamental nature of D-branes
       (already at the classical, nonsupersymmetric,
        not necessary stable situation)
       could be that its appearance (explicitly or hiddenly) in literature
       is constantly accompanied by something else, notably
       under a supersymmetry setting, as a quantum space,
       appearance of another type of noncommutativity structure,
       and/or with a nontrivial $B$-field background.
      That makes its role in its own right hidden behind a thick veil
       since all these additional structures are themselves important,
       much-more-discussed-and-favored subjects.
      {To} make things even worse,
      whether it's the space-time that gets noncommutatized
       matrically or the D-brane itself when D-branes are stacked
       divide the literature
       and many highly cited stringy works favor the former.
      {To} our best understanding, these two seemingly dual aspects
       are not tradable to each other.
      If there is any (world-volume)-vs.-(space-time) type duality
       between the two, it can only be at best partial
        and only in some reduced situations.
      Furthermore,
       \begin{itemize}
        \item[$\cdot$]
         mathematically: even if the space-time itself does get
         noncommutatized somehow,  only a noncommutative probe
         can detect that
         (after a correct notion of `morphism' is defined; see below),

        \item[$\cdot$]
         physically: when the space-time is treated
          as an emerged/derived notion
         and one sees some change (in particular, noncommutativity)
          to ``it",
         the first question one should ask is {\it not}
          what happens to the ``space-time", but rather,
          what happens to the probe!
       \end{itemize}
      This makes the emergence of Azumaya-type noncommutativity
       on (stacked) D-branes an unavoidable path to take.
      Different string theorists working in D-branes
       may have come across this ansatz in/on their own way/path
       with or without our (or even their own) knowing.

      With the significance of this ansatz
        being pointed out and emphasized,
      now comes the immediate technical issue:
      one needs also a
       {\it matching notion of ``maps" from D-branes to a space-time}.
      This re-started project began with the re-attempt to understand
       D-branes by working out one such notion
       that can match stringy behavior of D-branes
       as ``seen" by open strings while incorporating this ansatz into it.
      {From} our current point of view,
       it is the notion of:

       \begin{itemize}
        \item[$\cdot$]
         a ``{\it morphism}" as an
         {\it equivalence classes of gluing systems of ring-homomorphisms}
          -- as in Grothendieck's theory of schemes --
         but {\it without} assigning an underlying topological space
             nor maps between topological spaces
                 even if the latter can be assigned contravariantly
                 functorially in the commutative specialization
       \end{itemize}
       that unlocks unexpectedly, mathematically unorthodoxically,
        and yet stringy correctly
        the power of this ansatz,
       rendering many D-brane phenomena as its consequences.
      In particular,
       the main body of an Azumaya space $(X,{\cal O}_X^{A\!z})$
       is the fuzzy noncommutative cloud ${\cal O}_X^{A\!z}$,
       {\it not} $X$.
      The latter should be taken as only auxiliary.
      (The role of $X$ is only to compensate the limitedness of
        what localizations of (unital, associative) noncommutative rings
        can provide in the current situation.)
      And it is only through a morphism
        $\varphi:(X,{\cal O}_X^{A\!z})\rightarrow (Y,{\cal O}_Y)$
        defined by
        $\varphi^{\sharp}:{\cal O}_Y\rightarrow {\cal O}_X^{A\!z}$
        in the above sense,
        {\it without} a continuous map $X\rightarrow Y$,
       that the Azumaya cloud ${\cal O}_X^{A\!z}$ may ``condense" into
        an image-object that is more familiar-looking
        in the target-space(-time) geometry $Y$,
        particularly when $Y$ is commutative.

      The name of the ansatz given in this project reflects
        how it is derived ([L-Y1: Sec.~2] D(1))
       and is meant to bring it to the forefront due to its importance
        as a fundamental source of the master nature of D-branes,
        cf.\ [L-L-S-Y] (D(2)), [L-Y2] (D(3)), [L-Y3] (D(4)),
             [L-Y4] (D(5)), and [L-Y5] (D(6)).
      It indicates how it naturally arises via the fusion of two thoughts
       - one from the string-theory side
         and the other from the algebraic geometry side,
         with both re-writing/revolutionizing their own field -.
      An additional hidden reason
        we chose this name when writing [L-Y1] (D(1)), spring 2007,
       is that
      we already foresaw at that time
       that several pivotal existing stringy works on D-branes
       can be re-done via this ansatz
       -- not surprising at all for a project on D-branes
          that had to take a decade just to make the first step --  and
      the name would help us assign more easily a subtitle to its sequels,
       e.g.\ `Douglas-Moore vs.\ Polchinski-Grothendieck'
        for [L-Y2] D(3) and the subtitle of the current note,
       for comparison/contrast with and fusion of existing works.
      On the other hand, our naming is destined not to be perfect,
       (cf.\ [Joh: Introduction, 2nd paragraph, p.~xx],
             which is also applicable to the mathematical side).
      Indeed, despite our fixing to this name
       after consulting another string-theorist
        who himself has also contributed a lot to many topics
         in string theory including branes before and after 1995
        (cf.~[L-Y2: footnote~1] (D(3))),
       one may as well, if one wishes, name/call it directly
       {\it Azumaya-Type Noncommutativity Ansatz for D-Branes}
       since this is exactly what it says.
      (One may also attempt a thorough survey of related history
         on both the mathematics and the string-theory side
        of independent works
         that either exhibit or hint strongly at this ansatz
        to produce a complete name:
  {\it
  ??-Gelfand-??-Grothendieck-??-??-Ho-??-Polchinski-??-Witten-??-Wu-??
  Ansatz}.
       {For} the moment, we
        choose to focus on works that remain ahead and
        leave the naming issue to future/other historians/researchers
         in this field.)
      As this ansatz is solely a consequence of open strings (alone!),
      we believe that everything related to D-branes
       in a geometric phase always has to do with this ansatz and
       morphisms from such noncommutative spaces one way or another.

      Having said all this, as already pointed out in [L-Y1] (D(1)),
      while this ansatz (and the correct notion of morphisms)
       give us an alternative gateway
        to access D-branes in their own right,
      {\it it says only a beginning, lowest level structure thereon}
       (albeit a very rich one)
       {\it and provides
         \begin{itemize}
          \item[$(1)$]
           a ground (sheaf of local) function ring(s) with all other
           fields thereupon realized as (local sections of) its modules
           and

          \item[$(2)$]
           a basis to define the notion of morphisms from a D-brane
           (world-volume) as one would for a particle (world-line)
           and a string (world-sheet)
           in an (either commutative or noncommutative) space-time.
         \end{itemize}}
      \noindent
      The full structure on and complication of D-branes
       go much beyond this.
      {For} us, D-branes in various contexts,
       including those in a geometric phase -- as in this project --
        and other in a non-geometric phase
         -- e.g.\ as boundary conditions/states
         in a $d=2$ boundary conformal field theory --
       remain a largely mysterious yet so amazing object
       that we have still a lot to learn about.

      We thank Pei-Ming for explaining the related points in [H-W]
          with comments on our work and
       for discussions on and illuminations of D-branes and M-branes
        in late spring, 2010.
      } 
   } 

 \begin{itemize}
  \item[$\cdot$]
   \makebox[8em][l]{([L-Y5: Sec.~2.2])}
   discussed the four aspects of morphisms
    from Azumaya schemes with a fundamental module
    to a (commutative) projective scheme,  and

  \item[$\cdot$]
   \makebox[8em][l]{([L-Y5: Sec.~3.1])}
   introduced the notion of {\it Azumaya $C^{\infty}$-manifolds}
   and {\it morphisms} therefrom to a complex manifold.
 \end{itemize}
(See loc.\ cit.\ for more details.)
In the current sequel,
 we use the latter notion to give
 \begin{itemize}
  \item[$\cdot$]
   a prototypical definition of
    supersymmetric D-branes of A-type (i.e.\ {\it A-branes})
     -- in an appropriate region of the Wilson's theory-space
        of string theory --
    as {\it special Lagrangian morphisms
    with a unitary minimally flat connection-with-singularity}.
 \end{itemize}
This merges
  Donaldson's picture of special Lagrangian submanifolds (1999)
   as a special class of maps  and
  the Polchinski-Grothendieck Ansatz for D-branes
   on a Calabi-Yau space
 (Sec.~2.1).
The phenomena of
 \begin{itemize}
  \item[$\cdot$]
   {\it Higgsing/un-Higgsing} of Chan-Paton sheaves on A-branes and
    of the {\it large- vs.\ small-brane wrapping} of A-branes on a cycle
    in a Calabi-Yau space in superstring theory
 \end{itemize}
 can be achieved/realized
 via deformations of such special Lagrangian morphisms
 (Sec.~2.2 and Sec.~2.3).

{For} the case of Calabi-Yau $3$-folds,
 a classical result on $3$-manifolds branched-covering $S^3$,
   beginning with James Alexander (1920),
   completed and refined
    by Hugh Hilden (1974) and Jos\'{e} Montesinos (1976)
       in the form with a universal degree-bound $3$ and
    by William Thurston (1982) and
       Hilden, Mar\'{i}a Lozano and Montesinos (1983)
       in the form of universal links and universal knots respectively
  implies then, in particular, that
 \begin{itemize}
  \item[$\cdot$] {\it
   any embedded special Lagrangian submanifold (without singularity)
     with a complex vector bundle with a unitary flat connection
     in a Calabi-Yau $3$-fold
    is the image of a special Lagrangian morphism $\varphi$ from
      an Azumaya $3$-sphere $(S^{3,A\!z},{\cal E})$
       with a fundamental module,
      with a unitary minimally flat connection $\nabla$}.
 \end{itemize}
This singles out
  special Lagrangian morphisms from Azumaya $3$-spheres
   $(S^{3,A\!z},{\cal E})$ with a fundamental module,
   with a minimally flat connection-with-singularity,
 as the most basic, seed-like D3-branes of A-type
  in a Calabi-Yau $3$-fold.
Thus,
 replacing maps from strings moving along in space-time
  in string theory
 by maps from Azumaya $3$-spheres $S^{3,A\!z}$
 suggests
 \begin{itemize}
  \item[$\cdot$]
   a {\it genus-like expansion of the path-integral of D3-branes}
    in type IIB string theory compactified on Calabi-Yau $3$-folds
    that resembles the genus expansion
    of the path-integral of strings
 \end{itemize}
 (Sec.~2.4.2).
Similarly, for
 the path-integral of D2-branes in type IIA string theory
  compactified on Calabi-Yau $3$-folds and
 the path-integral of M2-branes in M-theory
  compactified on Joyce/$G_2$ $7$-manifolds.

In Sec.~3,
 we use the technical results of Dominic Joyce (2002-2003)
   on desingularizations of special Lagrangian submanifolds
   with conical singularities
  to explain
 \begin{itemize}
  \item[$\cdot$]
   {\it how supersymmetric D3-branes thus defined can be driven
    and re-assemble under a reverse split attractor flow
    at a point on the wall of marginal stability
    in Type IIB superstring theory compactified
     on varying Calabi-Yau $3$-folds},
 \end{itemize}
 studied by Frederik Denef (2001).
This last section is to be read alongside the works
 [De3] (arXiv:hep-th/0107152) of Denef  and
 [Joy3: V] (arXiv:math.DG/0303272) of Joyce.

{To} cover the basic type of deformations of morphisms
 from Azumaya spaces in this note and its sequel,
we discuss in Sec.~1
 Morse cobordisms of manifolds,
 their promotion to {\it Morse cobordisms} of Azumaya manifolds
   with a fundamental module, and
   {\it of morphisms} therefrom to complex manifolds.
The notion of {\it cone of special Lagrangian cycles}
   of a Calabi-Yau manifold --
  as a special-Lagrangian analogue to Mori cone of curves
   of a smooth projective variety --
 is also brought out in Sec.~2.4.1 for further study.
A summary of the needed results of Joyce is given in the appendix A.~1.

\bigskip

Readers are suggested to go through
  [De3] (resp.\ [Joy3: V]; [L-Y5] (D(6))) first
 to get a feel of
  split attractor flow
  (resp.\ desingularization of a special Lagrangian submanifold
           with conical singularities in a Calabi-Yau manifold;
          Azumaya geometry and morphisms from an Azumaya space)
 before reading the current note.

%
%
%
%
%
%
%
%

\bigskip

\bigskip

\noindent
{\bf Convention.}
 Standard notations\footnote{{\it Apology}:
         With a project that incorporates/merges many things
          from various well-established mathematical and stringy
           disciplines  and
         also to take into account notations
          from earlier parts of the project,
         we find it more and more difficult to keep
          the terminology/notations/symbols distinct for different objects.
         However, what a terminology/notation/symbol means
           is usually immediately clear
          either from the context
          or from the additional label/supscript/subscript
           to that symbol.
         Listed in Convention are a few essential ones used in this note,
          each of which is almost already carved into a stone
          in its own field.
         We decide that
          it is better/more meaningful to get used to them
          rather than to try to make any change
           merely for the consistency of notations in a note.},
  terminology, operations, and facts in
  (1) physics aspects of D-branes;
  (2) (commutative) algebraic geometry/stacks;
  (3) complex geometry;
  (4) symplectic/calibrated geometry;
  (5) sheaves on manifolds;
  (6) Hodge theory;
  (7) surgery and topological cobordism theory
  can be found respectively in
  $\,$(1) [Po], [Joh];
  $\,$(2) [Hart]$\,/\,$[L-MB];
  $\,$(3) [G-H];
  $\,$(4) [McD-S]$\,/\,$[Ha-L], [Harv], [McL];
  $\,$(5) [K-S], [Dim];
  $\,$(6) [Vo];
  $\,$(7) [Mi1], [Hir].
 \begin{itemize}
  \item[$\cdot$]
   A {\it real} manifold of dimension $n$
    is called an {\it $n$-manifold}
    while a {\it complex} manifold of complex dimension $n$
    is called an {\it $n$-fold}.

  \item[$\cdot$]
   {For} a smooth/$C^{\infty}$-manifold $X$,
    \begin{itemize}
     \item[-]
      ${\cal O}_X=$ the sheaf of $C^{\infty}$-functions on $X$,

     \item[-]
      ${\cal O}_{X,{\Bbb C}}
       = {\cal O}_X^{\infty}\otimes{\Bbb C}=$
       the sheaf of complex-valued $C^{\infty}$-functions on $X$.
    \end{itemize}

  \item[$\cdot$]
   {For} a complex manifold $Y$,
    \begin{itemize}
     \item[-]
      ${\cal O}_Y=$ the sheaf of holomorphic functions on $Y$,

     \item[-]
      ${\cal O}_Y^{\infty}=$
       the sheaf of real-valued $C^{\infty}$-functions on $Y$,

     \item[-]
      ${\cal O}_{Y,{\Bbb C}}^{\infty}
       = {\cal O}_Y^{\infty}\otimes{\Bbb C}=$
       the sheaf of complex-valued $C^{\infty}$-functions on $Y$.
    \end{itemize}

  \item[$\cdot$]
   {\it $n$-fold} $M$ as an $n$-{\it dimension}al complex manifold
   vs.\ {\it $n$-fold} (branched-)cover(ing space) $L\rightarrow N$
         as a (branched-)covering map of {\it degree} $n$.

  \item[$\cdot$]
   A (real) `{\it singular $C^{\infty}$-manifold}'
     or a  (real) `{\it $C^{\infty}$-manifold with singularities}' $M$
   means a topological space that can be stratified locally finitely
    into a union of (real) $C^{\infty}$-manifolds
    in such a way that it contains an open dense
     (possibly disconnected) stratum
    that is a (possibly disconnected) $C^{\infty}$-manifold
     (of uniform dimension if disconnected).
   {For} such $M$, let $U$ be the interior of the intersection
    of maximal open $C^{\infty}$-manifold-subsets of $M$.
   Then,
    ${\cal O}_M=$ the sheaf of continuous functions on $M$
     that is $C^{\infty}$ on $U$ and
    ${\cal O}_{M,{\infty}}:= {\cal O}_M\otimes{\Bbb C}$
     its complexification.

  \item[$\cdot$]
   $D^n =$ $n$-dimensional (closed) disk/ball,
   $S^n =$ $n$-dimensional sphere, and
   $T^n$ or ${\Bbb T}^n =$ $n$-dimensional torus;
    all as real smooth manifolds.
   Particularly, $3$-{\it disk} $D^3$ vs.\ D$3$-{\it brane} $X$.

  \item[$\cdot$]
   A {\it Calabi-Yau $n$-fold} $\,Y$
     (with a specified holomorphic $n$-form)
    is denoted in full by $(Y,J,\omega,\Omega)$,
    where
     $J$ is the complex structure on $Y$,
     $\omega$ is the K\"{a}hler class
      of the underlying Ricci-flat metric on $Y$,
     $\Omega$ is a holomorphic $n$-form on $Y$
    that satisfies the identity
     $$
      \omega^n/n!\;=\;
        (-1)^{n(n-1)/2}(i/2)^n\Omega\wedge\bar{\Omega}\;\;
        (\,=\; \vol_M\,)\,.
     $$
   Thus, $\Omega$ is uniquely determined by $(J,\omega)$,
    up to a phase factor.

  \item[$\cdot$]
   {\it D}(irichlet)-brane
    vs.\ {\it d}isk
    vs.\ ${\cal D}$-module.

  \item[$\cdot$]
   The word `{\it relative}' has two different meanings:
   (1) with respect to a subset,
       e.g.\ the {\it relative} cohomology $H^{\ast}(M,S;{\Bbb R})$,
       vs.\
   (2) with respect to the base of a family,
       e.g.\ the {\it relative} cotangent sheaf $\Omega_{X/S}$.
   {For} the current work, it is (2) that most often appears.

  \item[$\cdot$]
   {\it Graph} $\Gamma$
    vs.\ the {\it global-section functor} $\Gamma (\,\cdot\,)$.

  \item[$\cdot$]
   The {\it sheaf} $\Omega^p_Y$ of holomorphic $p$-forms
     on a complex manifold $Y$
   vs.\ a {\it holomorphic $n$-form} $\Omega$
                              on a Calabi-Yau $n$-fold $Y$
   vs.\ a (smooth or holomorphic) {\it section} $\Omega$
    on a holomorphic line bundle ${\cal H}^{n,0}$
    on the moduli space ${\cal M}$ of complex deformations of $Y$.

  \item[$\cdot$]
   (1)
    $i=\sqrt{-1}$
    vs.\ $i$ as an {\it index}
    vs.\ $i$ as an {\it inclusion/embedding};$\;$
   (2)
    $\alpha$ as a {\it phase} or a {\it phase function}
    vs.\ $\alpha$ as an {\it index}.

  \item[$\cdot$]
   The {\it canonical line bundle} $K_Y$ of a complex manifold $Y$
   vs.\ a (local) {\it K\"{a}hler potential} $K$
        of a K\"{a}hler manifold $Y$
        (particularly in SUSY QFT and stringy literatures,
          in which the latter $K$ is also related to the {\it kinetic term}
          in a supersymmetric action/Lagrangian density).

  \item[$\cdot$]
   A {\it K\"{a}hler metric} $g$ on a complex manifold $Y$ is usually
    also denoted by its associated {\it K\"{a}hler $2$-form} $\omega$
    on $Y$.

  \item[$\cdot$]
   $N$ that counts the {\it number} of sets of
    minimal {\it supersymmetries} in each space-time dimension
    (e.g.\ $d=10,\,N=2\,\Rightarrow\, 32$ supercharges)
   vs.\ $N$ as a {\it manifold}
   vs.\ $N$ as a {\it tubular neighborhood}
   vs.\ $N_{Z/Y}$ as a {\it normal bundle} of $Z$ in $Y$
   (vs.\ $N$ as an unspecified (usually large) number (usually integer)
        as in the `large $N$ limit').

  \item[$\cdot$]
   {\it rank} of a locally-free sheaf/vector bundle
   vs.\ {\it rank} of an algebra
   vs.\ {\it rank} of a Lie group.

  \item[$\cdot$]
   $Z$ as a {\it subscheme/submanifold/cycle/chain}
   vs.\ $Z$ as a {\it central charge}
   (vs.\ $Z$ as a {\it partition function} in QFT and string theory.)
 \end{itemize}

\bigskip

\bigskip

\begin{flushleft}
{\bf Outline.}
\end{flushleft}
{\small
\baselineskip 12pt  
\begin{itemize}
 \item[0.]
  Introduction.

 \item[1.]
  Morphisms from Azumaya manifolds with a fundamental module
  in a Morse family.
  \vspace{-.6ex}
  \begin{itemize}
   \item[1.1]
    Morse family of manifolds with singularities.

   \item[1.2]
    Morphisms from Azumaya manifolds with a fundamental module
    in a Morse family.
  \end{itemize}

 \item[2.]
  Supersymmetric D-branes of A-type and their deformations:\\
  Donaldson meeting Polchinski-Grothendieck.
  \vspace{-.6ex}
  \begin{itemize}
   \item[2.1]
    Supersymmetric D-branes of A-type
    as morphisms from Azumaya manifolds\\ with a fundamental module.

   \item[2.2]
    Higgsing/un-Higgsing of A-branes via deformations of morphisms.

   \item[2.3]
    Large- vs.\ small-brane wrapping via deformations of morphisms.

   \item[2.4]
    Remarks/Questions/Conjectures.
    \begin{itemize}
     \item[2.4.1]
      Cones of special Lagrangian cycles.

     \item[2.4.2]
      A genus-like expansion of the path-integral
       of lower-dimensional branes:
      Alexander-Hilden-Lozano-Montesinos-Thurston/Hurwitz
       meeting Polchinski-Grothendieck.
    \end{itemize}
  \end{itemize}

 \item[3.]
  Morse cobordisms of A-branes on Calabi-Yau $3$-folds
  under a reverse split attractor flow\\
  at a wall of marginal stability:
  Denef-Joyce meeting Polchinski-Grothendieck.
  \vspace{-.6ex}
  \begin{itemize}
   \item[3.1]
    Evolution of {\it Im}$\,(e^{-i\alpha_{\Gamma}}\Omega)$
    along a $\Gamma$-attractor flow trajectory \`{a} la Denef.
   %

   \item[3.2]
    Morse cobordisms of A-branes on Calabi-Yau $3$-folds
    under a reverse split attractor flow\\
    at a wall of marginal stability.
  \end{itemize}

 %
 %

 \item[]\hspace{-1.7em}
  Appendix.
  \vspace{-.6ex}
  \begin{itemize}
   \item[A.1]
    Desingularizations of immersed special Lagrangian submanifolds
    with transverse intersections and their moduli space \`{a} la Joyce.
   %
 \end{itemize}
\end{itemize}
} 

\newpage

\section{Morphisms from Azumaya manifolds with a fundamental\\ module
         in a Morse family.}

In studying deformation problems in algebraic geometry,
 there is the notion of {\it flatness}
  that characterizes a ``good family" of objects in question.
In the category $C^{\infty}$-manifolds(-with-singularity),
 one would very much like to have such a notion as well.
{For} this note, we take a Morse family from manifold/cobordism theory
 to play the role of a flat family in algebraic geometry.
The issue of deformations of A-branes studied
  in Sec.~2 and Sec.~3 of this note
 will be based on such families.
In Sec.~1.1, we extend the standard notion of
  a Morse family over an interval in ${\Bbb R}$
 to one over a general base $S\subset {\Bbb R}^l$ for some $l$.
In Sec.~1.2, we
 give/review the definition of morphisms from
  Azumaya manifolds to a complex manifold based on [L-Y5: 3.1] (D(6))
 and extend it to the case of a Morse family of Azumaya manifolds.

\bigskip

\subsection{Morse family of manifolds with singularities.}

\begin{definition-fact}
{\bf [handlebody decomposition, Mores function, singularity of
Morse type].} {\rm ([G-St], [Ki1], and [Mi].)} {\rm
 Associated to a {\it handlebody decomposition}
  $$
   M_{(0)}=D^{n+1}(=D^0\times D^{n+1})\,
    \subset\, M_{(1)}\, \subset \,\cdots\,
    \subset\, M_{(k-1)}\, \subset\, M_{(k)}=M
  $$
  of a smooth $(n+1)$-manifold $M$,
  where $M_{(i)}$ is obtained from $M_{(i-1)}$ by
   attaching a $k_i$-handle $D^{k_i}\times D^{(n+1)-k_i}$
    via a smooth embedding
    $f_i:\partial D^{k_i}\times D^{(n+1)-k_i}\rightarrow \partial M_{(i)}$
   plus a smoothing after gluing via $f_i$,
 (in notation,
   $M_{(i)}=M_{(i-1)}\cup_{f_i}(D^{k_i}\times D^{(n+1)-k_i})$),
 there is a {\it Morse function}
  $$
   h:M\rightarrow {\Bbb R}
  $$
  with critical values $t_0<t_1<\,\cdots\,<t_{k-1}<t_k$ and
       non-degenerate critical points
       $p_i\in M$ with $h(p_i)=t_i$ and $\indexx(p_i)=k_i$.
 $p_i$ corresponds the center
   $(0,0)\in D^{k_i}\times D^{(n+1)-k_i}$,
  with $D^{k_i}\times\{0\}$ the descending manifold and
       $\{0\}\times D^{(n+1)-k_i}$ the ascending manifold of $h$.
 The singular $n$-manifold $M_{t_i}:=h^{-1}(t_i)$
  arises from degenerating a smooth $n$-manifold
   $M_t:=h^{-1}(t) \simeq \partial M_{(i-1)}$, $t\in (t_{i-1}, t_i)$,
    and then deform to another smooth $n$-manifold
   $M_{t^{\prime}}:= h^{-1}(t^{\prime})
     \simeq \partial M_{(i)}$, $t^{\prime}\in (t_i,t_{i+1})$,
  by:
  \begin{itemize}
   \item[$\cdot$] For $k_i=0$,
    $M_{t_i} \simeq M_t\amalg\{p_i\}$,
     and then deform to $M_{t^{\prime}}\simeq M_t\amalg S^n$.

   \item[$\cdot$] For $k_i=1$,
    $M_{t_i} \simeq M_t$
             with two points $\{p_-, p_+\}$ identified (to $p_i$),
     and then deformed to
      $M_{t^{\prime}} \simeq$
      the (self-){\it connected sum} of $M_t$ at $\{p_-, p_+\}$.

   \item[$\cdot$] For $2\le k_i\le n-1$,
    $M_{t_i} \simeq M_t$ with an embedded $S^{k_i-1}$
     (whose tubular neighborhood
       $\nu_{M_t}(S^{k_i-1})\simeq S^{k_i-1}\times D^{n-(k_i-1)}$)
     shrunk to a point (i.e.\ $p_i$),
     and then deform to $M_{t^{\prime}}$ via evolving $p_i$
      to an $S^{n-k_i}$.
    This corresponds to a {\it surgery} of $M_t$ along $S^{k_i-1}$
     by removing $\nu_{M_t}(S^{k_i-1})$ and then filling in
     $D^{k_i}\times S^{n-k_i}$
     via the isomorphisms
      $\partial(\nu_{M_t}(S^{k_i-1}))
        \simeq S^{k_i-1}\times S^{n-k_i}
        \simeq \partial(D^{k_i}\times S^{n-k_i})$.

   \item[$\cdot$] For $k_i=n$,
    $M_{t_i}\simeq M_t$
     with a two-sided embedded $S^{n-1}$ shrunk to a point
      (i.e.\ $p_i$),
     and then deform to $M_{t^{\prime}}$ by compactifying
      $M_{t_i}-\{p_i\}$ via filling in points $p_-, p_+$.

   \item[$\cdot$] For $k_i=n+1$,
    $M_t\simeq N\amalg S^n$, $M_{t_i}\simeq N\amalg\{p_i\}$,
     and $M_{t^{\prime}}\simeq N$ for a smooth $n$-manifold $N$.
  \end{itemize}
 Situations $k_i=l$ and $k_i=(n+1)-l$ are reverse to each other,  and
 Situations $k_i=0,\, 1,\, n\,, n+1$ are the only ones
  that may change the number of connected components of $M_{t_i}$.
 {For} convenience, we will call these singular manifolds $M_{t_i}$
  that appear as a singular fiber of a Morse function $h$
  a {\it manifold with Morse-type singularities}.
}\end{definition-fact}

{For} the purpose of this note, we define the following version
 of the notion of a Morse family
 that extends the notion of Morse function slightly:

\begin{definition}
{\bf [Morse family].} {\rm
 Let $S$ be an open domain in some ${\Bbb R}^l$.
 A smooth map $\pi:X\rightarrow S$ from a smooth manifold $X$ to $S$
  is said to be
  a {\it Morse family of manifolds with singularities over $S$}
  (in short, a {\it Morse family over $S$})
  if the following conditions are satisfied:
  \begin{itemize}
   \item[$\cdot$]
    $\pi$ is a surjective;

   \item[$\cdot$]
    for all $s\in S$,
     there exists a smooth embedded curve
      $\gamma:(-\varepsilon, \varepsilon)\rightarrow S$, $\varepsilon>0$,
      with $\gamma(0)=s$
    such that
    \begin{itemize}
     \item[(1)]
      the total space
       $\gamma^{\ast}X\,
        :=\,     (-\varepsilon, \varepsilon)\times_{S}X\,
        \simeq\, \pi^{-1}(\gamma((-\varepsilon,\varepsilon)))\,$
      of the pull-back family is a smooth manifold,

     \item[(2)]
     the pull-back map
     $\gamma^{\ast}\pi :
        \gamma^{\ast}X
          \rightarrow (-\varepsilon, \varepsilon)$
     is a Morse function on $\gamma^{\ast}X$.
    \end{itemize}
  \end{itemize}
}\end{definition}

\begin{remark}
{\it $[$topologists' definition$]$.} {\rm
 A complete general definition
  of higher dimensional Morse families requires a study
  of the classification of stable singularities of smooth maps
  that allows enforced merging of simple nondegenerate singularities
  on fibers of $\pi:X\rightarrow S$ when $\dimm_{\Bbb R} S>1$.
 The very restrictive definition we use here
  is tailored to the situation of the current note.
 It corresponds to the case when no such merging occurs.
 Note also that it is important that
  {\it in the above definition,
  the fiber $X_s:=\pi^{-1}(s)$ of $\pi$ over $s\in S$
   is allowed to be disconnected}.
}\end{remark}

\begin{example}
{\bf [from a manifold with Morse-type singularities].}
{\rm
 Let $M_0$ be a $n$-manifold with Morse-type singularities
   $\{p_1,\,\ldots\,,\,p_k\}$
  via pinching a smooth $n$-manifold $M_{-}$
   along a disjoint collection of spheres $S^{\,l_i}$, $0\le l_i\le n-1$,
   with tubular neighborhood $\nu_M(S^{\,l_i})$
    a trivial $D^{n-l_i}$-bundle, $i=1,\,\ldots\,,\,k$.
 Topologically,
  a tubular neighborhood $\nu_{M_0}(p_i)$ of $p_i$ in $M_0$
   is homeomorphic to a (real) cone over $S^{\,l_i}\times S^{n-l_i-1}$.
 Denote the interval $(-\varepsilon, \varepsilon)$ by $I_{\varepsilon}$.
 Shrinking the disk bundles if necessary,
  we may assume that $\nu_M(S^{\,l_i})$, $i=1,\,\ldots\,,\,k$,
  are disjoint from each other.
 Then $M_0$ is realizable as $h^{-1}(0)$ of a Morse function
  $$
    h\, :\, X\;  \longrightarrow\;  I_{\varepsilon}
  $$
  with
   $h^{-1}(s)\simeq M_-$, for $s\in (-\varepsilon,0)$,  and
   $h^{-1}(s)\simeq M_+$, for $s\in (0,\varepsilon)$,
    where $M_+$ is a $n$-manifold obtained from $M_-$ by a surgery
     $$\mbox{$
      \left(M_--\coprod_{i=1}^k\nu_{M_-}(S^{\,l_i})\right)
        \bigcup_{\coprod_{i=1}^k g_i}
          \coprod_{i=1}^k (D^{l_i+1}\times S^{n-l_i-1})\,,
     $}$$
    with
     $g_i:\partial (D^{l_i+1}\times S^{n-l_i-1})
          \stackrel{\sim}{\rightarrow}
          \partial_iM_-
            :=\partial(\nu_{M_-}(S^{\,l_i}))
              \simeq S^{\,l_i}\times S^{n-l_i-1}$
    (with orientations taken into account).
 There is a special Morse family
  $\pi[k]:M_0[k]\rightarrow I_{\varepsilon}[k]:= I_{\varepsilon}^{\,k}$
   associated to $M_0$
  that arises from a ``prolonged/expanded realization of $h$"
  as follows:
 \begin{itemize}
  \item[(1)]
   By construction, one has an embedding over $I_{\varepsilon}$:
    $$
     \xymatrix{
      I_{\varepsilon}
        \times \left(M_--\coprod_{i=1}^k\nu_{M_-}(S^{\,l_i})\right)
       \ar @{^{(}->}[rrr] \ar[rd]_{\prscriptsize_1}  &&& X \ar[lld]^h \\
      & I_{\varepsilon}
     }\,,
    $$
    where $\pr_1$ is the projection map to the first factor.
   Let
    $$
     X\,-\,
        I_{\varepsilon}
         \times \left(M_--\coprod_{i=1}^k\nu_{M_-}(S^{\,l_i})\right)\;
     =\; \coprod_{i=1}^k\,K_i\,,
    $$
    where $K_i$ is the connected component
      that contains $p_i\in M_0=h^{-1}(0)\subset X$.
   By construction, $K_i$ is a manifold over $I_{\varepsilon}$
    with boundary
    $(I_{\varepsilon}\times (S^{\,l_i}\times S^{n-l_i-1}))
                                               /I_{\varepsilon}$.

  \item[(2)]
   Let
    $H_i[k]=I_{\varepsilon}^{\,i-1}
             \times \{0\}\times I_{\varepsilon}^{\,n-i}
           \subset I_{\varepsilon}[k]$
    be the $i$-th coordinate hyperplane of $I_{\varepsilon}[k]$.
   Consider the manifold with boundary
    $$
     I_{\varepsilon}[k]\times M_-\,
      -\,\coprod_{i=1}^k\,
           H_i[k]\times I_{\varepsilon} \times \nu_{M_-}(S^{\,l_i})
    $$
    over $I_{\varepsilon}[k]$  and
    the filling
    $$
     M_0[k]\;
      =\; \left(
            I_{\varepsilon}[k]\times M_-\,
             -\,\coprod_{i=1}^k\,
                  H_i[k]\times I_{\varepsilon}\times \nu_{M_-}(S^{\,l_i})
          \right)
          \;\bigcup_{\coprod_{i=1}^kf_i}\;
          \coprod_{i=1}^k\,\left(H_i[k]\times K_i \right)\,.
    $$
    where $f_i$ is the built-in isomorphism
    $$
     \xymatrix{
      H_i[k]\times \partial K_i \ar[rr]^-{f_i} \ar[rd]
       &&  H_i[k] \times I_{\varepsilon}
                  \times \partial \nu_{M_-}(S^{\,l_i}) \ar[ld] \\
      & I_{\varepsilon}[k]
     }\,.
    $$
   Here, as manifolds over $I_{\varepsilon}[k]$,
    the $I_{\varepsilon}$-factor of
     $H_i[k] \times I_{\varepsilon} \times \nu_{M_-}(S^{\,l_i})$
     is mapped to the $i$-th $I_{\varepsilon}$-factor
      of $I_{\varepsilon}[k]$ by the identity map,  and
    the $K_i$-factor of $H_i[k]\times K_i$ is mapped to
     the $i$-th $I_{\varepsilon}$-factor of $I_{\varepsilon}[k]$
     by the restriction $h|_{K_i}$.
   Then,
   since $M_0[k]$ is obtained from
     manifolds and gluing morphisms over $I_{\varepsilon}[k]$,
    there is a built-in morphism of smooth manifolds
    $$
     \pi[k]\, :\, M_0[k]\; \longrightarrow\; I_{\varepsilon}[k]\,.
    $$
   Furthermore, since
     $I_{\varepsilon}[k]\times M_-\,
         -\,\coprod_{i=1}^k\,
             H_i[k]\times I_{\varepsilon} \times \nu_{M_-}(S^{\,l_i})
        \rightarrow I_{\varepsilon}[k]$
      is a submersion and
     $H_i[k]\times K_i\rightarrow I_{\varepsilon}[k]$
      is a Morse family,
   $\pi[k]$ defines a Morse family.
   By construction,
    the set of critical points of $\pi[k]$ is given by
    $\coprod_{i=1}^k(H_i[k]\times \{p_i\})$,
     which is contained in $\coprod_{i=1}^k(H_i[k]\times K_i)$.
 \end{itemize}
\noindent\hspace{15.7cm}$\square$
}\end{example}

\begin{remark}
{\it $[$relative handlebody attachment$]$.}
{\rm
 The above construction of
   $\pi[k]: M_0[k]\rightarrow I_{\varepsilon}[k]$
  is equivalent to
   a step-by-step relative-handlebody attachment,
   beginning with $M_-\times (-\varepsilon,0]$,
  that avoids the singularities in the fibers of the previous family
  at each step.
}\end{remark}

{For} the conceptual appeal and convenience,
 we will borrow a terminology from complex algebraic geometry
 to define:

\begin{definition}
{\bf [expanded deformation space].}
{\rm
 The $M_0[k]$ constructed in Example~1.1.4
  will be called an {\it expanded Morse family of smoothings}
  of the manifold $M_0$ with Morse-type singularities and
 the base $I_{\varepsilon}[k]$
  an {\it expanded deformation space} of $M_0$.
} \end{definition}

\begin{example}
{\bf [expanded deformation space for connected sum].} {\rm
 In particular, when all $l_i=0$, $i=1,\,\ldots\,,\, k$,
  in Example~1.1.4,
 $M_0$ is the singular manifold obtained from identifying
  each pairs of points $\{p_{i-}, p_{i+}\}\,(\simeq S^0=\partial D^1)$,
   $i=1,\,\ldots\,,k$,
  in a possibly disconnected $n$-manifold $M_-$,  and
 $M_+$ is the (self-)connected sum of $M_-$
  at each $\{p_{i-},p_{i+}\}$.
 The construction gives then an expanded deformation space
  $S=I_{\varepsilon}[k]$ for $M_0$
  {\it with an expanded Morse family}
   $\pi[k]: M_0[k]\rightarrow I_{\varepsilon}[k]$
  of smoothings.
}\end{example}

\begin{remark}
{\it $[$mixed real-complex version$]$.}
{\rm
 In the passing, we remark that
 one may extend Definition~1.1.2 slightly by requiring instead:
  %
  \begin{itemize}
   \item[$\cdot$]
    for all $s\in S$,
     there exists
     {\it either} a smooth embedded curve
      $\gamma_1:(-\varepsilon, \varepsilon)\rightarrow S$,
       $\varepsilon>0$, with $\gamma_1(0)=s$
      such that
      \begin{itemize}
       \item[(1)]
        the total space
         $\gamma_1^{\ast}X\,
          :=\,     (-\varepsilon, \varepsilon)\times_{S}X\,
          \simeq\, \pi^{-1}(\gamma_1((-\varepsilon,\varepsilon)))\,$
        of the pull-back family is a smooth manifold  and

       \item[(2)]
       the pull-back map
       $\gamma_1^{\ast}\pi :
          \gamma_1^{\ast}X
            \rightarrow (-\varepsilon, \varepsilon)$
       is a Morse function on $\gamma_1^{\ast}X$,
      \end{itemize}
     {\it or} a smooth embedded $2$-disk
      $\gamma_2: \Delta^2_{\varepsilon}=\{z\in {\Bbb C}:|z|<\varepsilon\}$,
       $\varepsilon>0$, with $\gamma_2(0)=s$
      such that
      \begin{itemize}
       \item[(1$^{\prime}$)]
        the total space
         $\gamma_2^{\ast}X\, :=\, \Delta^2_{\varepsilon}\times_{S}X\,
          \simeq\, \pi^{-1}(\gamma_2(\Delta^2_{\varepsilon}))\,$
        of the pull-back family is a smooth manifold  and

       \item[(2$^{\prime}$)]
       the pull-back map
       $\gamma_2^{\ast}\pi :
          \gamma_2^{\ast}X \rightarrow \Delta^2_{\varepsilon}$
       is locally modelled on a complex Morse function
       on $\gamma_2^{\ast}X$.
      \end{itemize}
  \end{itemize}
 Examples of such families include
  families of nodal curves in complex algebraic geometry,
  families of complex surfaces with $A_1$-singularities, and
  conifold degenerations of Calabi-Yau $3$-folds.
}\end{remark}

\bigskip

\subsection{Morphisms from Azumaya manifolds with a fundamental module\\
            in a Morse family.}

\bigskip

\begin{flushleft}
{\bf Morphisms from Azumaya manifolds with a fundamental module
      to a complex manifold.}
\end{flushleft}
(Cf.\
 [L-Y1: Sec.~1, Sec.~2] (D(1));
 [L-L-S-Y: Sec.~1, Sec.~2.1, Sec.~2.2] (D(2));
 [L-Y5: Sec.~2.1, Sec.~2.2, Sec.~3.1] (D(6)).)
Given
  a ($C^{\infty}$-)manifold $X$,
  a locally free ${\cal O}_{X,{\Bbb C}}$-module ${\cal E}$, and
  a complex manifold $Y$,
Let $\Endsheaf_{{\cal O}_{X,{\Bbb C}}}({\cal E})$
 be the sheaf of ${\cal O}_{X,{\Bbb C}}$-module endomorphisms
 of ${\cal E}$.
It is an Azumaya algebra over ${\cal O}_{X,{\Bbb C}}$.
A {\it morphism}
 $$
  \varphi\;:\;
   (X^{A\!z},{\cal E})\,:=\,
    (X,
     {\cal O}_X^{A\!z}:= \Endsheaf_{{\cal O}_{X,{\Bbb C}}}({\cal E}),
    {\cal E})\;
   \longrightarrow\; Y
 $$
 from the {\it Azumaya manifold with a fundamental module}
  $(X^{A\!z},{\cal E})$
 to $Y$ is by definition
 an {\it equivalence class},
  in notation
   $$
    \varphi^{\sharp}\;:\; {\cal O}_{Y,{\Bbb C}}^{\infty}\;
      \longrightarrow\; {\cal O}_X^{A\!z}
   $$
  {\it of ${\Bbb C}$-algebra homomorphisms}
   from a gluing system of ${\Bbb C}$-algebras
     associated to ${\cal O}_{Y,{\Bbb C}}^{\infty}$
   to a fine-enough (with respect to the $C^{\infty}$-topology on $X$)
    gluing system of Azumaya algebras over ${\cal O}_{X,{\Bbb C}}$
    associated to ${\cal O}_X^{A\!z}$.$\,$\footnote{{\it
                    String-Theoretical Remark.}
                    This fundamental picture is
                     what makes our definition of morphisms
                     from an Azumaya space with a fundamental module
                     linked with D-branes in string theory.
                    It retraces how the
               Polchinski-Grothendieck/Azumaya-Type Noncommutativity Ansatz
                     for D-branes appears; cf.\ [L-Y1: Sec.~2.2] (D(1)).}
$\;$
In general, there is {\it no} map from $X$ to $Y$ directly.
However, the ${\cal O}_{X,{\Bbb C}}$-algebra ${\cal A}_{\varphi}$
 generated by the image ${\Bbb C}$-algebras of $\varphi^{\sharp}$
  under ${\cal O}_{X,{\Bbb C}}$
 defines a $C^{\infty}$-manifold-with-singularity $X_{\varphi}$
 -- the {\it surrogate of $X^{A\!z}$ associated to $\varphi$} --
 with structure sheaf ${\cal A}_{\varphi}$ and
 with the underlying topology $X_{\varphi}$
  canonically embedded in $X\times Y$.
By construction, ${\cal E}$ has a tautological
 ${\cal A}_{\varphi}$-module structure,
 denoted by $_{{\cal A}_{\varphi}}{\cal E}=:{\cal E}_{\varphi}$.
In summary and in a re-packaged form:

\begin{definition}
{\bf [morphism from Azumaya manifold].} {\rm
 Given an Azumaya manifold with a fundamental module
  $(X^{A\!z},{\cal E})$ and a complex manifold $Y$,
 a {\it morphism} $\varphi:(X^{A\!z},{\cal E})\rightarrow Y$
  is given by the following data:
  $$
   \xymatrix{
    {\cal E}_{\varphi}\ar@{.>}[rd] \\
     & (X_{\varphi},{\cal A}_{\varphi})
        \ar[rr]^-{f_{\varphi}} \ar[d]_{\pi_{\varphi}}
     && Y\;,\\
     & X
   }
  $$
  where
  \begin{itemize}
   \item[$\cdot$]
    $X_{\varphi}$ is $C^{\infty}$-manifold-with singularity
    that contains an open dense manifold-subset $V_{\varphi}$
     such that
     $\pi_{\varphi}|_{V_{\varphi}}:
      V_{\varphi}\rightarrow V:=\pi_{}(V_{\varphi})$
     is a covering map of finite order;

   \item[$\cdot$]
    ${\cal A}_{\varphi}$ is an ${\cal O}_{X_{\varphi},{\Bbb C}}$-algebra
     and
    $(\pi_{\varphi},f_{\varphi}):(X_{\varphi},{\cal A}_{\varphi})
                                                  \rightarrow X\times Y$
     is an embedding
      as a map of ringed topological spaces over $X$,
       with $X\times Y$
         as a smooth manifold fibered over $X$
       with an analytic structure along fibers $Y$;

   \item[$\cdot$]
    as an ${\cal A}_{\varphi}$-module,
    the support of ${\cal E}_{\varphi}$
     is $(X_{\varphi},{\cal A}_{\varphi})\;$
    (i.e.\ there exists no local section $a$ of ${\cal A}_{\varphi}$
      on some open set $U$ of $X_{\varphi}$ such that
      $a\cdot({\cal E}_{\varphi}|_U)=0$);

   \item[$\cdot$]
    ${\pi_{\varphi}}_{\ast}{{\cal E}_{\varphi}}={\cal E}$.
  \end{itemize}
 A {\it morphism} (or {\it arrow}) $\varphi_1\rightarrow \varphi_2$
  is the data $(h,\tilde{h}, \tilde{\tilde{h}})$,
  where
   $h:X_1\rightarrow X_2$ is a diffeomorphism,
   $\tilde{h}:(X_{1,\varphi_1},{\cal A}_{\varphi_1})
     \rightarrow (X_{2,\varphi_2},{\cal A}_{\varphi_2})$
    is an isomorphism of ringed topological space that lifts $h$, and
   $\tilde{\tilde{h}}: \tilde{h}^{\ast}{\cal E}_{\varphi_2}
                       \rightarrow {\cal E}_{\varphi_1}$
    is an ${\cal A}_{\varphi_1}$-module isomorphism
  such that the following diagram commutes:
  $$
   \xymatrix{
    {\cal E}_{\varphi_1}\ar@{.>}[dd]       \\
    & {\cal E}_{\varphi_2} \ar@{.>}[dd]\\
    (X_{1,\varphi_1},{\cal A}_{\varphi_1})
        \ar'[r][rrrrr]^-{f_{\varphi_1}}
        \ar[d]_{\pi_{\varphi_1}}
        \ar[dr]^{\tilde{h}}      & \hspace{2em} &&&& Y\;.\\
    X_1\ar[dr]_-{h}
     &  (X_{2,\varphi_2},{\cal A}_{\varphi_2})
           \ar[urrrr]_-{f_{\varphi_2}} \ar[d]^{\pi_{\varphi_2}}\\
    & X_2
   }
  $$
}\end{definition}

\begin{remark}
{$[\,$other aspects of morphisms from Azumaya manifolds
      with a fundamental module$\,]$.}
{\rm
 The way
  in which
   an arrow $\varphi_1\rightarrow \varphi_2$ between two morphisms
   is defined
  says that
 a morphism $\varphi:(X,{\cal E})\rightarrow Y$,
   with ${\cal E}$ of rank $r$,
  is the same as a morphism
  $$
   \phi\; :\; X\;
              \longrightarrow\; {\frak M}^{\,0^{A\!z^f}}_{\,r}\!\!(Y)\,,
  $$
  where ${\frak M}^{\,0^{A\!z^f}}_{\,r}\!\!(Y)$
   is the stack of $0$-dimensional ${\cal O}_Y$-modules of length $r$.
 ${\frak M}^{\,0^{A\!z^f}}_{\,r}\!\!(Y)$
  admits a representation-theoretical atlas
  from the Douady space of $0$-dimensional, length $r$ quotients
   of ${\cal O}_Y^{\oplus r}\,$:
  $$
   \Quot_{\mbox{\scriptsize\it Douady}}^{\,H^0}
                                 ({\cal O}_Y^{\oplus r},r)\;
    :=\;
    \{\, {\cal O}_Y^{\oplus r}
           \rightarrow \widetilde{\cal E}\rightarrow 0\,,\;
         \length\widetilde{\cal E}=r\,,\;
         H^0({\cal O}_Y^{\oplus r})
           \rightarrow H^0(\widetilde{\cal E})\rightarrow 0\,
       \}
  $$
 This is a $\GL_r({\Bbb C})$-space.
 In terms of this,
  $\phi$, and hence $\varphi$, is encoded in a map
   $$
    \tilde{\phi}\; :\; P_X\;  \longrightarrow\;
                \Quot_{\mbox{\scriptsize\it Douady}}^{\,H^0}
                                          ({\cal O}_Y^{\oplus r},r)
   $$
   of $\GL_r({\Bbb C})$-spaces,
  where $P_X$ is a principal $\GL_r({\Bbb C})$-bundle over $X$.
 See [L-Y5: Sec.~2.2] D(6) for the analogue in the realm of
  projective algebraic geometry,
  in which Douady spaces are replaced by Grothendieck's Quot-schemes.
}\end{remark}

\bigskip

\begin{flushleft}
{\bf Morphisms from a Morse family of Azumaya manifolds
      with a fundamental module.}
\end{flushleft}

\begin{definition}
{\bf [Morse family of morphisms].} {\rm
 Let
  $X_S$ be a Morse family over $S$,
  ${\cal E}_S$ be a locally free ${\cal O}_{X_S}$-module of finite rank,
    and
  $Y_S$ be an $S$-family of complex manifolds over $S$,
  $S\subset {\Bbb R}^l$ for some $l$.
 Denote the fiber of $(X_S,{\cal E}_S, Y_S)$ at an $s\in S$ by
  $(X_s,{\cal E}_s,Y_s)$.
 Then,
  a morphism
   $\varphi_S:(X_S^{A\!z},{\cal E}_S)\rightarrow Y_S$
   in the sense of Definition~1.2.1
   that takes $(X_s^{A\!z},{\cal E}_s)$ to $Y_s$
  is called a {\it Morse family of morphisms}
   (from Azumaya manifolds with Morse type singularities
    with a fundamental module to complex manifolds)
   {\it over $S$}.
}\end{definition}

\bigskip

\section{Supersymmetric D-branes of A-type and their deformations:\\
         Donaldson meeting Polchinski-Grothendieck.}

The notion of morphisms from an Azumaya manifold with a fundamental module
 to a target space(-time) gives a basic tool/language to study D-branes
 mathematically in their geometric phase.
{For} D-branes in the space(-time) that preserve
  part of the supersymmetry in
   either the related $d=2$ field theory
    on the open superstring world-sheet(-with-boundary)
   (cf.\ [H-I-V], [O-O-Y]; see also [A-L-Z], [L-Z], and [M-P-R])
  or the ($d=10$ or a lower-dimensional effective) supergravity theory
   with branes (cf.\ [B-B-St] and [M-M-M-S]),
 there are constraints on
  the morphisms and the gauge field on the fundamental module
  one needs to add to the notion of morphisms above.
Depending on
 what supersymmetry remains,
 what regime/location
  in the Wilson's theory-space of string theory we are in/at, and
 what other fields either on the background space-time and
      on the D-branes world-volume itself are brought into play,
these additional constraints may take different mathematical forms.
{For} the current note,
 we address D-branes of A-type in the sense of [O-O-Y] and [B-B-St]
 in the regime where
  the string coupling constant $g_s$ is small,
  the energy scale on the D-branes field theory and
   the related ambient supergravity theory is low,  and
  the background $B$-field is set to zero.\footnote{We
                           thank Andrew Strominger and Cumrun Vafa
                           for consultation on the absoluteness/variation
                           of the definition of A-branes.
                           It remains to us a challenging question
                            as how (the working mathematical definition
                             for) A-branes should vary/interpolate
                             when one moves around
                             in the Wilson's theory-space of string theory.}
The definition of A-branes in Sec.~2.1
  along the line of the Polchinski-Grothendieck Ansatz,
  with the above considerations taken into account,
 can be thought of as an extension of Donaldson's
  viewpoint on special Lagrangian submanifolds in Calabi-Yau spaces
   -- as a subspace of maps in a space of all maps [Don] --
 by promoting the domain of maps from
  a(n ordinary) manifold to an Azumaya manifold with a fundamental module.
(Cf.\ [L-Y5: Sec.~3] (D(6)).)
Through deformations of such morphisms,
 the most basic phenomena of D-branes,
  Higgsing/un-Higgsing and large- vs.\ small-brane wrapping,
 can be reproduced;
Sec.~2.2 and Sec.~2.3.
Two immediate themes are listed in Sec.~2.4 as guiding questions
 for further study.

\bigskip

\subsection{Supersymmetric D-branes of A-type as morphisms
            from Azumaya\\ manifolds with a fundamental module.}

It's well-known that
 an isomorphism class of complex vector bundles-with-flat-connection
  $(E,\nabla)$ of rank $r$ over a (real) smooth manifold $M$
  is given by a conjugacy class of group representations
  $\rho:\pi_1(M)\rightarrow U(r)$.
However,
 from the lesson of D-branes of B-type (e.g.\ [G-Sh] and [D-K-S]) and
 from the aspects of morphisms $\varphi:(X^{A\!z},{\cal E})\rightarrow Y$
  from an Azumaya manifold to a commutative target-space,
 the surrogate $X_{\varphi}$ of $\varphi$ in general
  has a scheme-like structure
  that contains nilpotent elements in its sheaf of local function rings.
See also [L-Y5: Remark~4.2.5] (D(6))
 for how this may be encoded in symplectic geometry.
Thus, to understand A-brane in full,
 we begin with the notion of flat connections
  on a coherent sheaf on a scheme and a $C^{\infty}$ version of this,
  and then
 give a prototypical definition of A-branes
  guided particularly by [B-B-St], [O-O-Y], and [H-I-V].

\bigskip

\begin{flushleft}
{\bf Connections on a quasi-coherent sheaf on a scheme and its flatness.}
\end{flushleft}
(Cf.\ [Be], [Bj], [Ka], and [Ko] (but without assuming smoothness);
      and [Ei], [Mat].)

\begin{definition}
{\bf [connection, curvature, and flatness].} {\rm
 Let
  $Z$ be a scheme over a base $T$  and
  ${\cal F}$ be a quasi-coherent sheaf of ${\cal O}_Z$-modules.
 Recall the canonical $T$-differential
  $d: {\cal O}_Z\rightarrow \Omega_{Z/T}$.
 An {\it $T$-connection $\nabla$} on ${\cal F}$
  is a homomorphism of ${\cal O}_T$-modules
  $$
   \nabla\; :\;  {\cal F}\;
   \longrightarrow\; \Omega_{Z/T}\otimes_{{\cal O}_Z}{\cal F}
  $$
  such that
  $$
   \nabla(fs)\;=\; df\otimes s + f\nabla s\,,
  $$
   for functions $f$ of $Z$ and sections $s$ of ${\cal F}$
    on the same open subset of $Z$.
 $\nabla$ extends to a homomorphism of ${\cal O}_T$-modules
  $$
   \nabla\;:\;
    \Omega_{Z/T}^{\,i}\otimes_{{\cal O}_Z}{\cal F}\;
    \longrightarrow\;
    \Omega_{Z/T}^{\,i+1}\otimes_{{\cal O}_Z}{\cal F}
  $$
  by
  $$
   \nabla(\omega\otimes s)\;
    =\; d\omega\otimes s + (-1)^i\omega\wedge \nabla s\,.
  $$
 In particular, one has
  $$
   \xymatrix{
    {\cal F} \ar[r]^-{\nabla}
    & \Omega_{Z/T}\otimes_{{\cal O}_Z}{\cal F} \ar[r]^-{\nabla}
    & \Omega_{Z/T}^{\,2}\otimes_{{\cal O}_Z}{\cal F}\,.
   }
  $$
  The ${\cal O}_T$-module homomorphism from the composition
   turns out to be an ${\cal O}_Z$-module homomorphism  and, hence,
  defines an $\Endsheaf_{{\cal O}_Z}({\cal F})$-valued $2$-form
   $R\in \Gamma(
           \Endsheaf_{{\cal O}_Z}({\cal F})
               \otimes_{{\cal O}_Z} \Omega_{Z/T}^{\,2})$
   on $Z$.
 $R$ is called the {\it curvature $2$-form} of $\nabla$.
 We say that $\nabla$ is {\it flat} if $R=0$.
}\end{definition}

\begin{remark}
{$[\,$existence/non-existence of (flat) connection$\,]$.} {\rm
 In general,
  a coherent or quasi-coherent ${\cal O}_Z$-module ${\cal F}$
  on a scheme $Z$ may not admit a connection.\footnote{Even
                      for a coherent sheaf ${\cal F}$
                       on a $0$-dimensional/punctual scheme $Z/{\Bbb C}$,
                      the existence of a connection already puts
                       a highly nontrivial constraint on ${\cal F}$,
                      let alone a flat connection.}
 When it does, it may not admit a flat connection.
 The difference $\nabla_1-\nabla_2$
  of two connections $\nabla_1$ and $\nabla_2$ on ${\cal F}$
  is an ${\cal O}_Z$-module homomorphism
  ${\cal F}\rightarrow \Omega_Z\otimes_{{\cal O}_Z}{\cal F}$.
 In particular,
  the structure sheaf ${\cal O}_Z$ of $Z$ (over ${\Bbb C}$)
  admits a connection, given by the canonical differential
   $d:{\cal O}_Z\rightarrow \Omega_Z$, which is also flat.
 A connection on ${\cal O}_Z$ is thus of the form $d+h$,
  where $h:{\cal O}_Z\rightarrow \Omega_Z$ is an arbitrary
   ${\cal O}_Z$-module homomorphism (i.e.\ $h\in \Gamma(\Omega_Z)$).
}\end{remark}

\begin{remark}
{$[\,$flat connection and ${\cal D}$-module$\,]$.}
{\rm ([Be].)
 Let $\Theta_{Z/T}$ ($\simeq \Omega_{Z/T}^{\vee}$ canonically)
  be the sheaf of $T$-derivations on ${\cal O}_Z$  and
 ${\cal D}_{Z/T}$ be
  be the ${\cal O}_Z$-algebra of differential operators on $Z/T$
  generated by $\Theta_{Z/T}$.
 When $Z$ is smooth over $T$,
 a flat connection $\nabla$ on ${\cal F}$ realizes
  ${\cal F}$ as a ${\cal D}_{Z/T}$-module via
  $$
   (\partial_1\,\cdots\,\partial_l)s\;
    :=\; \nabla_{\!\partial_1}\,\cdots\,\nabla_{\!\partial_l}s\,,
  $$
  where
   $s$ is a local section of ${\cal F}$ and
   $\partial_1\,,\,\cdots\,,\,\partial_l$
    are commuting local sections of $\Theta_{Z/T}$,
   all on the same open set of $Z$.
}\end{remark}


\begin{remark}
{$[$lifting and descent of flat connection$]$.} {\rm
 Given a finite morphism $\pi: Z_1\rightarrow Z_2$
   with a coherent ${\cal O}_{Z_1}$-module ${\cal F}_1$ on $Z_1$,
  let ${\cal F}_2:= \pi_{\ast}{\cal F}_1$.
 Then, in general,
  a flat connection on ${\cal F}_2$
   does not lift to a flat connection on ${\cal F}_1$
   (even if we assume in addition that ${\cal F}_1$ is flat over $Z_2$);
  nor is a flat connection on ${\cal F}_1$ descend to a flat connection
   on ${\cal F}_2$.
}\end{remark}

\noindent
However, such lifting and descent do exist in a special case
 in the analytic category:

\begin{lemma-definition}
{\bf [lifting and descent of flat connection
      under proper \'{e}tale morphism].}
{\rm
 {\it Let
  $\pi: Z_1\rightarrow Z_2$ be a proper \'{e}tale morphism
    (of schemes of finite type over ${\Bbb C}$),
  ${\cal F}_i$, $i=1,\,2\,$, be
    a locally-free coherent ${\cal O}_{Z_i}$-module
   with ${\cal F}_2=\pi_{\ast}{\cal F}_1$.
 Then
  a flat connection on ${\cal F}_1$
   descends to a flat connection on ${\cal F}_2$.}
 $\pi$ determines a direct-sum decomposition
   ${\cal F}_2|_U=\oplus_j{\cal F}_{2,U}^{(j)}$
    on small enough open sets $U$ in the analytic topology
    such that $\pi: \pi^{-1}(U)\rightarrow U$ is a disjoint union
     of biholomorphic maps.
 $\nabla$ is said to be {\it $\pi$-admissible}
  if 
      ${\cal F}_{2,U}^{(j)}$ is invariant under $\nabla|_U$  
      for all such $(U,j)$.
 In terms of this, {\it
  a $\pi$-admissible flat connection on ${\cal F}_2$
   lifts to a flat connection on ${\cal F}_1$,
  analytically locally via the canonical isomorphism
   ${\cal F}_1|_{\pi^{-1}(U)}\simeq \oplus_j{\cal F}_{2,U}^{(j)}$
   as ${\cal O}_U$-modules.}
}\end{lemma-definition}


\bigskip

\begin{flushleft}
{\bf Connections on a coherent sheaf on a scheme and its flatness
     - $C^{\infty}$ version.}
\end{flushleft}
We now give a $C^{\infty}$-version of the previous theme.\footnote{The
                           language in this note remains
                           manifold-'n'-scheme direct.
                          It is particularly tailored to fit the situation
                           of the (commutative) surrogate of a morphism
                           from an Azumaya manifold to a commutative
                           target-space.
                          One should finally bring in the notion of
                           {\it $C^{\infty}$-schemes}
                           (cf.\ [Joy4] and references therein)
                           for the completeness of language
                           to study A-branes in the current setting.}
Let
 $M$ be a ($C^{\infty}$-)manifold,
 ${\cal F}$ be a sheaf of finitely presentable
  ${\cal O}_{M,{\Bbb C}}$-modules on $M$, and
 ${\cal A}$ be a sheaf of commutative ${\cal O}_{M,{\Bbb C}}$-algebra
  that
   is finitely generated as an ${\cal O}_{M,{\Bbb C}}$-module and
   acts on ${\cal F}$, rendering ${\cal F}$ an ${\cal A}$-module as well.

\begin{definition}
{\bf [differential and derivation on ${\cal A}$,
      the sheaf $\Omega_{\cal A}$ and $\Theta_{\cal A}$].}
{\rm
 The sheaf of {\it differentials} on ${\cal A}$ (over ${\Bbb C}$)
  is the sheaf of ${\cal O}_{M,{\Bbb C}}$-modules on $M$
  that is associated to the presheaf
  $$
   U\;\longmapsto\;
    \Omega_{{\cal A}(U)}\,
     :=\, \Span_{{\cal A}(U)}\{df:f\in {\cal A}(U)\}/\!\sim,,
  $$
  where
   $\Span_{{\cal A}(U)}\{df:f\in {\cal A}(U)\}$
    is the ${\cal A}(U)$-module generated by the set
     $\{df:f\in {\cal A}(U)\}$  and
   $\sim$ is the equivalence relation on
    $\Span_{{\cal A}(U)}\{df:f\in {\cal A}(U)\}$ generated by relators:
    $$
     \begin{array}{lllcll}
      \mbox{(${\Bbb C}$-linearity)}
       &&& d(af+bf^{\prime})\,
           -\, a\,df\,-\,b\,df^{\prime}
       && \mbox{for $\;a,b\,\in\,{\Bbb C}\;$
                and $\;f,f^{\prime}\,\in\,{\cal A}(U)$}\,,  \\[.2ex]
      \mbox{(Leibniz rule)}
       &&& d(ff^{\prime})\,
           - f^{\prime}(df)\,-\, f\,df^{\prime}
       && \mbox{for $\;f,f^{\prime}\,\in\,{\cal A}(U)$}\,,  \\[.2ex]
      &&& d(f+f^{\prime})\,-\,dg
       && \mbox{whenever $\;f+f^{\prime}=g\;$ in $\;{\cal A}(U)$}\,,
                                                            \\[.2ex]
      &&& d(ff^{\prime})\,-\,dg
       && \mbox{whenever $\;\:\:\,ff^{\prime}=g\;\:\:\,$
                                              in $\;{\cal A}(U)$}\,.
    \end{array}
    $$
 $\Omega_{\cal A}$ is tautologically a sheaf of ${\cal A}$-modules.
 The sheaf $\Theta_{\cal A}$ of {\it ${\Bbb C}$-derivations} on ${\cal A}$
  is defined to be the dual sheaf
  $\Homsheaf_{\cal A}(\Omega_{\cal A}, {\cal A})$ of $\Omega_{\cal A}$
  as an ${\cal A}$-module.
 By taking the anti-symmetric tensor products over ${\cal A}$,
  one has also $\Omega_{\cal A}^{i} := \bigwedge^i\Omega_{\cal A}$,
   $i\in {\Bbb Z}_{\ge 0}$,
   with $\Omega_{\cal A}^0:= {\cal A}$ and
        $\Omega_{\cal A}^1=\Omega_{\cal A}$.
}\end{definition}

\begin{definition}
{\bf [connection on ${\cal F}$ as ${\cal A}$-module, curvature,
      and minimal flatness].} {\rm
 Denote by $_{\cal A}{\cal F}$
  the sheaf ${\cal F}$ as an ${\cal A}$-module.
 A (${\Bbb C}$-){\it connection} $\nabla$ on $_{\cal A}{\cal F}$
  is a homomorphism of ${\Bbb C}$-modules
  $$
   \nabla\; :\; _{\cal A}{\cal F}\;
   \longrightarrow\; \Omega_{\cal A}\otimes_{\cal A}{\cal F}
  $$
  such that
  $$
   \nabla(fs)\;=\; df\otimes s + f\nabla s
  $$
  for sections $f$ of ${\cal A}$ and sections $s$ of ${\cal F}$
  on the same open subset of $M$.
 $\nabla$ extends to a homomorphism of ${\cal A}$-modules
  $$
   \nabla\;:\;
    \Omega_{\cal A}^{\,i}\otimes_{\cal A}{\cal F}\;
    \longrightarrow\;
    \Omega_{\cal A}^{\,i+1}\otimes_{\cal A}{\cal F}
  $$
  by
  $$
   \nabla(\omega\otimes s)\;
    =\; d\omega\otimes s + (-1)^i\omega\wedge \nabla s\,.
  $$
 In particular, one has
  $$
   \xymatrix{
    {\cal F} \ar[r]^-{\nabla}
    & \Omega_{\cal A}\otimes_{\cal A}{\cal F} \ar[r]^-{\nabla}
    & \Omega_{\cal A}^{\,2}\otimes_{\cal A}{\cal F}\,.
   }
  $$
  The ${\Bbb C}$-module homomorphism from the composition
   turns out to be an ${\cal A}$-module homomorphism  and, hence,
  defines an $\Endsheaf_{\cal A}({\cal F})$-valued $2$-form
   $R\in \Gamma(
           \Endsheaf_{\cal A}({\cal F})
               \otimes_{\cal A} \Omega_{\cal A}^{\,2})$
   on ${\cal A}$.
 $R$ is called the {\it curvature $2$-form} of $\nabla$.
 We say that $\nabla$ is {\it flat} if $R=0$.
 Suppose that there is a ($C^{\infty}$-)manifold $M^{\prime}$ over $M$
  with an ${\cal O}_{M,{\Bbb C}}$-algebra homomorphism
   $\jmath^{\sharp}: {\cal A}\rightarrow {\cal O}_{M^{\prime},{\Bbb C}}$.
 This defines an ${\cal A}$-module homomorphism
  $j^{\ast}:
    \Omega_{\cal A}\otimes_{\cal A}{\cal O}_{M^{\prime},{\Bbb C}}
     \rightarrow \Omega_{M^{\prime},{\Bbb C}}$.
 We say that $\nabla$ is {\it flat along $(M^{\prime},\jmath^{\sharp})$}
   (or along $M^{\prime}$ when $\jmath^{\sharp}$ is understood),
   in notation $R|_{(M^{\prime},\jmath)}=0$ (or $R|_{M^{\prime}}=0$),
  if $j^{\ast}R=0$.
 Here,
  $\Omega_{M^{\prime},{\Bbb C}}
    :=\Omega_{{\cal O}_{M^{\prime},{\Bbb C}}}
     =\Omega_{M^{\prime}}\otimes{\Bbb C}$
    is the complexified cotangent sheaf of $M^{\prime}$.
 {For} convenience, we'll call such $\nabla$ also
  a {\it minimally flat connection},
  with $(M^{\prime},\jmath^{\sharp})$ being kept implicit.
}\end{definition}

\begin{remark}
{$[\,$meaning of these structures -- why we set them as above$\,]$.}
{\rm
 With the notation and the situation in Definition~2.1.7,
 when $j^{\sharp}$ is a ${\cal O}_{M,{\Bbb C}}$-algebra quotient,
 one should think of ${\cal A}$ as the manifold $M^{\prime}$
  with an extension of its standard (complexified) structure sheaf
   ${\cal O}_{M^{\prime},{\Bbb C}}$, as a $C^{\infty}$-manifold,
   extended to ${\cal A}$ by nilpotents elements.
 {For} our application to D-branes,
  such nilpotent structure is meant to encode an infinitesimal data
  of how a collection of D-branes of A-type in a space(-time) collide
  to form a single D-brane supported on $M^{\prime}$.\footnote{See
                       [G-Sh] of G\'{o}mez and Sharpe  and
                       [D-K-S] of Donagi, Katz, and Sharpe
                        for a related discussion in the case of B-branes.
                       The same behavior should also happen for A-branes
                         from the viewpoint of
                         morphisms from Azumaya manifolds
                        since the Azumaya structure sheaf contains
                         nilpotent elements.}
 In case there is also an ${\cal O}_{M,{\Bbb C}}$-algebra homomorphism
  ${\cal O}_{M^{\prime},{\Bbb C}}\rightarrow {\cal A}$
  such that the composition
   ${\cal O}_{M^{\prime},{\Bbb C}}\rightarrow {\cal A}
             \stackrel{\jmath^{\sharp}}{\rightarrow}
                                       {\cal O}_{M^{\prime},{\Bbb C}}$
   is the identity map,
  $_{\cal A}{\cal F}$ is pushed forward to $M^{\prime}$,
    becoming an ${\cal O}_{M^{\prime},{\Bbb C}}$-module.
 The notion of a connection on the ${\cal A}$-module $_{\cal A}{\cal F}$
  is then meant to be a connection on ${\cal F}$,
   as a sheaf on $M^{\prime}$,
  that commutes with these nilpotent linear operators on ${\cal F}$
   since these nilpotent elements $f$ is meant to correspond to
   infinitesimal transverse directions to $M^{\prime}$ of various order
   (and hence the evaluation of $df$ from such nilpotent $f$
     on $\Theta_{M^{\prime},{\Bbb C}}$ is meant to be zero).
 In view of this,
  it is very natural to consider connections $\nabla$
   on $_{\cal A}{\cal F}$ that are flat only along $M^{\prime}$,
   rather than all over ${\cal A}$,
  since these are flat connections on ${\cal F}$ on $M^{\prime}$
   that are compatible with the nilpotent linear operators on ${\cal F}$
   encoded in ${\cal A}$.
}\end{remark}

Similar existence/non-existence of lifting and descent statements
 as in the previous theme hold in the current category.
In particular,

%
%
%

\begin{lemma-definition}
{\bf [lifting and descent of flat connection under covering map].}
{\rm
 {\it Let
  $\pi: M_1\rightarrow M_2$ be a finite covering\footnote{Note
               that
               in algebraic geometry an {\it \'{e}tale morphism}
                may not be proper;  however
               in algebraic topology
                a {\it covering map} is by definition always proper:
                 for all $p\in M_2$, there exists an open neighborhood
                  $U\ni p$ in $M_2$ such that
                  $\pi: \pi^{-1}(U)\rightarrow U$
                  is a disjoint union of diffeomorphisms
                  (in the $C^{\infty}$ category).
               Cf.\ [Hart] vs.\ [Sp].}
   map of $C^{\infty}$-manifolds,
  ${\cal F}_i$, $i=1,\,2\,$, be locally-free
   ${\cal O}_{M_i,{\Bbb C}}$-module of finite rank on $M_i$
   with ${\cal F}_2=\pi_{\ast}{\cal F}_1$.
 Then
  a flat connection on ${\cal F}_1$
   descends to a flat connection on ${\cal F}_2$.}
 $\pi$ determines a direct-sum decomposition
   ${\cal F}_2|_U=\oplus_j{\cal F}_{2,U}^{(j)}$
    on contractible open sets $U\subset M_2$.
 A connection $\nabla$ on ${\cal F}_2$
   is said to be {\it $\pi$-admissible}
  if ${\cal F}_{2,U}^{(j)}$ is invariant under $\nabla|_U$
   for all such $(U,j)$.
 In terms of this, {\it
  a $\pi$-admissible flat connection on ${\cal F}_2$
   lifts to a flat connection on ${\cal F}_1$,
   via the canonical ${\cal O}_U$-module isomorphism
   ${\cal F}_1|_{\pi^{-1}(U)}\simeq \oplus_j{\cal F}_{2,U}^{(j)}$}.
}\end{lemma-definition}

%
%

\bigskip

\begin{flushleft}
{\bf $G$-reduced flat connections with respect to a covering.}
\end{flushleft}
Let
 $c:\hat{M}\rightarrow M$ be a covering map of finite degree $\hat{d}$
  between two (not-necessarily-compact) manifolds,
 $\hat{\cal E}$ be a locally free ${\cal O}_{\hat{M},{\Bbb C}}$-module
  of finite rank $\hat{r}$,  and
 ${\cal E}:=c_{\ast}\hat{\cal E}$,
  which is a locally free ${\cal O}_{M,{\Bbb C}}$-module
  of rank $r=\hat{r}\hat{d}$.

\begin{definition}
{\bf [$G$-reduced flat connection].} {\rm
 Let $G\subset \GL_r({\Bbb C})$
  be a subgroup of the complex general linear group.
 A flat connection $\hat{\nabla}$ on $\hat{\cal E}$
   is said to be {\it $G$-reduced with respect to $c$}
  if its descent $\underline{\nabla}$ on ${\cal E}$
   (cf.\ Lemma/Definition~2.1.9)
   has holonomy in $G$
  (i.e.\ the holonomy group of $\underline{\nabla}$ on ${\cal E}$
      at each $p\in M$ is a conjugate of $G$ in $\GL_r({\Bbb C})$).
}\end{definition}

\begin{example}
{\bf [$U(r)$-reduced flat connection w.r.t.\ $c$].} {\rm
 Take
  $\hat{\cal E}$ to be the sheaf of ($C^{\infty}$-)sections of
  a complex Hermitian vector bundle $\hat{E}$ on $\hat{M}$
   of rank $\hat{r}$
  with a compatible flat connection $\hat{\nabla}$
   (so that the parallel transports are isometries
       of the Hermitian fibers of $\hat{E}$).
 I.e.\ $(\hat{E},\hat{\nabla})$ is the descent, via a representation
  $$
   \hat{\rho}\;:\; \pi_1(\hat{M})\; \longrightarrow\;  U(\hat{r})
  $$
  (after a base-point $\hat{\ast}\in \hat{M}$ is specified implicitly
    for each connected component of $\hat{M}$
    so that $\ast:=c(\hat{\ast})\in M$ are all identical),
  of a trivialized trivial Hermitian ${\Bbb C}^{\hat{r}}$-bundle
   on the universal covering space $\widetilde{\hat{M}}$ of $\hat{M}$.
 The push-forward ${\cal E}= c_{\ast}\hat{\cal E}$ is now
  the sheaf of sections of a vector bundle $E$ of rank $r$ on $M$
   with a Hermitian structure and a flat connection $\underline{\nabla}$
   canonically induced from from $(\hat{E},\hat{\nabla})$ via $c$.
 Furthermore, by construction,
  the parallel transports determined by $\underline{\nabla}$
   are isometries of fibers of $E$.
 It follows that
  $\underline{\nabla}$ is a flat $U(r)$-connection on $E$.
 %
 %
 %
}\end{example}

\bigskip

\begin{flushleft}
{\bf Supersymmetric D-branes of A-type (i.e.\ A-branes).}
\end{flushleft}

\begin{definition}
{\bf [special Lagrangian submanifold with a phase (factor)].} {\rm
 Let
  $Y=(Y,J,\omega,\Omega)$ be a Calabi-Yau manifold  and
  $\theta\in [0,2\pi)$ (or $(-\pi,\pi]$ by convention) be a constant.
 Then
 a special Lagrangian submanifold $L$
  with respect to the calibration $\Real(e^{-\theta}\Omega)$ is called
  a {\it special Lagrangian submanifold with a phase
         factor}\footnote{{\it Terminology.}
                          We will also call this factor directly
                           a {\it phase} to synchronize with
                           some stringy literatures,
                           though it is also standard
                          to leave the latter term for $\theta$ alone.}
  $e^{i\theta}$ in $Y$.
 In particular,
  $\Omega|_L=e^{i\theta}\vol_L$ on $L$,
   where $\vol_L$ is the volume-form on $L$
    induced by the K\"{a}hler metric $\omega$.
}\end{definition}

\begin{definition}
{\bf [connection-with-singularity/singular connection].} {\rm
 Given a ($C^{\infty}$-)manifold-with-singularity $M$ and
  a finitely presented ${\cal O}_{M,{\Bbb C}}$-module ${\cal F}$.
 By a {\it connection-with-singularity} (or {\it singular connection}),
  we mean a connection $\nabla$ on ${\cal F}|_U$
   for some open dense manifold-subset $U\subset M$.
 Flatness and holonomy of $\nabla$ are defined as flatness and
  holonomy of $\nabla$ on ${\cal F}|_U$.
}\end{definition}

\begin{definition-prototype}
{\bf [A-brane (with unitary minimally flat singular connection)
      on Calabi-Yau space].} {\rm
 Let $Y=(Y,J,\omega,\Omega)$ be a Calabi-Yau $n$-fold.
 A {\it D-brane of A-type} (i.e.\ {\it A-brane})
  {\it with a phase factor} $e^{i\theta}$ on $Y$
  (in the regime of the Wilson's theory-space of string theory
   specified at the beginning of this section)
 is a morphism
  $$
   \varphi\; :\; (X^{A\!z},{\cal E})\; \longrightarrow\;  Y
  $$
  together with constraints and data encoded in the following diagram:
  $$
   \xymatrix{
    ({\cal E}_{\varphi},\nabla)\ar@{.>}[rd] \\
     & (X_{\varphi},{\cal A}_{\varphi})
        \ar@{->>}[rr]^-{f_{\varphi}}\ar[d]_{\pi_{\varphi}}
     && L\;\subset \; Y\;,\\
     & X
   }
  $$
  where the following Properties (1) -- (5) hold:
  \begin{itemize}
   \item[(1)]
    $$
     \xymatrix{
      {\cal E}_{\varphi}\ar@{.>}[rd] \\
       & (X_{\varphi},{\cal A}_{\varphi})
         \ar@{->>}[rr]^-{f_{\varphi}}\ar[d]_{\pi_{\varphi}}
       && Y\\
       & X
     }
    $$
    is the surrogate and the related maps and sheaves
     associated to $\varphi\,$.

   \item[(2)]
    The underlying $n$-manifold $X$ in $(X^{A\!z},{\cal E})$
     is oriented;
    the underlying singular $n$-manifold
     $X_{\varphi}$ in $(X_{\varphi},{\cal A}_{\varphi})$
     is equipped with the induced orientation from that of $X$
     via $\pi_{\varphi}\,$.

   \item[(3)]
    $L=\Image\varphi$
     is a special Lagrangian singular submanifold
     with a phase factor $e^{i\theta}$ in $Y\,$.

   \item[(4)]
    There exists an open dense submanifold
     $V_{\varphi}\subset X_{\varphi}$ such that
     \begin{itemize}
      \item[(4.1)]
       $V:= \pi_{\varphi}(V_{\varphi})$
        is an open dense submanifold of $X\,$;\\
       $\pi_{\varphi|_{V_{\varphi}}}:V_{\varphi} \rightarrow V$
        is a covering map,

      \item[(4.2)]
       $f_{\varphi}|_{V_{\varphi}}:
        V_{\varphi}\rightarrow Y$ is an immersion, and

      \item[(4.3)]
       $f_{\varphi}|_{V_{\varphi}}:V_{\varphi}\rightarrow L$
       is orientation-preserving.
     \end{itemize}
    Note that there are then
     built-in ${\cal O}_{V_{\varphi},{\Bbb C}}$-algebra homomorphisms
     $$
      {\cal O}_{V_{\varphi},{\Bbb C}}\;
               \longrightarrow {\cal A}_{\varphi}|_{V_{\varphi}}\;
               \longrightarrow {\cal O}_{V_{\varphi},{\Bbb C}}
     $$
     with the composition being the identity map.

   \item[(5)]
    Let $r$ be the rank of ${\cal E}$;
    $\nabla$ is a {\it singular} connection
      on $_{{\cal A}_{\varphi}}{\cal E}_{\varphi}$
     that is defined on and is flat along an open dense subset of
      $V_{\varphi}\subset X_{\varphi}$ in Item (4)
      with holonomy, when descends to $V\subset X$,
      in a subgroup of $\GL_r({\Bbb C})$
       that is isomorphic to the unitary group $U(r)$.
  \end{itemize}
 On the mathematical side, we will call the above data
  a {\it special Lagrangian morphism
         with a unitary minimally flat
          connection(-with-singularity)}$^{\,}$\footnote{{\it
             String-Theoretical Remark
              $[\,$pure D-brane vs.\ D-brane smearing$\,]$.}
             The gauge field $A$ (i.e.\ connection $\nabla$)
              on (the Chan-Paton module/sheaf of) a D-brane
              is the most fundamental field thereupon besides
              fields that govern the deformations of the brane
              (i.e.\ $\varphi$ in our setting).
             In general, when the gauge field strength $F_A$
               (i.e.\ curvature) of $A$ is nonzero,
              $A$ can couple, via $F_A$, with the Ramond-Ramond fields $C$
               on target the space-time and
              serves as a source/charge for $C$
               as if there are lower-dimensional D-branes
                that are smeared along the D-brane one begins with.
             (See, e.g.,
              [Doug] (1995) for an early discussion,
              [D-M: Sec.~2.1] (2007)
               for the case of supersymmetric D-branes of B-type, and
              [Joh: Chapter~9] (2003)
               for a related highlight/review.)
              {From} this aspect, a D-brane with a gauge field that is flat
               (i.e.\ $F_A=0$) is special
              in the sense that it is {\it pure} without being mixed
               implicitly/effectively with lower-dimensional D-branes.
             In our situation,
             such flat connection $\nabla$ is defined only
              on an open dense manifold-subset of $X_{\varphi}$
              and can have singularity around the singular locus
               $X_{\varphi,\mbox{\tiny\it sing}}$
               of the surrogate $X_{\varphi}$ of $\varphi$.
             $\nabla$ may still have non-trivial holonomy/monodromy
               for a small meridian circle
               around  $X_{\varphi,\mbox{\tiny\it sing}}$.
             In other words,
              the curvature of $\nabla$ now can be
               a Lie-algebra-valued distribution-like $2$-form and
               be supported on $X_{\varphi,\mbox{\tiny\it sing}}$.
             When this happens, it indicates that
              lower-dimensional D-branes are smeared effectively
              along $X_{\varphi,\mbox{\tiny\it sing}}$.}
  and denote it collectively by $(\varphi,\nabla)$.

 A {\it morphism} (or {\it arrow})
  $(\varphi_1,\nabla)\rightarrow (\varphi_2,\nabla)$
  is a morphism/arrow $\varphi_1\rightarrow \varphi_2$,
   encoded by the data $(h,\tilde{h}, \tilde{\tilde{h}})$
   in Definition~1.2.1,
  that satisfies the additional condition that
   the ${\cal A}_{\varphi_1}$-module isomorphism
   $\tilde{\tilde{h}}: \tilde{h}^{\ast}{\cal E}_{\varphi_2}
                                \rightarrow {\cal E}_{\varphi_1}$
   takes $\tilde{h}^{\ast}\nabla_2$ to $\nabla_1$
   over an open dense subset of $X_1$.
}\end{definition-prototype}

\begin{remark}
{$[\,$from A-brane to constructible sheaf and perverse sheaf$\,]$.}
{\rm
 The pair $(_{{\cal A}_{\varphi}}{\cal E}_{\varphi},\nabla)$
  reminds one very strongly of constructible sheaves and, hence,
  perverse sheaves on a stratified manifold-with-singularity.
 If would be very interesting if such a link/correspondence
  can truly be built naturally/functorially.
}\end{remark}

\begin{example}
{\bf [simple A-brane (with unitary flat connection-with-singularity)].}
{\rm
 A special class of A-branes in the sense of Definition~2.1.14
  with ${\cal A}_{\varphi}|_{V_{\varphi}}={\cal O}_{V_{\varphi},{\Bbb C}}$
  can be constructed from the following data
  $((c,f), (\hat{\cal E},\hat{\nabla}))\,$:
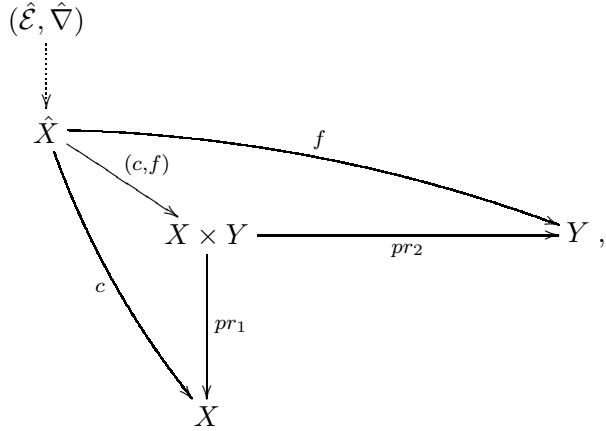
  $$
   \xymatrix{
    (\hat{\cal E},\hat{\nabla})\ar@{.>}[d]   & & & & & \\
    \hat{X} \ar[rd]^{(c,f)} \ar@/^2ex/[rrrrrd]^-{f}
            \ar@/_/[rddd]_-{c}               & &       \\
     & X\times Y \ar[rrrr]_-{pr_2}  \ar[dd]^-{pr_1} & & & & Y\;, \\ \\
     & X
   }
  $$
  where
   \begin{itemize}
    \item[$\cdot$]
     $Y=(Y,J,\omega,\Omega)$ is a Calabi-Yau $n$-fold;

    \item[$\cdot$]
     $X$, $\hat{X}$ are closed oriented $n$-manifold and
     $c:\hat{X}\rightarrow X$ is an orientation-preserving
       branched covering map of finite degree $\hat{d}$
       over a codimension-$2$ submanifold;

    \item[$\cdot$]
     $f:\hat{X}\rightarrow Y$
      is a smooth map that is an immersion on
       an open dense subset $\hat{V}\subset \hat{X}$
      such that
       $f|_{\hat{V}}$ defines a special Lagrangian submanifold of
        phase factor $e^{i\theta}$ on $Y$;
      (thus, $f:\hat{X}\rightarrow Y$
        defines a special Lagrangian submanifold-with-singularity
        with a phase factor $e^{i\theta}$ in $Y$);

    \item[$\cdot$]
     $(\hat{\cal E},\hat{\nabla})$
      is a locally free ${\cal O}_{\hat{X},{\Bbb C}}$-module
       of finite rank $\hat{r}$ with a flat $U(\hat{r})$-connection.
   \end{itemize}
 {From} this data,
 one can recover the underlying special Lagrangian morphism
  with a minimally flat connection-with-singularity
  $\left(\varphi:(X^{A\!z},{\cal E})\rightarrow Y\,,\,\nabla\right)$
  as follows:
  \begin{itemize}
   \item[$\cdot$]
    \makebox[13ex][l]{\it $[\,$domain$\,]$}
    $(X^{A\!z}, {\cal E})
     =(X,
       {\cal O}_{X,{\Bbb C}}^{A\!z}
        =\Endsheaf_{{\cal O}_{X,{\Bbb C}}}({\cal E}),
       {\cal E}:= c_{\ast}\hat{\cal E} )\,$;\\
     \makebox[13ex]{}
    as a ${\cal O}_{X,{\Bbb C}}$-module,
     $\;{\cal E}$ has rank $r=\hat{r}\hat{d}$.

   \item[$\cdot$]
    \makebox[13ex][l]{\it $[\,$surrogate$\,]$}
    $X_{\varphi}= (c,f)(\hat{X})
     = (c\times f)(\hat{X})\subset X\times Y\,$ \\
     \makebox[13ex]{} and
    ${\cal A}_{\varphi}={\cal O}_{X_{\varphi},{\Bbb C}}$
     as a smooth manifold with singularities, which acts\\
     \makebox[13ex]{}
     tautologically on ${\cal E}$
     since $(c,f)_{\ast}\hat{\cal E}
                =_{{\cal O}_{X_{\varphi},{\Bbb C}}}{\cal E}$.

   \item[$\cdot$]
    \makebox[13ex][l]{\it $[\,$maps$\,]$}
    $\pi_{\varphi}=\pr_1|_{X_{\varphi}}:X_{\varphi}\rightarrow X\,$  and
    $\,f_{\varphi}=\pr_2|_{X_{\varphi}}:X_{\varphi}\rightarrow Y\,$.

   \item[$\cdot$]
    \makebox[46ex][l]{\it
     $[\,$minimally flat connection-with-singularity$\,]$}
    {From} the given generic covering/immersion property of $(c,f)$,
    there exists an open dense submanifold
     $V_{\varphi}\subset X_{\varphi}$ such that
     \begin{itemize}
      \item[-]
       $f_{\varphi}|_{V_{\varphi}}: V_{\varphi}\rightarrow Y$
       is a special Lagrangian immersion of phase factor $e^{i\theta}$;

      \item[-]
       $V:= \pi_{\varphi}(V_{\varphi})\subset X$
       (resp.\ $\hat{V}:=(c,f)^{-1}(V_{\varphi})\subset \hat{X}$)
        is an open dense submanifold
        with the property that the three maps,
          $c|_{\hat{V}}$, $(c,f)|_{\hat{V}}$, and
          $\pi_{\varphi}|_{V_{\varphi}}$,
        in the commutative diagram
        $$
         \xymatrix{ 
          **[r]\hat{V}\subset \hat{X}
           \ar[rd]^-{(c,f)|_{\hat{V}}} \ar@/_/[rddd]_-{c|_{\hat{V}}} \\
           & **[r]V_{\varphi} \subset X_{\varphi}
             \ar[dd]^-{\pi_{\varphi}|_{V_{\varphi}}} \\  \\
           & **[r]V\subset X
          }
        $$
        are all covering maps;

       \item[-]
        $V_{\varphi}=\pi_{\varphi}^{-1}(V)$.
     \end{itemize}
    In particular,
     $$
      (c,f)\;:\; (c,f)^{-1}(U^{\prime})\; \longrightarrow\; U^{\prime}
       \hspace{1em}\mbox{($\;$resp.
        $\;\pi_{\varphi}\;:\;
                \pi_{\varphi}^{-1}(U)\longrightarrow U\;,\;\;$
        $\;c\;:\; c^{-1}(U)\;\longrightarrow\; U\;$)}
     $$
     is a disjoint unions of diffeomorphisms
     for contractible open sets $U\subset V$
      (resp.\ $U^{\prime}\subset V_{\varphi}$, $U\subset V$).
    It follows from Lemma/Definition~2.1.9 that
     %
     \begin{itemize}
      \item[-]
       the flat connection $\hat{\nabla}$ on $\hat{\cal E}$
        descends to a flat connection-with-singularity $\nabla$
        on $_{{\cal O}_{X_{\varphi},{\Bbb C}}}{\cal E}$.
     \end{itemize}
    The holonomy of $\nabla$, when descends to $X$, lies $U(r)$
     by considering $c|_{\hat{V}}$ and Example~2.1.11.
  \end{itemize}
 In this way, we obtain all the data that is needed to describes
  a morphism $\varphi$
  from an Azumaya manifold with a fundamental module to $Y$,
   with a unitary minimally flat (singular) connection $\nabla$.
 In this note, we'll use such {\it simple A-branes} $(\varphi,\nabla)$
  -- in the sense that
     their image $\varphi(X^{A\!z})$ in the Calabi-Yau space $Y$
     has no nilpotent structure along it on an open dense subset --
  to illustrate some well-known D-brane behaviors in string theory.
}\end{example}

\bigskip

\subsection{Higgsing/un-Higgsing of A-branes via deformations of morphisms.}

The Higgsing/un-Higgsing behavior of the gauge symmetry on D-branes
  in string theory
 arises, in the current setting,
  from deformations of morphisms from Azumaya spaces
  with a fundamental module.
We have already seen this
  in [L-Y1: Sec.~4] (D(1)), [L-Y2: Sec.~2 and {\sc Figure}~2-1] (D(3)),
   and [L-Y4: Example~5.1.11] (D(5))
 (cf.\ [L-Y5: {\sc Figure}~2-1-1 and caption] (D(6)))
 for the case of supersymmetric D-branes of B-type (i.e.\ B-branes)
  in various contexts.
{For} A-branes, it follows from the work of McLean [McL]  that
 such a behavior natural occurs
 whenever the target special Lagrangian submanifold $L\subset Y$
  in a Calabi-Yau manifold admits a finite covering
  $\tilde{L}\rightarrow L$ with the first Betti number strictly increased:
  $b^1(\tilde{L})>b^1(L)$.
(Here, $\tilde{L}$ is allowed to be disconnected
       even when $L$ is connected.)
The work of Joyce [Joy3] allows one to extend such result to
 $L$ with isolated conical singularities as well.
The following example on a complex Calabi-Yau torus is motivated by
 the example of Denef in [De3: Sec.~6.1].
It illustrates a Higgsing (resp.\ un-Higgsing) behavior of A-branes
 that involves also an assembling (resp.\ disassembling) of
 a collection of ``small branes" into a ``large brane"
 (resp.\ a ``large brane" into a collection of ``small branes")
 in the process.
Some necessary background from the work of Joyce [Joy3] and notations
 to understand this example are summarized in Appendix~A.1.

\begin{example}
{\bf [Higgsing and un-Higgsing of A-branes under a Morse cobordism].}
{\rm
 We explain the construction in five steps.

 \medskip

 \noindent
 $(a)$ {\it A necklace of special Lagrangian submanifolds.}\hspace{1em}
 Let
  ${\Bbb C}^3$ be the complex $3$-space
   with coordinates $(z^1,z^2,z^3)$,
     the standard flat K\"{a}hler structure
      $\frac{i}{2} (dz^1\wedge d\bar{z}^1
           + dz^2\wedge d\bar{z}^2 + dz^3\wedge d\bar{z}^3)$,
     and the holomorphic $3$-form $dz^1\wedge dz^2\wedge dz^3$;
  $C_{\tau}:={\Bbb C}/({\Bbb Z}+{\Bbb Z}\tau)$
   be the complex $1$-torus of modulus $\tau$ with $\Imaginary\tau>0\,$;
   and
  $$
   Y\;=\; C_{\tau_1}\times C_{\tau_2}\times C_{\tau_3}\,,
    \hspace{2em}\mbox{with
      $\;\tau_1\tau_2\tau_3\,\in\, {\Bbb R}_{<0}\;$  and
      $\;(\tau_1-1)(\tau_2-1)(\tau_3-1)\,\in\, {\Bbb R}_{>0}\;$},
  $$
   be a product complex $3$-torus
    equipped with the complex structure $J$,
    the flat metric with the K\"{a}hler form $\omega$,
     and
    the holomorphic $3$-form $\Omega$,
   all from the descent as a quotient space of ${\Bbb C}^3$.
 We will denote $(Y,J,\omega,\Omega)$ also simply by $Y$.
 The quotient $C_{\tau_i}={\Bbb C}/({\Bbb Z}+{\Bbb Z}\tau_i)$
  specifies an isomorphism
   $H_1(C_{\tau_i};{\Bbb Z}) \simeq {\Bbb Z}\oplus {\Bbb Z}\tau$
   as ${\Bbb Z}$-modules.
 Let $(\alpha_i,\beta_i)$ be the basis of $H_1(C_{\tau_i};{\Bbb Z})$
  that corresponds $(1,\tau_i)$; cf.~[De3: Sec.~6.1, Figure~9].
 One has from the K\"{u}nneth formula that
  $$
   H_3(Y;{\Bbb Z})\;
   =\; \oplus_{j_1+j_2+j_3=3}\,
         H_{j_1}(C_{\tau_1};{\Bbb Z})
          \times H_{j_2}(C_{\tau_2};{\Bbb Z})
          \times H_{j_3}(C_{\tau_3};{\Bbb Z})\,.
  $$

 Consider the following
  three embedded special Lagrangian submanifolds in $Y$
   from special Lagrangian $3$-planes in ${\Bbb C}^3$,
  with the orientation specified by the restriction of $\Real\Omega\,$:
  $$
   \mbox{
    \begin{tabular}{ccccc}
     sL in $Y$
      && lifting in ${\Bbb C}^3$
      && $[\,\cdot\,]\in H_3(Y;{\Bbb Z})$  \\ \hline
     \rule{0ex}{3ex}
     $L_1$
      && {\scriptsize $\Real {\Bbb C}^3\;
          =\;{\Bbb R}(1,0,0)+{\Bbb R}(0,1,0)+{\Bbb R}(0,0,1)$}
      && {\scriptsize $\alpha_1\times\alpha_2\times\alpha_3$} \\
     \rule{0ex}{3ex}
     $L_2$
      && {\scriptsize
         $-{\Bbb R}(\tau_1,0,0)-{\Bbb R}(0,\tau_2,0)-{\Bbb R}(0,0,\tau_3)$}
      && {\scriptsize $-\,\beta_1\times\beta_2\times\beta_3\;\;\;\,$} \\
     \rule{0ex}{3ex}
     $L_3$
      && {\scriptsize
          $(\frac{1}{2},\frac{1}{2},\frac{1}{2})\,
           +\, {\Bbb R}(\tau_1-1,0,0)+{\Bbb R}(0,\tau_2-1,0)
              +{\Bbb R}(0,0,\tau_3-1)$}
      && {\scriptsize
          $(\beta_1-\alpha_1)\times(\beta_2-\alpha_2)
           \times(\beta_3-\alpha_3)$}
    \end{tabular}
   }
  $$
 It follows from the constraints
  $\tau_1\tau_2\tau_3 \in {\Bbb R}_{<0}\,$,
  $(\tau_1-1)(\tau_2-1)(\tau_3-1) \in {\Bbb R}_{>0}\,$
  and, hence,
  $\left( \frac{\tau_1-1}{\tau_1} \right)
   \left( \frac{\tau_2-1}{\tau_2} \right)
   \left( \frac{\tau_3-1}{\tau_3} \right) \in {\Bbb R}_{<0}\;$
  (note $\;\Imaginary\tau>0$ implies
             $\Imaginary(\tau-1)>0$ and
             $\Imaginary(\frac{\tau-1}{\tau})>0\;$) that:
  \begin{itemize}
   \item[$\cdot$]
    the sum of the characteristic angles from $L_1$ to $L_2\;$
     (resp. from $L_1$ to $L_3\,$, $\;$from $L_2$ to $L_3\,$)
     is $\pi$ (resp. $2\pi\,$, $\;\pi\,$).
  \end{itemize}
 $L_i$ and $L_j$, $i\ne j$, intersect transversely at exactly one point.
 The oriented intersection numbers of $L_1$, $L_2$, and $L_3$
  at their intersection point are given by
  $$
   L_1\cdot L_2\;=\;-L_2\cdot L_1=+1\,, \hspace{1em}
   L_2\cdot L_3\;=\;-L_3\cdot L_2=+1\,, \hspace{1em}
   L_3\cdot L_1\;=\;-L_1\cdot L_3=+1\,.
  $$
  All other intersection numbers $=0$.
 Note that this is consistent with Remark~3.2.5 of Sec.~3.2.

 \medskip

 \noindent
 $(b)$ {\it Smoothing of $L_1\cup L_2\cup L_3$ in $Y$.}\hspace{1em}
 Let $y_{ij}=L_i\cap L_j\in Y$ for $(ij)=(12)$, $(23)$, $(31)$.
 Consider the (tautological) special Lagrangian immersion
  $$
   f\;:\; L\,:=\, L_1\amalg L_2\amalg L_3\;\longrightarrow\;
      L_1\cup L_2\cup L_3\subset Y
  $$
   with isolated transverse intersections and
  let
   $x_{ij}^+:=f^{-1}(y_{ij})\cap L_i$ and
   $x_{ij}^-:=f^{-1}(y_{ij})\cap L_j$.
 Then, the sum of the characteristic angles from
  $f_{\ast}T_{x_{ij}^+}L$ to $f_{\ast}T_{x_{ij}^-}$ at $y_{ij}$
  is $\pi$.
 Recall [Joy3: V.~Sec.~9.2, Theorem~9.7]
  (cf.~Theorem~A.1.4 in Appendix~A.1)  and
 consider the following linear system with constraints:
  $$
   A_{12}\,-\, A_{31}\;=\;0\,,\;\;\;
   A_{23}\,-\, A_{12}\;=\;0\,,\;\;\;
   A_{31}\,-\, A_{23}\;=\;0\,,
   \hspace{2em}\mbox{with}\hspace{2em}
   A_{12}\,,\; A_{23}\,,\; A_{31}\; >\; 0\,.
  $$
 This has solutions:
  $$
   A_{12}\;=\; A_{23}\;=\; A_{31}\; >\; 0\,.
  $$
 It follows that the special Lagrangian submanifold
  $L_1\cup L_2\cup L_3\subset Y$ with conical singularities
  $\{y_{12}\,,\, y_{23}\,,\, y_{31}\}$ is smoothable:
  \begin{itemize}
   \item[$\cdot$]
    There exists a smooth family of special Lagrangian embeddings
     $$
      f^{(t)}\: :\; L^{(t)}\;\longrightarrow\; Y\,,
       \hspace{2em}\mbox{$t\in (0,\varepsilon)\;\;$
                   for some $\;\;\varepsilon>0$}
     $$
     such that
      $L^{(t)}\simeq$
       the (self-)connected sum $N$ of $L$
       at the pairs of points
        $(x_{12}^+\,,\, x_{12}^-)$, $(x_{23}^+\,,\,x_{23}^-)$, and
        $(x_{31}^+\,,\, x_{31}^-)$
      and that $f^{(t)}\rightarrow f$
       in the sense of currents as $t\rightarrow 0$.
  \end{itemize}
 Since
  each $L_i$ is diffeomorphic to the real $3$-torus ${\Bbb T}^3$,
 one has
  $$
   b^1(L_1\amalg L_2\amalg L_3)\; =\; 9
   \hspace{2em}\mbox{and}\hspace{2em}
   b^1(L_1\cup L_2\cup L_3)\;=\; b^1(N)=10\,.
  $$

 \medskip

 \noindent
 $(c)$ {\it A family of A-branes on $Y$ that wrap $f^{(t)}(L^{(t)})$,
            $t\in [0,\varepsilon)$.}\hspace{1em}
 Let
  $\pi_{(-\varepsilon,\varepsilon)}:
    X_{(-\varepsilon, \varepsilon)}
    \rightarrow (-\varepsilon,\varepsilon)$
  be a Morse family of $3$-manifolds with singularities,
  with
  $$
   X_t\; :=\; \pi_{(-\varepsilon,\varepsilon)}^{-1}(t)\;
   \simeq\; \left\{
   \begin{array}{lcl}
     L= L_1\amalg L_2\amalg L_3
      && \mbox{for $\;\;t\in (-\varepsilon,0)\,$,}\\[.6ex]
     L_1\cup L_2\cup L_3
      && \mbox{for $\;\;t=0\,$,}\\[.6ex]
     L^{(t)} && \mbox{for $\;\;t\in (0,\varepsilon)\,$.}
   \end{array} \right.
  $$
 Then, the family $f^{(t)}$, $t\in (0,\varepsilon)$, and $f$ together
  define a continuous map
  $$
   \raisebox{4ex}{\xymatrix{
    X_{(-\varepsilon,\varepsilon)}
      \ar[rd]_-{\pi_{(-\varepsilon,\varepsilon)}}
      \ar[rr]^-{F}
     &&  (-\varepsilon,\varepsilon)\times Y
         \ar[ld]^-{pr_1}\\
    & (-\varepsilon,\varepsilon)
   }}\;,
   \hspace{2em}
   F(t,\,\cdot)\;
   =\; \left\{
   \begin{array}{lcl}
     f(\,\cdot\,)
      && \mbox{for $\;\;t\in (-\varepsilon,0)\,$,}\\[.6ex]
     \Id_{L_1\cup L_2\cup L_3}(\,\cdot\,)
      && \mbox{for $\;\;t=0\,$,}\\[.6ex]
     f^{(t)}(\,\cdot\,) && \mbox{for $\;\;t\in (0,\varepsilon)\,$,}
   \end{array} \right.
  $$
  over $(-\varepsilon,\varepsilon)$ that is smooth on
  $\pi_{(-\varepsilon,\varepsilon)}^{-1}
                   ((-\varepsilon,0)\cup (0,\varepsilon))$.
 Let $c$ in
  $$
   \xymatrix{
    (\hat{\cal E}_{(-\varepsilon,\varepsilon)},
     \hat{\nabla}^{(-\varepsilon,\varepsilon)})\ar@{.>}[rd] \\
    & \hat{X}_{(-\varepsilon,\varepsilon)}
       \ar[rr]^-c \ar[rd]_-{\hat{\pi}_{(-\varepsilon,\varepsilon)}}
      && X_{(-\varepsilon,\varepsilon)}
         \ar[ld]^-{\pi_{(-\varepsilon,\varepsilon)}} \\
    && (-\varepsilon,\varepsilon)
   }
  $$
  be a covering map over $(-\varepsilon,\varepsilon)$
   of (finite) degree $r>1$  and
 $(\hat{\cal E}_{(-\varepsilon,\varepsilon)},
   \hat{\nabla}^{(-\varepsilon,\varepsilon)})$
  be a complex line bundle on $\hat{X}_{(-\varepsilon,\varepsilon)}$,
   with a $U(1)$ flat connection.
 Then one has the following diagram of maps
 $$
  \xymatrix{
   (\hat{\cal E}_{(-\varepsilon,\varepsilon)},
    \hat{\nabla}^{(-\varepsilon,\varepsilon)}) \ar@{.>}[d]
    & & & & & \\
   \hat{X}_{(-\varepsilon,\varepsilon)}
     \ar[rd]^{(c, F\circ c)}
     \ar@/^2ex/[rrrrrd]
      ^-{\hat{\pi}_{(-\varepsilon,\varepsilon)}\times(F\circ c)}
     \ar@/_/[rddd]_-{c}     & & \\
    & X_{(-\varepsilon,\varepsilon)}\times Y
        \ar[rrrr]_-{\pi_{(-\varepsilon,\varepsilon)}\times pr_2}
        \ar[dd]^-{pr_1}
      & & & & (-\varepsilon,\varepsilon)\times Y\;, \\  \\
    & X_{(-\varepsilon,\varepsilon)}
  }
 $$
 in which
  $\Image(c,F\circ c)=\Graph(F)$ with multiplicity $r$.
 Let
  $$
   {\cal E}_{(-\varepsilon,\varepsilon)}\;
     :=\; c_{\ast}\hat{\cal E}_{(-\varepsilon,\varepsilon)}\;
      =\; {\pr_1}_{\ast}\circ(c,F\circ c)_{\ast}
              \hat{\cal E}_{(-\varepsilon,\varepsilon)}\,.
  $$
 The diagram defines a ${\Bbb C}$-algebra homomorphism
  $\varphi_{(-\varepsilon,\varepsilon)}^{\sharp}:
   {\cal O}_{(-\varepsilon,\varepsilon)\times Y,{\Bbb C}} \rightarrow
   \Endsheaf_{{\cal O}_{X_{(-\varepsilon,\varepsilon)},{\Bbb C}}}
    ({\cal E}_{(-\varepsilon,\varepsilon)})$
  and, hence, a morphism
  $\,\varphi_{(-\varepsilon,\varepsilon)}\,:\,
     (X_{(-\varepsilon,\varepsilon)}^{A\!z},
         {\cal E}_{(-\varepsilon,\varepsilon)})\,
    \rightarrow\,(-\varepsilon,\varepsilon)\times Y\,$
  over $(-\varepsilon,\varepsilon)\,$,
  with a unitary minimally flat connection-with-singularity.
 We will think of the latter also interchangeably
  as a family of morphisms
   $\{\varphi_t:(X_t^{A\!z},{\cal E}_t)\rightarrow Y\,|\,
       t\in (-\varepsilon,\varepsilon)\}\,$,
   with a unitary minimally flat connection-with-singularity.

 \bigskip

 \noindent
 {\bf Lemma~2.2.1.(d) [first Betti number].} {\it
  Let
   $\hat{X}_t:= \hat{\pi}_{(-\varepsilon,\varepsilon)}^{-1}(t)$
    for $t\in (-\varepsilon,\varepsilon)$  and
   $\Gamma$ be the dual graph of $\hat{X}_0$
    with the number of vertices $|\Gamma_{(0)}|$  and
         the number of edges $|\Gamma_{(1)}|=3r$.
  Then,
  $b^1(\hat{X}_t)=3\,|\Gamma_{(0)}|\,$, for $t\in (-\varepsilon,0)$, and
  $\;b^1(\hat{X}_t)\,=\, 3r+2|\Gamma_{(0)}|+1\,$,
   for $t\in [0,\varepsilon)$.
  In particular,
  $b^1(\hat{X}_t)\ge 3r+7> b^1(L_1\cup L_2\cup L_3) =b^1(N) (=10)$
   for $t\in [0,\varepsilon)$
   since $r>1$ and $3\le |\Gamma_{(0)}|\le 3r$.
 } 

 \bigskip

 \noindent
 Here, recall that
  the dual graph $\Gamma$ of $\hat{X}_0$ has
   one vertex $v_i$ for each manifold-component $M_i\simeq {\Bbb T}^3$
    of $\hat{X}_0$ and
   one edge $e_{ij}$ connecting $v_i$ and $v_j$
    for each intersection point in $M_i\cap M_j$;
  $\Gamma_{(0)}$ is the set of vertices of $\Gamma$ and
   $\Gamma_{(1)}$ is the set of edges of $\Gamma$.

 \bigskip

 \noindent
 {\it Proof.}
  That $b^1(\hat{X}_t)=3|\Gamma_{(0)}|$ for $t\in (-\varepsilon,0)$
   follows from the fact that the only finite-order covering space
   of ${\Bbb T}^3$ is homeomorphic to ${\Bbb T}^3$.
  That $b^1(\hat{X}_t)=b^1(\hat{X}_0)$ for $t\in (0,\varepsilon)$
   follows from the fact that $\hat{X}_0$ is topologically obtained
    from $\hat{X}_t$, $t\in (0,\varepsilon)$, by pinching
    a disjoint union of two-sided embedded $S^2$'s and, hence,
    the fundamental groups $\pi_1(\hat{X}_t)\simeq \pi_1(\hat{X}_0)$
     for $t\in (0,\varepsilon)$.
  It remains to compute $b^1(\hat{X}_0)$.
  Which follows from the following basic facts:
   \begin{itemize}
    \item[$\cdot$]
     the exact sequence of groups
     $$
      1\;\longrightarrow\; \pi_1(\vee_{v_i\in\Gamma_{(0)}}M_i)\;
         \longrightarrow\; \pi_1(\hat{X}_0)\;
         \longrightarrow\; \pi_1(\Gamma)\; \longrightarrow\; 1\,,
     $$
     where $\vee_{v_i\in \Gamma_{(0)}}M_i$ is the bouquet of
      $\{M_i:v_i\in \Gamma_{(0)}\}$
      following a(ny) spanning tree of $\Gamma$,

    \item[$\cdot$]
     $\pi_1(\vee_{v_i\in\Gamma_{(0)}}M_i)$
      is isomorphic to the free product of $|\Gamma_{(0)}|$-many
      copies of\\ $(\pi_1(M_i)\simeq)\pi_1({\Bbb T}^3)\simeq {\Bbb Z}^3$,

    \item[$\cdot$]
     $\pi_1(\Gamma)$ is isomorphic to the free group on
      $1-\chi(\Gamma)=1-|\Gamma_{(0)}|+3r$ generators,  and

    \item[$\cdot$]
     $H_1(\,\bullet\,; {\Bbb Z})$
     is the abelianization of $\pi_1(\,\bullet\,)$.
   \end{itemize}
  This concludes the proof.

  \noindent\hspace{13.7cm}$\square$

 \medskip

 \noindent
 $(e)$ {\it The dimension of deformation spaces and Higgsing/un-Higgsing.}
 \hspace{1em}
 Let
  $\hat{f}_t:= F(\cdot,t)\circ c(\cdot,t): \hat{X}_t\rightarrow Y$
  for $t\in (-\varepsilon,\varepsilon)$.
 {For} $t\in (-\varepsilon,0)$, $\hat{X}_t$ is smooth.
 The deformation space ${\cal M}_t^{\scriptsizesLag}$
  of special Lagrangian immersions from $\hat{X}_t$ to $Y$ is thus
  a manifold of dimension $b^1(\hat{X}_t)=3|\Gamma_{(0)}|$
  around $[\hat{f}_t]$.
 {For} $t=0$, $\hat{X}_t$ is a union of $|\Gamma_{(0)}|$-many
  ${\Bbb T}^3$-components.
 The deformation space ${\cal M}_0^{\scriptsizesLag}$
  of special Lagrangian immersions from $\hat{X}_0$ to $Y$ contains
  thus a manifold of dimension $3|\Gamma_{(0)}|$ around $[\hat{f}_0]$.
 {For} $t\in (0,\varepsilon)$, $\hat{X}_t$ is smooth again.
 The deformation space ${\cal M}_t^{\scriptsizesLag}$
  of special Lagrangian immersions from $\hat{X}_t$ to $Y$ is thus
  a manifold of dimension $b^1(\hat{X}_t)=3r+2|\Gamma_{(0)}|+1$
  around $[\hat{f}_t]$.
 As $b^1(\hat{X}_t)$ is locally constant on
  $(-\varepsilon,0)\cup (0,\varepsilon)$,
 the collection
  $\{{\cal M}_t^{\scriptsizesLag}:t\in (-\varepsilon,\varepsilon)\}$
  forms a topological space ${\cal M}_{(-\varepsilon,\varepsilon)}$
  over $(-\varepsilon,\varepsilon)$,
  with the topology from the topology of
   $\hat{X}_{(-\varepsilon,\varepsilon)}$ and
   the topology on the space of maps in question
   in the sense of currents.
 By construction,
  ${\cal M}_{(-\varepsilon,\varepsilon)}/(-\varepsilon,\varepsilon)$
  contains $\{\hat{f}_t:t\in (-\varepsilon,\varepsilon)\}$
  as a (continuous) section $s$ that is smooth on
  $(-\varepsilon,0)\cup (0,\varepsilon)$.
 Furthermore, there is a submanifold in
  ${\cal M}_{(-\varepsilon,\varepsilon)}$
  of relative dimension $>10= b^1(L_1\cup L_2\cup L_3) = b^1(N)$
  over $(0,\varepsilon)$
  that contains $\{\hat{f}_t:t\in (0,\varepsilon)\}$.
 It follows that one can perturb $s$ to another section $s^{\prime}$
  -- representing a new family
     $\{\hat{f}^{\prime}_t=:\hat{X}_t\rightarrow Y\,|\,
                                    t\in (-\varepsilon,\varepsilon)\}$
     of special Lagrangian maps --
  that remains continuous, is identical to $s$ on $(-\varepsilon,0]$,
   and is smooth on $(-\varepsilon)\cup(0,\varepsilon)$ such that
   the following condition holds:
  \begin{itemize}
   \item[$\cdot$]
    For all $t\in (0,\varepsilon)$,
    in the local parameterization of ${\cal M}_t^{\scriptsizesLag}$
     around $[\hat{f}_t]$ by a (finite-dimensional) submanifold
     in the (infinite-dimensional) Banach manifold of graphs
     in $T^{\ast}\hat{X}_t\simeq {\cal N}_{\hat{f}_t}$
     of closed $1$-forms on $\hat{X}_t$ under a Sobolev norm,
    the closed $1$-form $\hat{\xi}^{\prime}_t$ on $\hat{X}_t$
     associated to $\hat{f}^{\prime}_t$
     is not the pull-back of a closed $1$-form on $X_t$
     under the covering map $c$.
    Here, ${\cal N}_{\hat{f}_t}\subset \hat{f}_t^{\ast}T_{\ast}Y$
     is the normal bundle of $\hat{X}_t$ in $Y$ along $\hat{f}_t$.
  \end{itemize}
 In particular, for example,
  $\hat{f}^{\prime}_t:\hat{X}_t\rightarrow Y$ does not
  factor through a special Lagrangian map from $X_t$ to $Y$
  for $t\in (0,\varepsilon)$.
 As a result, for $t\in (0,\varepsilon)$,
  the overlapped sheets of $c:\hat{X}_t\rightarrow X_t$ under $\hat{f}_t$
   are now separated by $\hat{f}^{\prime}_t$  and
  the image Chan-Paton module ${\varphi^{\prime}_t}_{\ast}{\cal E}_t$
   of the associated new family of morphisms
   $\{\varphi^{\prime}_t:(X_t^{A\!z},{\cal E}_t)\rightarrow Y\,|\,
        t\in (-\varepsilon,\varepsilon)\}$
   from Azumaya spaces with a fundamental module to $Y$
  exhibits now a Higgsing (resp.\ un-Higgsing) phenomenon
   as $t$ moves away from $0$ for $t\in [0, \varepsilon)$
   (resp.\ as $t$ moves to $0$ for $t\in [0, \varepsilon)$).
 Note that the data of the unitary minimally flat
  connection-with-singularity that accompanies the deformed family
  of morphisms $\varphi^{\prime}_t$
  still comes from the $U(1)$ flat connection
  $\hat{\nabla}^{(-\varepsilon,\varepsilon)}$
   on $\hat{\cal E}_{(-\varepsilon,\varepsilon)}$
        over $\hat{X}_{(-\varepsilon,\varepsilon)}$.
 This concludes the example.

 \noindent\hspace{15.7cm}$\square$
 %
 %
 %
}\end{example}

\bigskip

In this example, we fix the Calabi-Yau $3$-fold $Y$ in question.
In Sec.~3.2, we will see that
 such Higgsing/un-Higgsing behavior of D-branes
  -- as morphisms from Azumaya spaces with a fundamental module --
  mixed with assembling/disassembling of branes
  can also occur
 when the D-brane is driven to deform alongside with the deformation
  of the complex structures on $Y$ along an attractor flow.

\bigskip

\subsection{Large- vs.\ small-brane wrapping via deformations
            of morphisms.}

The long vs.\ short string wrapping behavior of
  matrix-strings in string theory
  (e.g., [D-V-V], [Joh: Sec.~16.3.3], [M-S])
 generalizes to
  a {\it large- vs.\ small-brane wrapping} behavior of D-branes.
Such phenomenon can be produced in our context
 via morphisms from an Azymaya manifold/scheme
 (with a fundamental module) and their deformations.
We explain
 first two basic local differential topological operations
  for such a purpose in the case of A-branes and
 then how transitions between large-brane wrapping and
  small-brane wrapping can be realized via deformations of morphisms.
A simplest example that distinguishes this nature of A-branes
 and a related question are given in the end.

\bigskip

\begin{flushleft}
{\bf Gluing of manifolds with singularities along codimension-1 loci.}
\end{flushleft}
\begin{definition}
{\bf [irreducible component].} {\rm
 Let $M$ be a manifold with singularities with
  the smooth locus $M_{\smoothscriptsize}$ and
  the singular locus $M_{\singularscriptsize}:=M-M_{\smoothscriptsize}$.
 Then $M_{\smoothscriptsize}$ is a disjoint union $\coprod_iU_i$
  of smooth manifolds $U_i$.
 The closure $\overline{U_i}\subset M$ of each $U_i$ in $M$
  is called an {\it irreducible component} of $M$.
}\end{definition}

Given a (possibly disconnected) oriented manifold $M$ with singularities,
let
 $Z=\coprod_{i=1}^lZ_i \subset M_{\smoothscriptsize}$ be a (disconnected)
  codimension-$1$ compact smooth oriented embedded
  submanifold-with-boundary in $M$ and
 $U=\coprod_{i=1}^lU_i\subset M_{\smoothscriptsize}$
  be a manifold-neighborhood of $Z$ in $M$
  with fixed diffeomorphisms $(U_i, Z_i)\simeq (U_0,Z_0)$
   for $i=1,\,\ldots,\,l\,$  and
   for some smooth submanifold-with-(smooth-)boundary $Z_0$
        in a smooth manifold $U_0$.
Let
 $DZ:=Z^+\cup_{\partial Z} Z^-$ be the doubling of $Z$
  along $\partial Z$,
  (here, $Z^+=Z$ and $Z^-$ is $Z$ with the orientation reversed),
  and
 $\widehat{M}$ be the oriented manifold-with-boundary with singularities,
  obtained from the tautological compactification of $M-Z$ by $DZ$.
By construction,
 $DZ$ constitutes now some boundary components of $\widehat{M}$,
  with the induced orientation,  and
 there is an orientation-reversing involution
  $DZ\rightarrow DZ$
  that leaves $\partial Z\subset DZ$ fixed  and
       descends to the identity map on $Z_0$.
Let $g:Z^+\rightarrow Z^-$ be an orientation-reversing diffeomorphism
 that descends to the identity map on $Z_0$.
(Caution that, in general, this is not a diffeomorphism on $DZ$.)
Let $M^{\prime}$ be the quotient manifold with singularities
 by the equivalence relation generated by
  $z\sim g(z)$ for $z\in DZ\subset \widehat{M}$.
Then it follows from the construction that:
 \begin{itemize}
  \item[(1)]
   The $DZ$-boundary of $\widehat{M}$ is closed up to
    a compact embedded submanifold-with-singularity
     $Z^{\prime}\subset M^{\prime}$
    that consists of only manifold-points in $M^{\prime}$.

  \item[(2)]
   $M^{\prime}$ has a natural smooth structure along $Z^{\prime}$
    in terms of which $Z^{\prime}\subset M^{\prime}_{\smoothscriptsize}$
     and the tautological inclusion $M-Z\hookrightarrow M^{\prime}$
      is a smooth embedding.

  \item[(3)]
   Let $U^{\prime}$ be the open submanifold in $M^{\prime}$
    that arises from $U\subset M$.
   Then the construction defines a branched covering map
    $\pi_0:U^{\prime}\rightarrow U_0$ of degree $l$
    with branch locus given by the codimension-$2$ submanifold
     $\partial Z_0$ in $U_0$.
   Let $\sigma\in\Sym_l$ be the permutation of elements of
    $\{1,\,\cdots\,,l\}$ defined by $g(Z^+_i)=Z^-_{\sigma(i)}$.
   Up to an inner automorphism of $\Sym_l$ on $\Sym_l$,
    the monodrony of $\pi_0$ is given by $\sigma$.
   The number of components in the cyclic decomposition of $\sigma$
    gives then the number of the connected components of $U^{\prime}$.

  \item[(4)]
   Let
    $M=\cup_{j=1}^kM_j$ be the decomposition of $M$
     by its irreducible components.
   Then $g$ induces an equivalence relation on $\{1,\,\cdots\,,k\}$
    by generated by $j\sim j^{\prime}$
     if $g(Z^+_{i})=Z^-_{i^{\prime}}$
      for some $Z^+_{i}\subset M_j$
          with $Z^-_{i^{\prime}}\subset M_{j^{\prime}}$.
   The number of irreducible components $M^{\prime}$
    is then the number of equivalence classes in $\{1,\,\cdots\,,k\}$.
   In general, this is smaller than $k$.
 \end{itemize}
All these are direct generalizations from the case of Riemann surfaces,
 possibly bordered or with nodes.

One can put the above construction also
 into a locally generically constant family of
 manifolds-with-singularities over an interval
 $(-\varepsilon, \varepsilon)$ for some $\varepsilon>0$, as follows:
 (continuing the notation from the previous discussion)
 \begin{itemize}
  \item[$\cdot$]
   Let $\underline{M}:=$ the quotient manifold-with-singularities
     $\,M/\!\sim\,$,
    where $\sim$ is an equivalence relation on $M$,
     generated by $m_1\sim m_2$
     if $m_1,\,m_2\in Z$ and $g(m_1^+)=m_2^-$ or $m_1^-=g(m_2^+)$.
   Here, $m_i^{\pm}$ is the lifting of $m_i$ to $Z^{\pm}$ in $DZ$.
   Let $\underline{Z}:=Z/\!\sim$ with the tautological embedding
    $\underline{Z}\hookrightarrow \underline{M}$.
   Then, there are tautological surjections
    $$
     h_-\;:\; M\;\longrightarrow\; \underline{M}
       \hspace{2em}\mbox{and}\hspace{2em}
     h_+\;:\; M^{\prime}\; \longrightarrow\; \underline{M}
    $$
    that restrict to the built-in identity maps
     $M-Z=\underline{M}-\underline{Z}=M^{\prime}-Z^{\prime}$.

  \item[$\cdot$]
   Define the sought-for family over $(-\varepsilon,\varepsilon)$ by
    $$
     {\cal M}_{(-\varepsilon,\varepsilon)}
      := (M\times (-\varepsilon,0])_{h_-\times\{0\}}\cup\underline{M}
       \cup_{h_+\times\{0\}}(M^{\prime}\times [0,\varepsilon))\,,
    $$
    where
     $h_-\times\{0\}=h_-:M\times\{0\}\rightarrow \underline{M}$ and
     $h_+\times\{0\}=h_+:M^{\prime}\times\{0\}\rightarrow \underline{M}$.
   The built-in projection
    $\pi_{(-\varepsilon,\varepsilon)}:{\cal M}_{(-\varepsilon,\varepsilon)}
               \rightarrow (-\varepsilon,\varepsilon)$
    has ${\cal M}_t:=\pi_{(-\varepsilon,\varepsilon)}(t)
         =M\times\{t\}$ for $t\in (-\varepsilon,0)$;
        $\underline{M}$ for $t=0$; and
        $M^{\prime}\times\{t\}$ for $t\in(0,\varepsilon)$.
 \end{itemize}
Note that, by construction, there is also a built-in map
 ${\cal M}_{(-\varepsilon,\varepsilon)}
         \rightarrow \underline{M}={\cal M}_0$.

\bigskip

\begin{flushleft}
{\bf Creation of codimension-1 incidence loci via isotopies.}
\end{flushleft}
Recall first the following theorem,
 stated with an adaptation to our situation:

\begin{theorem}
{\bf [isotopy of disk].} {\rm [Hir: Sec.~8.3, Theorem~3.1].}
 Let $f_1$ and $f_2:D^s\rightarrow {\Bbb R}^n$, $s\le n$,
  be smooth embeddings.
 When $s=n$, assume also that
  $f_1$ and $f_2$ are both orientation-preserving.
 Note that, as $D^s$ is compact in our notation,
  $f_1(D^s)$ and $f_2(D^s)$ are contained
  in a compact subset of ${\Bbb R}^n$.
 Then, $f_1$ and $f_2$ are isotopic.
 Furthermore, an isotopy between $f_1$ and $f_2$ can be realized
  by a diffeotopy of ${\Bbb R}^n$ with compact support.
\end{theorem}

The same proof of Theorem~2.3.2 gives indeed another form
 of Theorem~2.3.2:

\begin{theorem}
{\bf [creation of codimension-0 incidence locus via confined isotopy].}
 Given two orientation-preserving smooth embeddings
  $f_1$ and $f_2:D^n\rightarrow {\Bbb R}^n$ with $f_1(0)=f_2(0)=0$,
 there exists an isotopy $f_2^{(t)}$, $t\in[0,1]$ , of $f_2=:f_2^{(0)}$
 such that
 \begin{itemize}
  \item[$\cdot$]
   $f_2^{(t)}:D^n\rightarrow {\Bbb R}^n$ are (orientation-preserving)
    smooth embeddings with $f_2^{(t)}(0)=0$ and $f_2^{(t)}=f_2$
     on a neighborhood of $\partial D^n\subset D^n$,
    for all $t\in [0,1]$,

  \item[$\cdot$]
   there exists a small neighborhood $U$ of $0\in D^n$,
     with the closure $\overline{U}$ contained in the interior of $D^n$,
   such that $f_2^{(1)}|_U=f_1|_U$.
 \end{itemize}
 Furthermore, such an isotopy of $f_2$ can be realized
  by a diffeotopy of ${\Bbb R}^n$ with support contained in $f_2(D^n)$.
\end{theorem}

{For} a finite collection of orientation-preserving smooth embeddings
  $f_1, \,\cdots\,,f_l:D^n\rightarrow {\Bbb R}^n$ with $f_i(0)=0$,
one can keep $f_1$ fixed and perform the above isotopy to
  $f_2,\,\cdots\,,f_l$ one by one
 so that
 \begin{itemize}
  \item[$\cdot$]
   $f_i^{(t)}:D^n\rightarrow {\Bbb R}^n$
   are (orientation-preserving) smooth embeddings
     with $f_i^{(t)}(0)=0$ and
          $f_i^{(t)}=f_i$ on a neighborhood of $\partial D^n\subset D^n$,
   for $i=1,\,\ldots\,,l$ and $t\in [0,1]$,

  \item[$\cdot$]
   there exists a small neighborhood $U$ of $0\in D^n$,
     with $\overline{U}$ contained in the interior of $D^n$,
   such that $f_1^{(1)}|_U=\,\cdots\,=f_l^{(1)}|_U$.
 \end{itemize}
It follows then:

\begin{theorem}
{\bf [creation of codimension-1 incidence locus via confined isotopy].}
 Given a finite collection of orientation-preserving smooth embeddings
  $f_1, \,\cdots\,,f_l:D^n\rightarrow {\Bbb R}^n$
  with $f_i(0)=0$ for all $i$,
 there exists a neighborhood $U$ of $0\in D^n$,
  with $\overline{U}$ contained in the interior of $D^n$,
  for which the following holds:
  {For} any codimension-$1$ compact embedded submanifold-with-boundary
   $Z\subset U$,
  there exist isotopies
    $f_2^{(t)},\,\cdots\,,f_l^{(t)}$, $t\in [0,1]$, of
    $f_1=:f_2^{(0)},\,\cdots\,,f_l=:f_l^{(0)}$ respectively
   such that
   \begin{itemize}
    \item[$\cdot$]
     $f_i^{(t)}:D^n\rightarrow {\Bbb R}^n$
     are (orientation-preserving) smooth embeddings
     with $f_i^{(t)}(0)=0$ and
          $f_i^{(t)}=f_i$ on a neighborhood of $\partial D^n\subset D^n$,
     for $i=1,\,\ldots\,,l$ and $t\in [0,1]$,

    \item[$\cdot$]
     $f_1^{(1)}|_Z=\,\cdots\,=f_l^{(1)}|_Z$.
   \end{itemize}
\end{theorem}

Finally, the following lemma says that the condition $f_i(0)=0$
 can be achieved under a minimally required assumption:
(Which is indeed a special case of Theorem~2.3.2, with $s=0$.)

\begin{lemma}
{\bf [adjustment of center via confined isotopy].}
 Let $f:D^n\rightarrow {\Bbb R}^n$ be a (smooth) embedding
  such that $f(D^n)$ contains $0\in {\Bbb R}^n$ in its interior.
 Then there exists an isotopy $f^{(t)}$, $t\in [0,1]$, of $f=:f^{(0)}$
  that is identical to $f$, for all $t$, on a neighborhood of
  $\partial D^n$ in $D^n$
  and has the property that $f^{(1)}(0)=0$.
\end{lemma}

\begin{proof}
 Under the assumption, there exists an embedded smooth path $\gamma$
  that connects $0\in {\Bbb R}^n$ and $f(0)$ and
   is contained in the interior of $f(D^n)$.
 Let $U$ be an (arbitrarily small) neighborhood of $\gamma$
  in ${\Bbb R}^n$ with $\overline{U}\subset$ the interior of $f(D^n)$.
 The flow on ${\Bbb R}^n$ given by a smooth vector field
   that is supported in $U$,
           parallel to the tangent direction of $\gamma$,
           and is nowhere zero along $\gamma$
  can be adjusted to provide such an isotopy of $f_0$
  under the composition of $f_0$ with such a flow on ${\Bbb R}^n$.
 Such a vector field is easily constructed using a partition of unity.

\end{proof}

%
%
%
%
%
%
%

\bigskip

\begin{flushleft}
{\bf Transitions between large-brane wrapping and small-brane wrapping\\
     via deformations of morphisms.}
\end{flushleft}
Let
 $Y$ be a Calabi-Yau $n$-fold,
 ${\cal E}$ a complex vector bundle of rank $r$ on $X$,
   and
 $$
  \left(\,
   \varphi\::\:(X^{A\!z}, {\cal E})\: \longrightarrow\: Y\;,\;
   \nabla
  \,\right)
 $$
  be a special Lagrangian morphism
   with a unitary minimally flat connection-with-singularity,
  given by
  $\;\varphi^{\sharp}:{\cal O}_{Y,{\Bbb C}}^{\,\infty}\, \longrightarrow\,
    {\cal O}_X^{A\!z}:=\Endsheaf_{{\cal O}_{X,{\Bbb C}}}({\cal E})\;$
  with the associated surrogate
  $$
   \xymatrix{
    ({\cal E}_{\varphi},\nabla)\ar@{.>}[rd] \\
     & X_{\varphi} \ar@{->>}[rr]^-{f_{\varphi}}\ar[d]_{\pi_{\varphi}}
     && L\;\subset \; Y\\
     & X
   }
  $$
Suppose that:
 \begin{itemize}
  \item[] \hspace{-1em}[{\it Assumption.}]
   There exists an embedding $D^n\hookrightarrow X$
   such that
   \begin{itemize}
    \item[(1)]
     $\pi_{\varphi}^{-1}(D^n)$ contains a disjoint union
      $V=\coprod_{i=1}^lD^n_{(i)}$ of connected components\\ that satisfy:
       $f_{\varphi}|_{D^n_{(i)}}\,$, $i=1,\,\ldots\,,l\,$,
        are (orientation-preserving) embeddings\\
       with $\,\bigcap_{i=1}^l\,f_{\varphi}(D^n_{(i)})\ne \emptyset\,$;

    \item[(2)]\hspace{-1ex}\footnote{Condition (2)
                          can be loosened/generalized
                          by introducing the notion of
                           {\it disks with a multiplicity} and
                           {\it bundles/sheaves with a filter of
                                subbundles/subsheaves}, and
                           {\it allowing the rank of
                                ${{\cal E}_{\varphi}}|_V$ to vary
                                on different connected components of $V$};
                         cf.\ [L-Y5: Sec.~4.2,
                          Theme: `{\it The generically filtered structure
                           on the Chan-Patan bundle
                           over a special Lagrangian cycle on
                           a Calabi-Yau torus.}'] (D(6)).
                         Since this is a separate issue for D-branes,
                          here, to make the presentation simple,
                          we take
                           all the multiplicity of disks to be $1$ and
                           the filter to be trivial.}$\hspace{.6ex}$
     $({\cal E}_{\varphi}|_{V},\nabla_{V})
      \simeq
       ({\cal O}_{V,{\Bbb C}}\otimes {\Bbb C}^{r^{\prime}},\nabla_0)$
     for some $r^{\prime}<r$,
     where $\nabla_0$ is the flat connection associated
      to the built-in trivialization of
      ${\cal O}_{V,{\Bbb C}}\otimes {\Bbb C}^{r^{\prime}}$.
   \end{itemize}
 \end{itemize}
{For} the following construction,
 we will assume that $r^{\prime}=1$ for simplicity of notation.
Once the $r^{\prime}=1$ case is understood,
 one can then recover the $r^{\prime}>1$ case by taking
  $(\,\cdot\,)\otimes{\Bbb C}^{r^{\prime}}$
  to the bundles/sheaves constructed in the $r^{\prime}=1$ case.

Applying the confined isotopy and local gluing discussed
  in the previous two themes
 to $V\subset X_{\varphi}$ over $D^n\subset X$
  with the codimension-$1$ embedded submanifold-with-boundary $Z$ there
  taken to be, for example, a small smoothly embedded $(n-1)$-disk
  in $D^n$,
one obtains
 a $(-\varepsilon,\varepsilon)$-family of local deformations
 of the pair $(\varphi,\nabla)\,$:
 $$
  \xymatrix{
   (\hat{\cal E}_{(-\varepsilon,\varepsilon)},
    \hat{\nabla}^{(-\varepsilon,\varepsilon)}) \ar@{.>}[d]
    & & & & & \\
   \hat{\cal X}_{(-\varepsilon,\varepsilon)}
     \ar[rd]^{(\hat{\pi}_{(-\varepsilon,\varepsilon)},
               \hat{f}_{(-\varepsilon,\varepsilon)})}
     \ar@/^2ex/[rrrrrd]^-{\hat{f}_{(-\varepsilon,\varepsilon)}}
     \ar@/_/[rddd]_-{\hat{\pi}_{(-\varepsilon,\varepsilon)}}   && \\
    & ((-\varepsilon,\varepsilon)\times X)
       \times_{(-\varepsilon,\varepsilon)}
        ((-\varepsilon,\varepsilon) \times Y)
        \ar[rrrr]_-{pr_2}  \ar[dd]^-{pr_1}
      & & & & (-\varepsilon,\varepsilon)\times Y\;, \\  \\
    & (-\varepsilon,\varepsilon)\times X
  }
 $$
 where
  \begin{itemize}
   \item[$\cdot$]
    $\hat{\cal X}_{(-\varepsilon,\varepsilon)}
     \simeq ((-\varepsilon,\varepsilon)\times (X_{\varphi}-V))
       \cup_{\partial}{\cal V}_{(-\varepsilon,\varepsilon)}$
    where
     ${\cal V}_{(-\varepsilon,\varepsilon)}$
     is ${\cal M}_{(-\varepsilon,\varepsilon)}$
      in the construction in the first theme, with $M=V$;

   \item[$\cdot$]
    $(\hat{\cal E}_{(-\varepsilon,\varepsilon)},
      \hat{\nabla}^{(-\varepsilon,\varepsilon)})$
     is the extension of
     $(-\varepsilon,\varepsilon)\times
        ({\cal E}_{\varphi}|_{X_{\varphi-V}},\nabla|_{X_{\varphi}-V})$
     over ${\cal V}_{(-\varepsilon,\varepsilon)}$
     by trivial complex line bundles with a trivial (flat) connection;

   \item[$\cdot$]
    $\hat{\pi}_{(-\varepsilon,\varepsilon)}$ and
    $\hat{f}_{(-\varepsilon,\varepsilon) }$
     are the built-in morphisms over $(-\varepsilon,\varepsilon)$
     in the construction;
    note that
     $\hat{\pi}_{(-\varepsilon,\varepsilon)}|
                              _{{\cal V}_{(-\varepsilon,\varepsilon)}}:
      {\cal V}_{(-\varepsilon,\varepsilon)}
      \rightarrow (-\varepsilon,\varepsilon)\times D^n
                  \subset (-\varepsilon,\varepsilon)\times X$
     has constant degree $l$, counted with multiplicity,
     with
      ${\cal V}_{(-\varepsilon,0)}\rightarrow (-\varepsilon,0)\times D^n$
       a covering map  and
      ${\cal V}_{(0,\varepsilon)}\rightarrow (0,\varepsilon)\times D^n$
       a branched covering map with branch locus
       $(0,\varepsilon)\times \partial Z$.
  \end{itemize}
{From} the above data, one obtains two families
 $$
  \xymatrix{
   ((\hat{\pi},\hat{f})_{\ast}\hat{\cal E}_I,\nabla^I) \ar@{.>}[rd] \\
    & (\hat{\pi},\hat{f})(\hat{\cal X}_I)
       \ar@{->>}[rr]^-{f_I}\ar[d]_{\pi_I}
    && I \times L\;
       \subset \; I \times Y\;,\\
    & I\times X
  }
 $$
 where $I=$ $(-\varepsilon,0)$ or $(0,\varepsilon)$,
 of special Lagrangian morphisms from Azumaya spaces
  with a fundamental module, with a unitary minimally flat connection,
 and a transition between them via the family with
  $I=(-\varepsilon,\varepsilon)$.

\begin{remark}
{\it $[$large- vs.\ small-brane wrapping$]$.} {\rm
 Given
  a Calabi-Yau $n$-fold $Y$ and
  a special Lagrangian morphism
   with a unitary minimally flat connection-with-singularity
   $$
    \left(\,
     \varphi\::\:(X^{A\!z}, {\cal E})\: \longrightarrow\: Y\,,\,
     \nabla\, \right)
   $$
   specified by
   $\;\varphi^{\sharp}:{\cal O}_{Y,{\Bbb C}}^{\,\infty}\, \longrightarrow\,
     {\cal O}_X^{A\!z}:=\Endsheaf_{{\cal O}_{X,{\Bbb C}}}({\cal E})\;$
   with the data on the associated surrogate
   $$
    \xymatrix{
     ({\cal E}_{\varphi},\nabla)\ar@{.>}[rd] \\
      & X_{\varphi} \ar@{->>}[rr]^-{f_{\varphi}}\ar[d]_{\pi_{\varphi}}
      && L\;\subset \; Y\\
      & X
    }
   $$
  that satisfies the beginning assumption in the construction above.
 Then, the above construction deforms $\varphi$
   via locally deforming and gluing different sheets of
      $X_{\varphi}$ over $X$.
 When the starting collection $\pi_{\varphi}^{-1}(D^n)$ contains
  disks in different irreducible components of $X_{\varphi}$,
  the procedure in general lead then to
  $\varphi^{\prime}:(X^{A\!z},{\cal E}^{\prime})\rightarrow Y$
  with $X_{\varphi^{\prime}}$ containing an irreducible component
   $X_{\varphi^{\prime}}^{(0)}$ with larger volume
   such that both
    $f_{\varphi^{\prime}}|_{X_{\varphi^{\prime}}^{(0)}}
       \rightarrow L=\Image(\varphi)=\Image(\varphi^{\prime})$  and
    $\pi_{\varphi^{\prime}}|_{X_{\varphi^{\prime}}^{(0)}}\rightarrow X$
     have larger degree
   since volume (of branes) and degree (of maps)
    add under the construction.
 This gives rise, thus, to the phenomenon of and a transition between
   a ``small-brane wrapping" and a ``large-brane wrapping"
   of a special Lagrangian cycle $L$ in $Y$ in superstring theory.
}\end{remark}

Below is a simplest example
 that illustrates this particularly behavior of D-branes distinctly.
It is also an example that resembles long- vs.\ short-string wrapping
 most directly.

\begin{example}
{\bf [large- vs.\ small-brane wrapping on special Lagrangian 3-sphere].}
{\rm
 Let
  $X=S^3$ be oriented,
  $\pi^{\prime}:X^{\prime}=S^3\rightarrow X$
   be an orientation-preserving branched covering of $S^3$ on itself,
  $f^{\prime}:X^{\prime} \rightarrow L=S^3$
   be an orientation-preserving diffeomorphism,
  ${\cal E}^{\prime}$ be a complex line bundle over $X^{\prime}$.
 Since
   ${\cal E}^{\prime}$ is isomorphic to ${\cal O}_{S^3,{\Bbb C}}$,
  we endow ${\cal E}^{\prime}$ with the trivial connection
   $\nabla^{\prime}=d$ under such an isomorphism.
 It follows that
 $\nabla^{\prime}$ induces a unitary flat connection
   on $\pi^{\prime}_{\ast}{\cal E}^{\prime}$,
   with singularity on the branch locus of $\pi^{\prime}$,
  by endowing ${\cal E}^{\prime}\simeq {\cal O}_{S^3,{\Bbb C}}$
   with the standard Hermitian metric, for which $d$ is $U(1)$-flat.
 The map $(\pi^{\prime},f^{\prime}):X^{\prime}\rightarrow X\times L$
  defines now an embedding.
 Recall that
  the tangent bundle of an orientable close $3$-manifold is always trivial.
 In particular,
  let $\chi_t:L\rightarrow L$, $t\in{\Bbb R}$,
   be a flow on $L$
   generated by a nowhere-zero smooth vector field on $L$;
  for example, a Hopf flow on $L=S^3$.
 Consider the map
  $f:= \coprod_{i=0}^{d-1}\chi_{i\delta}\circ f^{\prime}:
       \amalg^d X^{\prime}\rightarrow L$
  for some $\delta>0$.
 Let
  $(X^{A\!z},{\cal E}_-)
    = (S^{3,A\!z},{\cal E}_-)
   := (S^3,
      {\cal O}_{S^3}^{A\!z}
              =\Endsheaf_{{\cal O}_{S^3,{\Bbb C}}}({\cal E}_-),
      {\cal E}_-=\oplus^d\pi^{\prime}_{\ast}{\cal E}^{\prime})$
   be an Azumaya $3$-sphere with a fundamental module.
 Let $L$ be realized as an embedded special Lagrangian $3$-sphere
  in a Calabi-Yau $3$-fold $Y$.
 Assume that $\delta$ is small enough so that
  $(\pi:=\amalg^d\pi^{\prime}, f):
                 \amalg^d X^{\prime}\rightarrow X\times L$
  is an embedding.
 Then $((\pi,f),({\cal E}^{\prime},\nabla^{\prime}))$
  defines a special Lagrangian morphism
  $\varphi:(X^{A\!z},{\cal E}_-)\rightarrow Y$,
   with $X_{\varphi}=\amalg^d X^{\prime} =\amalg^d S^3$,
        $\pi_{\varphi}=\pi$,
        $f_{\varphi}=f$,
        image $\Image\varphi=L$ with a multiplicity $d$, and
        ${\cal E}_{\varphi}\simeq {\cal O}_{X_{\varphi},{\Bbb C}}$, and
  a unitary minimally flat connection $\nabla$ on ${\cal E}_{\varphi}$
   that is isomorphic to $d$ on ${\cal O}_{X_{\varphi},{\Bbb C}}$.
 In particular, $X_{\varphi}$ has $d$-many components
  with each mapped to $L$ by the degree-$1$ restriction of
   $f_{\varphi}$
  (i.e.\ with each component wrapping $L$ once).

 Now deform $\varphi$ by setting
  $f_t:= \coprod_{i=0}^{d-1}\chi_{-ti\delta} \circ f^{\prime}$,
  $t\in[-1,0]$.
 Then $f_{-1}=f$ and
 for $t\in[-1,0)$,
  the associated special Lagrangian morphism
  $\varphi_t:(X^{A\!z},{\cal E})\rightarrow Y$ has
   $X_{\varphi_t}=\amalg^d X^{\prime}$,
   $\pi_{\varphi_t}=\pi$,
   $f_{\varphi_t}=f_t$, and
   image $\Image\varphi_t=L$ with a multiplicity $d$.
 {For} $t=0$, all components of $X_{\varphi}$ are deformed to coincide
  and become a single-component $X_{\varphi_0}\simeq X^{\prime}$
  with multiplicity $d$ indicated by the rank $d$ of
   $\oplus^d{\cal E}^{\prime}
    \simeq {\cal O}_{X^{\prime},{\Bbb C}}\otimes{\Bbb C}^d$
  on $X^{\prime}$.
 Let
  $Z\simeq D^2$ be an embedded $2$-disc in $X=S^3$  and
   perform a gluing construction in this subsection,
  with $\sigma\in \Sym_d$, say, to be transitive,
 to obtain an orientation-preserving $d$-fold branched covering
  $g:X^{\prime\prime}\simeq S^3\rightarrow X_{\varphi}$.
 The composition $f^{\prime\prime}$ of
  $X^{\prime\prime}\stackrel{g}{\longrightarrow} X_{\varphi_0}
     \stackrel{f_{\varphi_0}}{\longrightarrow}L$
  then is also a $d$-fold orientation-preserving branched covering.
 A ``large brane" (i.e.\ $X^{\prime\prime}$) is thus formed from
   gluing ``small branes" (i.e.\ $X_{\varphi}=\amalg^d S^3$) and
  it wraps $L$ now via $f^{\prime\prime}$ of degree $d$.
 Let
  $\pi^{\prime\prime}:X^{\prime\prime}\rightarrow X$
   be the built-in orientation-preserving branched covering map,
  ${\cal E}^{\prime\prime}\simeq {\cal O}_{X^{\prime\prime},{\Bbb C}}\,$,
   and
  ${\cal E}_+=\pi^{\prime\prime}_+{\cal E}^{\prime\prime}$.
 By deforming $f^{\prime\prime}:X^{\prime\prime}\rightarrow L$,
  for example, via the geodesic flow along $f$
   (e.g.\ from the induced Riemann metric on $L$ as a submanifold in $Y$)
   governed by a smooth section $\xi$ of
   ${f^{\prime\prime}}^{\ast}T_{\ast}L$
   that is non-zero except at the singular locus of $f^{\prime\prime}$,
  one can obtain a family of smooth maps
   $(\pi^{\prime\prime},
      f^{\prime\prime}_t):X^{\prime\prime}\rightarrow X\times L$,
    $t\in(0,1]$,
   with each an orientation-preserving embedding
     on an open dense subset of $X^{\prime\prime}$.
 This defines thus a family of special Lagrangian morphisms
  $\varphi_t:(X^{A\!z},{\cal E}_+)\rightarrow Y$, $t\in (0,1]$,
  whose image remains $L$  but
  which now involve large-brane wrapping
   $f_{\varphi_t}:X_{\varphi_t}\rightarrow L$ on $L$.

 Note that behind the above two-part construction,
   one over $[-1,0)$ and the other over $(0,1]$,
  is a family data over $I=[-1,1]$, as in the beginning of the theme:
  $$
   \xymatrix{
    (\hat{\cal E}_I, \hat{\nabla}^I) \ar@{.>}[d]
     & & & & & \\
    \hat{\cal X}_I
      \ar[rd]^{(\hat{\pi}_I,\hat{f}_I)}
      \ar@/^2ex/[rrrrrd]^-{\hat{f}_I}
      \ar@/_/[rddd]_-{\hat{\pi}_I}   && \\
     & (I\times X)\times_I (I\times Y)
         \ar[rrrr]_-{pr_2}  \ar[dd]^-{pr_1}
       & & & & I\times Y\;, \\  \\
     & I\times X
   }
  $$
  with
   $(\hat{\cal E}_I, \hat{\nabla}^I)
    \simeq ({\cal O}_{\hat{\cal X}_I,{\Bbb C}}\,,\,d)\,$.
}\end{example}

%
%
%
%
%
%
%

\bigskip

We pose a question here before leaving this subsection:\footnote{We
                          thank Yng-Ing Lee and Wenxuan Lu
                          for some discussions on this problem.}

\begin{definition}
{\bf [immersion in codimension $1$].} {\rm
 Let $M$, $N$ be smooth manifolds.
 A smooth map $f:M\rightarrow N$ is called
  {\it an immersion in codimension $1$}
  if there exists an open dense submanifold $M^{\circ}\subset M$
   with the (Hausdorff) codimension of $Z:= M-M^{\circ}$ $\ge 2$
   such that $f|_{M^{\circ}}:M^{\circ}\rightarrow N$ is an immersion.
 $Z$ is called an {\it exceptional locus} of $f$.
}\end{definition}

\begin{question}
{\bf [de-rigidification via large-brane wrapping].}
 \begin{itemize}
  \item[$\cdot$]
   Does there exist a smooth special Lagrangian map
    $f:L\rightarrow S^3\subset Y$
    that is an immersion (presumably of sufficiently high degree)
    in codimension $1$
   such that there exists a family of deformations of $f=:f_0$
    into smooth special Lagrangian maps $f_t:L\rightarrow Y$,
    $t\in [0,\epsilon)$ for some $\epsilon>0$,
    with $f_{t}(L)\ne S^3$ for $t\in (0,\epsilon)$?
 \end{itemize}
\end{question}
%

\bigskip

\subsection{Remarks/Questions/Conjectures.}

Two themes that immediately arise from the previous discussion
 are given here.
Each deserves a study in its own right.
The first theme is also relevant to Sec.~3.

\bigskip

\subsubsection{Cones of special Lagrangian cycles.}

Given a Calabi-Yau $n$-fold $Y=(Y,J,\omega,\Omega)$,
 let $\alpha=[L]\in H_n(Y;{\Bbb Z})$ be a homology class
  that is representable by a special Lagrangian submanifold $L$.
Then
 $$
   [\Real\Omega]\cdot\alpha\; >\; 0
    \hspace{1em}\mbox{and, hence,}\hspace{1em}
   [\Real\Omega]\cdot(-\alpha)\; <\; 0\,.
 $$
This implies that
 $-\alpha\in H_n(Y;{\Bbb Z})$ cannot be represented
  by any special Lagrangian cycle (or current).
Furthermore,
if $\alpha=[L_1]$ and $\beta=[L_2]$
 are two classes that are representable
  by special Lagrangian submanifolds,
 then $\alpha+\beta$ is representable by the special Lagrangian cycle
  $L_1+L_2$.
This gives a foundation for the following definition:

\begin{ssdefinition-prototype}
{\bf [cone of special Lagrangian cycles].} {\rm
 The following cone
  $$
   C^{sL}(Y)\;:=\;
    \left\{\,
      \sum_{i\in I} a_i[L_i]\,:\,
       \begin{array}{l}
         |I|<\infty\,,\; a_i\in {\Bbb R}_{\ge 0}\,, \\[.2ex]
         \mbox{$L_i$
          a special Lagrangian submanifold-with-singularity}
       \end{array}
     \right\}
  $$
  in $H_n(Y;{\Bbb R})$ is called  the
  {\it cone of special Lagrangian cycles} of the Calabi-Yau $n$-fold $Y$.
 With ${\Bbb R}_{\ge 0}$ replaced by ${\Bbb Q}_{\ge 0}$,
  one can also define $C^{sL}_{\Bbb Q}(Y)\,$.
}\end{ssdefinition-prototype}

\noindent
This is only a prototypical definition as there are various
 enhancements/refinements to it:
 \begin{itemize}
  \item[$\cdot$]
   The above definition is based on the underlying choice of
    {\it equivalence relation of special Lagrangian cycles}:
    $L_1\sim L_2$ if $[L_1]=[L_2]$ in $H_n(Y;{\Bbb Z})$.
   One can use other finer equivalence relations, for example,
    via Lagrangian or special Lagrangian cobordisms.

  \item[$\cdot$]
   One may specify more specifically the {\it singularities} allowed
    in special Lagrangian cycles or currents.
   In particular, let
 \end{itemize}
 \vspace{-2.4ex}
    $$
     C^{sL}(Y)^0\;:=\;
      \left\{
        \sum_{i\in I} a_i[L_i]\,:\,
         \begin{array}{l}
           |I|<\infty\,,\; a_i\in {\Bbb R}_{\ge 0}\,, \\[.2ex]
           \mbox{$L_i$
            an immersed special Lagrangian submanifold}
         \end{array}
       \right\}\; \subset\; C^{sL}(Y)\,.
    $$
 \begin{itemize}
  \item[]
   Then, it follows from the immersed version of Theorem~A.1.2
   that $C^{sL}(\,\bullet\,)^0$ is locally constant
   in the dual Hodge bundle ${\cal H}^{\vee}$
   over the moduli space ${\cal M}$ of smooth deformations of $Y$.
 \end{itemize}
With cones of special Lagrangian cycles to complex deformations
 as Mori cones ([Ko-M]) to K\"{a}hler deformations in mind,
two major questions are then:

\begin{ssquestion}
{\bf [cone of special Lagrangian cycles].}   \samepage
 \begin{itemize}
  \item[$(1)$] {\bf [structure of cone].}\\
   Structure of $C^{sL}(Y)$ (resp.\ $C^{sL}(Y)^0\,$,
    $\;\overline{C^{sL}(Y)}\,$, $\;\overline{C^{sL}(Y)^0}\,$, ...$\,$),
    existence of extremal rays, ..., etc.?

  \item[$(2)$] {\bf [role in deformation].}
   How does the structure of $C^{sL}(Y)$ (resp.\ $C^{sL}(Y)^0\,$, ...$\,$)
    relate to the vanishing cycle of $\;Y$
    when the complex structure of $\;Y$ is deformed to singularity?
 \end{itemize}
\end{ssquestion}

\bigskip

\subsubsection{A genus-like expansion of the path-integral
    of lower-dimensional branes:
    Alexander-Hilden-Lozano-Montesinos-Thurston/Hurwitz\\
    meeting Polchinski-Grothendieck.}

Not all D3-brane topologies are equal from the viewpoint
 of Azumaya geometry.
This suggests a genus-like expansion of the path-integral
 of D3-branes in type IIB string theory.
Similarly for D2-branes in type IIA string theory and
 for M2-branes in M-theory.

\bigskip

\begin{flushleft}
{\bf Fundamental D$3$-branes from the viewpoint of Azumaya geometry:
     Alexander-Hilden-Lozano-Montesinos-Thurston
     meeting Polchinski-Grothendieck.}
\end{flushleft}
Recall\footnote{{\it Notation.}
           In this subsubsection, we will use: (only here)
            \begin{itemize}
             \item[$\cdot$]
              $L^1$ to denote a {\it link}
               (i.e.\ possibly disconnected, embedded,
                      $1$-dimensional submanifold) of $S^3$;
             \item[$\cdot$]
              $K^1$ to denote a {\it knot}
               (i.e.\  connected, embedded, $1$-dimensional submanifold)
              of $S^3$;  and
             \item[$\cdot$]
              $L^3$ to denote a {\it special Lagrangian}
               ($3$-dimensional) submanifold in a Calabi-Yau $3$-fold.
            \end{itemize}}
 the following classical theorems in the study
 of $3$-manifold topology:

\begin{sstheorem}
{\bf [branched covering].} {\rm (Alexander [Al], 1920.)}
 Any closed, connected, orientable $3$-manifold is realizable
 as a branched covering of $S^3$.
\end{sstheorem}

\begin{sstheorem}
{\bf [3-fold enough].} {\rm (Hilden [Hil] and Montesinos [Mon], 1976.)}
 Any closed, connected, orientable $3$-manifold is realizable
 as a $3$-fold (i.e.\ degree-$3$) irregular branched covering of $S^3$
  with the branch locus in $S^3$ a knot.
\end{sstheorem}

\begin{sstheorem}
{\bf [universal link].} {\rm (Thurston [Thu], 1982.)}
 There exists a ($6$-component) link $L^1$ in $S^3$
  such that any closed, connected, orientable $3$-manifold
   is realizable as a branched covering of $S^3$ that is branched
   only over $L^1$.
\end{sstheorem}

\begin{sstheorem}
{\bf [universal knot].} {\rm (Hilden-Lozano-Montesinos [H-L-M], 1985.)}
 There exists a knot $K^1$ in $S^3$
  such that any closed, connected, orientable $3$-manifold
   is realizable as a branched covering of $S^3$ that is branched
   only over $K^1$.
\end{sstheorem}

What the fundamental theorems of
 Alexander-Hilden-Lozano-Montesinos-Thurston
 mean to D3-branes along the line of the Polchinski-Grothendieck Ansatz
 is, in particular, that:

\begin{sstheorem}
{\bf [$S^{3,A\!z}$ and fundamental D3-brane].}
 Let
  $Y$ be a Calabi-Yau $3$-fold and
  $L^3$ be a finite union of compact smooth special Lagrangian
  submanifolds in $Y$, each of which is generically an embedding.
 Then there exists a morphism $\varphi:S^{3,A\!z}\rightarrow Y$
  from an Azumaya $3$-sphere such that the image $\varphi(S^{3,A\!z})$
  of $\varphi$ is exactly $L^3$.
 Furthermore, one can require that
  the rank of the fundamental module ${\cal E}$ of $S^{3,A\!z}$
  be $\,3\cdot($number of irreducible components of $L^3)\,$.
 Or one may require that
  $\pi_{\varphi}:S^3_{\varphi}\rightarrow S^3$
  be a branched-covering map
  over a universal knot or a universal link in $S^3$.
\end{sstheorem}

\noindent
This specifies morphisms from $(S^{3,A\!z},{\cal E})$
 as most fundamental D$3$-branes from the viewpoint of Azumaya geometry.

\bigskip

\begin{flushleft}
{\bf A genus-like expansion of the path-integral of D3-branes.}
\end{flushleft}
In understanding the path-integral for higher-dimensional objects,
 the usual sum-over-Feynman-diagrams in the quantum field theory of
 (point-like) particles is replaced by
  \begin{itemize}
   \item[(1)]
    a sum over the space of all the topologies of the
     (Wicked-rotated) D3-brane world-volume
     of the extended objects  and

   \item[(2)]
    an integration over the space of maps from each topology.
  \end{itemize}
Step~(1) has already imposed a very challenging difficulty
 to understanding the path-integral even in a physicist's way.
In general, one needs to specify these topologies by hand.
See, e.g., [B-P] for a recent study.

Recall
 how closed strings interact to give rise to Riemann surfaces
  as (Wick-rotated) closed-string world-sheets
  (cf.\ [G-S-W: Sec.~1.4, Sec.~3.3]) and
 the path-integral of closed strings
  (cf.\ [Po: vol.~I, Sec.~3.1, Sec.~3.2]).
Consider the replacements
  \begin{itemize}
   \item[$\cdot$]
    \parbox{9em}{string}
     $\longrightarrow$\hspace{2em} \parbox{16em}{D3-brane,}

   \item[$\cdot$]
    \parbox{9em}{ordinary maps\\ into space-time}
     $\longrightarrow$\hspace{2em}
     \parbox{16em}{morphisms from Azumaya manifolds\\ into space-time,}
  \end{itemize}
 with the same generic interaction assumption in string theory:
 \begin{itemize}
  \item[$\cdot$]
  $[\,${\sl generic interaction assumption}$\,]$.\hspace{1em}
  {\it Interactions of branes in a space-time happen one at a time
        at a point
       with respect to some local equal-time slicing of space-time.}
 \end{itemize}
Then,
Theorem~2.4.2.5 suggests that
 $$
  Z_{D3}(Y)\;
  =\; \sum_{\left\{\!\!\begin{array}{l}
               \mbox{\scriptsize\it fundamental}\\[-.8ex]
               \mbox{\scriptsize\it $4$-manifolds $X$}
                       \end{array}\!\!\right\}}\,
      \int_{\!\!\!\!
            \begin{array}{l}
             \\[-.8ex]
             \left\{\!\!\begin{array}{l}
                \mbox{\scriptsize\it morphisms}\\[-.8ex]
                \mbox{\scriptsize
                      $\varphi:(X^{\!A\!z},{\cal E})\rightarrow Y$}
                        \end{array}\!\!\right\}
            \end{array}}\!\!\!\!{\cal D}\varphi\:
        \int_{\{\cdots\}} {\cal D}\,(\mbox{\it other fields $\cdots$})\,
         e^{-S(\varphi,\,\cdots\,)}\;,
 $$
where
 \begin{itemize}
  \item[$\cdot$]
   $Z_{D3}(Y)$
    is the path-integral of D3-branes on a target-space(-time) $Y$,

  \item[$\cdot$]
   the set $\{$fundamental $4$-manifolds $X$$\}$
   consists of all homeomorphism classes of closed orientable
   $4$-manifolds $X$ that are obtained from (self-)connected sums
   of finite disjoint unions $\amalg^{\bullet}(S^3\times S^1)$
   of $S^3\times S^1$,

  \item[$\cdot$]
   the `{\it other fields $\cdots$}' here includes gauge fields
    on the branes, realized as connections-with-singularities
    on $(X_{\varphi},{\cal E}_{\varphi})$,

  \item[$\cdot$]
   $S(\varphi,\,\cdots\,)$
    is the action of the (Wicked-rotated) D3-brane theory.
 \end{itemize}
The $4$-manifolds $X$ involved here in the path-integral are now
 direct generalization of closed orientable $2$-manifolds
  in the case of strings
 in the sense that closed orientable $2$-manifolds
  can be obtained as direct sums
  of finite disjoint unions of $S^1\times S^1$.
This motivates a conjecture:

\begin{ssconjecture}
{\bf [genus-like expansion of path-integral for D3-branes].}
 There exists a built-in natural genus-like expansion
  for the path-integral $Z_{D3}(Y)$ of a perturbative D3-brane theory.
\end{ssconjecture}

\bigskip

\begin{flushleft}
{\bf D2-branes and M2-branes: Hurwitz meeting Polchinski-Grothendieck.}
\end{flushleft}
The same reasoning suggests a similar conjecture
 for D2-branes and M2-branes respectively.\footnote{We
                    thank Pei-Ming Ho for discussions and
                    for drawing our attention to recent new developments
                    in the study of multiple M$2$-branes.}
Denote by $S^{2,A\!z}$ for some $(S^{2,A\!z},{\cal E})$
 with ${\cal E}$ unspecified.
Recall that any Riemann surface branched-covers $S^2$.

\begin{sstheorem}
{\bf [$S^{2,A\!z}$ and fundamental D2/M2-brane].}
 Let
  $Y$ be a Calabi-Yau $3$-fold or Joyce/$\,G_2$ $7$-manifold and
  $C$ be a finite union of compact complex smooth curves immersed in $Y$,
   each of which is generically an embedding.
 Then there exists a morphism $\varphi:S^{2,A\!z}\rightarrow Y$
  from an Azumaya $2$-sphere such that the image $\varphi(S^{2,A\!z})$
  of $\varphi$ is exactly $C$.
\end{sstheorem}

$3$-manifolds that are obtained from a (self-)connected sum
 of a finite disjoint union $\amalg^{\bullet}(S^2\times S^1)$
 of $S^2\times S^1$ now play a special role.

\begin{ssconjecture}
{\bf [genus-like expansion of path-integral for D2/M2-branes].}
 There exists a built-in natural genus-like expansion
  for the path-integral $Z_{D2}(Y)$ of a perturbative D2-brane theory.
 Similarly, for the path-integral $Z_{M2}(Y)$
  of a perturbative M2-brane theory.
\end{ssconjecture}

\bigskip

\section{Morse cobordisms of A-branes on Calabi-Yau 3-folds
          under a reverse split attractor flow
          at a wall of marginal stability:
         Denef-Joyce meeting Polchinski-Grothendieck.}

In this section, we study how A-branes,
 in the sense of Definition-Prototype~2.1.14,
 deform under an attractor flow at the wall of marginal stability
 on a moduli space of complex structures on a Calabi-Yau $3$-fold.
We first set up notations and recap a result of Denef ([De3])
  in Sec.~3.1 and
 then use the results of Joyce ([Joy3]) to deal with the technical issue
 of smoothing a special Lagrangian submanifold
 with transverse self-intersections.
Together,
 this gives us a Morse cobordism of A-branes on Calabi-Yau 3-folds
 under a reverse split attractor flow at a wall of marginal stability.

\bigskip

\subsection{Evolution of {\it Im}$\,(e^{-i\alpha_{\Gamma}}\Omega)$
    along a $\Gamma$-attractor flow trajectory \`{a} la Denef.}

The following discussion follows works [De3] and [De1] of Denef,
 with some background from [Gr], [Str], [Ti], and [To]
  to set up notations and
 with a mild convention change to fit calibrated geometry ([Ha-L]).

\bigskip

\begin{flushleft}
{\bf Basic setup, notations, facts, and identities.}
\end{flushleft}
Let
 $(Y,\omega,J)$ be a (smooth) Calabi-Yau $3$-fold  and
 ${\cal M}:={\cal M}_{\complexscriptsize}^{\smoothscriptsize}$
  be the moduli space of complex deformations of $Y$.
Let $h^{p,q}:= \dimm_{\Bbb C}\,H_{\bar{\partial}}^{p,q}(X;{\Bbb C})$;
 then ${\cal M}$ is a complex manifold of dimension $h^{2,1}$,
 (cf.~[Ti] and [To]).
{For} the purpose of this note,
we will assume
 (by taking either a smaller ${\cal M}$ or a covering of it) that
 an universal Calabi-Yau $3$-fold
 $\pi:{\cal Y}\rightarrow {\cal M}$ over ${\cal M}$ exists,
  with a relative polarization $[\omega]_{{\cal Y}/{\cal M}}$
   from the K\"{a}hler class $[\omega]$ on $Y$.
Let
 ${\cal H}
   :=R^3\pi_{\ast}{\Bbb C}_{\cal Y}\otimes_{\Bbb C}{\cal O}_{\cal M}$
  be the Hodge bundle on ${\cal M}$,
  where ${\Bbb C}_{\cal Y}$ is the constant sheaf on ${\cal Y}$
   associated to ${\Bbb C}$.
${\cal H}$ is holomorphic of rank $\dim_{\Bbb C}H^3(Y,{\Bbb C})$.
The complex symplectic product
 on $\langle\,\,\cdot\,,\,\cdot\,\rangle$ on $H^3(Y,{\Bbb C})$,
 defined by
  $\langle \alpha,\beta \rangle := \int_Y\alpha\wedge\beta$,
 induces a complex symplectic product,
  still denoted by $\langle\,\,\cdot\,,\,\cdot\,\rangle$,
  on fibers of ${\cal H}$.
Let
 ${\cal H} = {\cal H}^{3,0}+{\cal H}^{2,1}+{\cal H}^{1,2}+{\cal H}^{0,3}$
 be the Hodge decomposition of ${\cal H}$.
In general,
 ${\cal H}^{p,q}$ is only a smooth complex vector bundle on ${\cal M}$.
However,
let
 $F^p{\cal H}:= \oplus_{q\ge p}{\cal H}^{q, 3-q}$, $p=0,\,1,\,2,\,3$.
Then $F^p{\cal H}$ are holomorphic on ${\cal M}$,
 and
$F^{\bullet}{\cal H}\, :\,
 {\cal H}=F^0{\cal H} \supset F^1{\cal H}
    \supset F^2{\cal H} \supset F^3{\cal H}={\cal H}^{3,0}$
 gives the Hodge filtration of ${\cal H}$.
In particular,
 ${\cal H}^{3,0}$ is a holomorphic (complex) line bundle on ${\cal M}$
 while ${\cal H}^{0,3}=\overline{{\cal H}^{3,0}}$ is an anti-holomorphic
  line bundle on ${\cal M}$.

The inclusion $\iota: {\Bbb Z}\subset {\Bbb C}$ as abelian groups
 induces an inclusion
 $\iota: R^3\pi_{\ast}{\Bbb Z}_{\cal Y}\subset {\cal H}$,
which induces a flat connection,
 namely the Gauss-Manin connection $\nabla$, on ${\cal H}$.
In terms of the canonical (up to a monodromy effect)
 local trivialization of ${\cal H}$
 from $\iota(R^3\pi_{\ast}{\Bbb Z}_{\cal Y})$,
$\nabla$ is simply
 the differential $d=\partial + \bar{\partial}$
 ($:= \sum_{a=1}^{h^{2,1}} \frac{\partial}{\partial t^a}\otimes dt^a
     + \sum_{a=1}^{h^{2,1}}
         \frac{\partial}{\partial \bar{t}^a}\otimes d\bar{t}^a$)
 with respect to complex local coordinates
  $\mbox{\boldmath $t$}\,= (t^a)_a$,
  (or more completely, $\mbox{\boldmath $t$}\,=(t^a, \bar{t}^a)_a$),
  on ${\cal M}$.
$\nabla$ preserves
 neither the decomposition $\oplus_{p=0}^3{\cal H}^{p,3-p}$
 nor the filtration $F^{\bullet}{\cal H}$ of ${\cal H}$.
However, it has the property that
 $\nabla F^p{\cal H} \subset F^{p-1}{\cal H}\otimes\Omega_{\cal M}^1$,
cf.\ [Gr].

In terms of the Hodge decomposition of ${\cal H}$ on ${\cal M}$,
 the Weil-Petersson metric $\omega_{\WPscriptsize}$ on ${\cal M}$,
  defined through the Ricci flat metric on each fiber
  $Y_{\mbox{\scriptsize\boldmath $t$}}:=\pi^{-1}(\mbox{\boldmath $t$})$
  of $\pi:{\cal Y}\rightarrow {\cal M}$
  determined by $[\omega]_{{\cal Y}/{\cal M}}$ (cf.~[Yau]),
 has a K\"{a}hler potential $K_{\WPscriptsize}$
  that can be expressed locally purely topologically
  as a smooth real-valued function
  $$
   K\;
    =\; K(\mbox{\boldmath $t$})\;
   :=\; K_{\WPscriptsize}(\mbox{\boldmath $t$})\;
    =\; - \log
          \left(i\,\int_{Y_{\mbox{\tiny\boldmath $t$}}}\,
            \Omega^{(0)}(\mbox{\boldmath $t$})
            \wedge\overline{\Omega^{(0)}(\mbox{\boldmath $t$})}
          \right)
          +\, \log\left(\rule{0em}{1em}8\cdot\vol(Y)\right)
  $$
 on ${\cal M}$,
 where $\Omega^{(0)}=\Omega^{(0)}(\mbox{\boldmath $t$})$
  is a local holomorphic section of ${\cal H}^{3,0}$,
(cf.~[Ti] and [To]).
Define
 $$
  \Omega\;
  =\; \Omega(\mbox{\boldmath $t$})\;
  =\; e^{K(\mbox{\scriptsize\boldmath $t$})/2}\,
      \Omega^{(0)}(\mbox{\boldmath $t$})
 $$
 so that $\Omega(\mbox{\boldmath $t$})$,
   as a holomorphic $3$-form on $Y_{\mbox{\scriptsize\boldmath $t$}}$,
  satisfies the normalization condition\footnote{This
                normalization condition
                 fits more naturally
                  into the standard setting of calibrated geometry and
                 is used, for example, in the work of Joyce [Joy3].
               Under this normalization,
                an orientable (real) $3$-dimensional submanifold $L$
                in $Y_{\mbox{\tiny\boldmath $t$}}$
                has volume {\it vol}\,$(L)\ge |Z(L)|$
               and the equality holds
                when $L$ is a special Lagrangian submanifold.
               In particular,
                $\Omega$ here
                 $= \sqrt{8\,\mbox{\it vol}\,(Y)}\,
                     \cdot\,\mbox{($\Omega$ in [De3])}$.}
 $$
  \frac{i}{2^3}\,
    \Omega(\mbox{\boldmath $t$})
    \wedge\overline{ \Omega(\mbox{\boldmath $t$})}\;
   =\; \frac{1}{3!}\,
       ( \omega_{{\cal Y}/{\cal M}}|
                          _{Y_{\mbox{\tiny\boldmath $t$}}} )
       \wedge ( \omega_{{\cal Y}/{\cal M}}|
                          _{Y_{\mbox{\tiny\boldmath $t$}}} )
       \wedge ( \omega_{{\cal Y}/{\cal M}}|
                          _{Y_{\mbox{\tiny\boldmath $t$}}} )\,.
 $$
Note that $\Omega$ is now only a smooth local section of ${\cal H}^{3,0}$.

On a local chart $U\subset {\cal M}$ on which
 ${\cal H}$ is canonically trivialized
   via the Gauss-Manin connection $\nabla$ and
  $\Omega^{(0)}$ and, hence, $\Omega$ are defined,
 one has the following basic identities and notations:
(Below,
  $\mbox{\boldmath $t$}\,=(t_a, \bar{t}_a)_a$
   is the local coordinates on $U$  and
  $\nabla=d$ under the trivialization of ${\cal H}|_U$.)
\begin{itemize}
 \item[(1)]
  $\frac{\partial}{\partial t^a}\Omega
    = \left(-\frac{1}{2}\frac{\partial}{\partial t^a}K\right)\,\Omega
      + \chi_a$
   with $\chi_a$ a section of ${\cal H}^{2,1}$.
  $\chi_a$ satisfies
   $\int_{Y_{\mbox{\tiny\boldmath $t$}}}\chi_a\wedge\overline{\chi_b}
    = 8\,\vol(Y)\cdot i g_{a\bar{b}}$,
   where
    $g_{a\bar{b}}
     = \frac{\partial}{\partial t^a}\frac{\partial}{\partial \bar{t}^b}K$
    is the Weil-Petersson metric on $U$.

 \item[$\cdot$]
  We adopt the notation from special geometry (cf.~[Str]) to denote
  $D_a:= \frac{\partial}{\partial t^a}
         + \frac{1}{2}\frac{\partial}{\partial t^a}K$;
 e.g.,
  the $(2,1)$-component $\chi_a$
   of $\frac{\partial}{\partial t^a}\Omega$
   is then denoted by $D_a\Omega$ and
  $D_aZ_{\Gamma} := \int_{\Gamma}D_a\Omega =\int_{\Gamma}\chi_a$.
 Similarly,
  $\bar{D}_{\bar{a}}
   :=  \frac{\partial}{\partial \bar{t}^a}
       + \frac{1}{2}\frac{\partial}{\partial \bar{t}^a}K$  and
  for $\bar{D}_{\bar{a}}\bar{\Omega}$ and
      $\bar{D}_{\bar{a}}\bar{Z}_{\Gamma}$
   by taking the complex conjugate.

 \item[(2)]
  A (real) harmonic $3$-form $\Xi$ on $Y_{\mbox{\tiny\boldmath $t$}}$
   has a Hodge decomposition given by
  $$
   \Xi\;=\; \frac{1}{8\,\vol(Y)}
            \left(
                 i\,\bar{Z}(\Xi)\,\Omega\,
             -\, i\, g^{a\bar{b}}\,\bar{D}_{\bar{b}}\bar{Z}(\Xi)\,D_a\Xi
             +\, i\, g^{b\bar{a}}\,D_bZ(\Xi)\,\bar{D}_{\bar{a}}\Xi\,
             -\, i\,Z(\Xi)\,\bar{\Omega}
            \right),
  $$
  where
    $Z(\Xi)$  (resp.\ $\bar{Z}(\Xi)$)  denotes
    $\int_{Y_{\mbox{\tiny\boldmath $t$}}}
               \Xi\wedge \Omega(\mbox{\boldmath $t$})$
    (resp.\ $\int_{Y_{\mbox{\tiny\boldmath $t$}}}
               \Xi\wedge \bar{\Omega}(\mbox{\boldmath $t$})$).
\end{itemize}
They follow from direct computations,
 the holomorphicity (resp.\ anti-holomorphicity) of $\Omega^{(0)}$
  (resp.\ $\overline{\Omega^{(0)}}$) on $U$,
 and the variation property of $\Omega^{(0)}$ (e.g.\ [Ti: Lemma 7.2]).

\bigskip

\begin{flushleft}
{\bf Well-definedness of
     $|Z_{\Gamma}|$ and $e^{-i\alpha_{\Gamma}}\Omega$ on ${\cal M}$.}
\end{flushleft}
By construction, $\Omega$ and, hence, $Z_{\Gamma}$
 are locally well-defined only up to
 a same smooth unit-modular complex-valued function.
Thus, both $|Z_{\Gamma}|$ and $e^{-i\alpha_{\Gamma}}\Omega$
  are well-defined locally.
They automatically become globally well-defined on ${\cal M}$
 respectively
 as a positive function and as a smooth section of ${\cal H}^{3,0}$.
It follows that
 $e^{i\alpha_{\Gamma}}\bar{\Omega}$,
 $\Real(e^{-i\alpha_{\Gamma}}\Omega)
  := \frac{1}{2}\left( e^{-i\alpha_{\Gamma}}\Omega
                     + e^{ i\alpha_{\Gamma}}\bar{\Omega} \right)$, and
 $\Imaginary(e^{-i\alpha_{\Gamma}}\Omega)
  := \frac{1}{2i}\left( e^{-i\alpha_{\Gamma}}\Omega
                      - e^{ i\alpha_{\Gamma}}\bar{\Omega} \right)$
 are all well-defined, as smooth sections of ${\cal H}$ on ${\cal M}$.

\bigskip

\begin{flushleft}
{\bf Evolution of {\it Im}$\,(e^{-i\alpha_{\Gamma}}\Omega)$
     along a $\Gamma$-attractor flow trajectory.}
\end{flushleft}
A trajectory of a flow on ${\cal M}$ gives an embedding
 $\gamma: I \hookrightarrow {\cal M}$,
 where $I$ is an interval in ${\Bbb R}_+$.
There exists thus a neighborhood ${\cal N}$
 of $\gamma(I)$ in ${\cal M}$ that is contractible.
The restriction ${\cal H}|_{\cal N}$ of the Hodge bundle to ${\cal N}$
 becomes canonically trivialized
  via the Gauss-Manin connection $\nabla$ on ${\cal H}$,
 through which all fibers of ${\cal H}$ along a trajectory of a flow
  can be identified.
With such an identification and
  with the preparation from the previous two themes,
 one can now express $\Imaginary(e^{-i\alpha_{\Gamma}}\Omega)$
  along a $\Gamma$-attractor flow trajectory explicitly:

\bigskip

\begin{proposition}
{\bf
[$\Imaginary(e^{-i\alpha_{\Gamma}}\Omega)$ along $\Gamma$-attractor flow].}
{\rm ([De3: Sec.~3.1, Eq.~(3.5)].)}
 With the canonical trivialization of the Hodge bundle along  and
  letting $\Gamma^{\vee}$ be the Harmonic $3$-form representing
   the Poincar\'{e} dual $\in H^3(Y;{\Bbb R})$ of $\Gamma$,
 then
 the $3$-form $\Imaginary(e^{-i\alpha_{\Gamma}}\Omega)$
  on Calabi-Yau $3$-folds
  along a trajectory of the attractor flow on ${\cal M}$
  associated to $Z_{\Gamma}$ is given by
 \begin{eqnarray*}
  \Imaginary\left(e^{-i\alpha_{\Gamma}(\mu)}\Omega(\mu)\right)
   & =
   & (-\,4\,\vol(Y)\cdot\mu\,\tau(\mu))\,\Gamma^{\vee}\,
       +\, \frac{\mu}{\mu_0}\,
             \Imaginary\left(
                 e^{-i\alpha_{\Gamma}(\mu_0)}\Omega(\mu_0)
                       \right),   \\[1.2ex]
  && \mbox{with}\;\;\;
  \tau(\mu)\,
   =\, -\int_{\mu_0}^{\mu}
         \frac{d\mu^{\prime}}
              {|Z_{\Gamma}(\mu^{\prime})|\,{\mu^{\prime}}^2}\,,
     \hspace{2em}\mbox{for}\hspace{1em}
     \mu\,,\; \mu_0\;\in\; {\Bbb R}_{>0}\,.
 \end{eqnarray*}
\end{proposition}

\begin{proof}
 With notations from previous themes in this subsection,
 let
  ${\cal U}:= \{U_{\beta}\}_{\beta}$
   be a locally finite good\footnote{I.e.\
           all the intersections
            $U_{\beta_1\,\cdots\,\beta_l}
              := U_{\beta_1}\cap\,\cdots\,\cap U_{\beta_l}$,
              $l\in {\Bbb N}$,
            are contractible.}
  atlas on ${\cal M}$ over which
   ${\cal H}^{3,0}$ becomes trivial (as a holomorphic bundle)  and
   one can introduce a nowhere-zero holomorphic section
    $\Omega^{(0)}_{\beta}$ of ${\cal H}^{3,0}|_{U_{\beta}}$
    on $U_{\beta}$ for each $\beta$.
 Consider a $U_{\beta}=:U$ and
  trivialize ${\cal H}|_U$ via the flat connection $\nabla$ on ${\cal H}$.
 Let $\Omega^{(0)}=\Omega^{(0)}_{\beta}$.
 Recall
  the K\"{a}hler potential $K:=K_{\beta}$ on $U_{\beta}$
   constructed from $\Omega^{(0)}$ and
  the smooth section $\Omega:=\Omega_{\beta}$
   of ${\cal H}^{3,0}|_{U_{\beta}}$
   from a normalization of $\Omega^{(0)}_{\beta}$.
 Let
  $Z = Z(\mbox{\boldmath $t$})
    := Z_{\Gamma}(\mbox{\boldmath $t$})
     = \int_{\Gamma}\Omega(\mbox{\boldmath $t$})
     = \int_{Y_{\mbox{\tiny\boldmath $t$}}}
         \Gamma^{\vee}\wedge\Omega(\mbox{\boldmath $t$})$
   be the central charge function on $U$ associated to $\Gamma$  and
  $\alpha
     = \alpha(\mbox{\boldmath $t$})
    := \alpha_{\Gamma}(\mbox{\boldmath $t$})$
   be the phase function on $U$ associated to $Z$.
 An attractor flow trajectory
  $\gamma(\mu)=\mbox{\boldmath $t$}(\mu)\,$, $\mu\in {\Bbb R}_{>0}\,$,
   on $U$ associated to $Z$
  satisfies the flow equation
   \begin{eqnarray*}
   \mu\frac{\partial}{\partial\mu}\;
    :=\; \gamma_{\ast}\,\left(\mu\frac{\partial}{\partial\mu}\right)
    & = & \left( d\left( \log|Z|^2 \right)\right)^{\sim} \\
    & = & \left( g^{a\bar{b}}\,
                 \frac{\partial}{\partial \bar{t}^b} \log|Z|^2 \right)
           \frac{\partial}{\partial t^a}\,
        +\, \left( g^{b\bar{a}}\,
                   \frac{\partial}{\partial t^b} \log|Z|^2 \right)
           \frac{\partial}{\partial \bar{t}^a}   \\[1.2ex]
    & = & \left( g^{a\bar{b}}\,\frac{\bar{D}_{\bar{b}}\bar{Z}}{\bar{Z}}
          \right) \frac{\partial}{\partial t^a}\,
          +\, \left( g^{b\bar{a}}\,\frac{D_b Z}{Z}
              \right) \frac{\partial}{\partial \bar{t}^a}\,,
        \hspace{1em}\mbox{restricted to the image of $\gamma$}\,,
   \end{eqnarray*}
  where $(\,\cdot\,)^{\sim}$ is the metrical equivalent vector field
   to a $1$-form $(\,\cdot\,)$ on $U$,
   with respect to the Weil-Petersson metric
    $\,g = \sum_{a,b}\,g_{a\bar{b}}\,dt^a\otimes d\bar{t}^b\,$
   (previously denoted by its associated $2$-form
    $\omega_{\WPscriptsize}$).
 On the other hand,
  \begin{eqnarray*}
   d\left( \Imaginary (e^{-i\alpha}\Omega) \right)
    & = & \frac{1}{2i}
          \left( e^{-i\alpha}\, D_a\Omega \,
           -\, \frac{e^{-i\alpha}\,\Omega\,
                        +\, e^{i\alpha}\,\bar{\Omega}}{2}\,
               \frac{D_aZ}{Z}
          \right) dt^a  \\[1.2ex]
    && \hspace{6em}
      +\, \frac{1}{2i}
          \left( -\,e^{i\alpha}\, D_{\bar{a}}\bar{\Omega} \,
           +\, \frac{e^{-i\alpha}\,\Omega\,
                        +\, e^{i\alpha}\,\bar{\Omega}}{2}\,
               \frac{D_{\bar{a}}\bar{Z}}{\bar{Z}}
          \right) d\bar{t}^a\,.
  \end{eqnarray*}
 Thus, along a flow trajectory,
  \begin{eqnarray*}
   && \mu\frac{d}{d\mu}\,
            \Imaginary (e^{-i\alpha}\Omega)\;
       =\; d\left( \Imaginary (e^{-i\alpha}\Omega) \right)
            \left( \mu\frac{\partial}{\partial\mu} \right) \\[1.2ex]
   && =\;
     \frac{1}{2i}\,
     g^{a\bar{b}}\,
     \left( e^{-i\alpha}\, D_a\Omega \,
      -\, \frac{e^{-i\alpha}\,\Omega\,
                   +\, e^{i\alpha}\,\bar{\Omega}}{2}\,
          \frac{D_aZ}{Z}
     \right)
     \frac{\bar{D}_{\bar{b}}\bar{Z}}{\bar{Z}}           \\[1.2ex]
   && \hspace{6em}
    +\, \frac{1}{2i}\,
        g^{b\bar{a}}\,
        \left( -\,e^{i\alpha}\, D_{\bar{a}}\bar{\Omega} \,
         +\, \frac{e^{-i\alpha}\,\Omega\,
                      +\, e^{i\alpha}\,\bar{\Omega}}{2}\,
             \frac{D_{\bar{a}}\bar{Z}}{\bar{Z}}
        \right)
        \frac{D_b Z}{Z}                                  \\[1.2ex]
   && =\;
     \frac{1}{2i|Z|}
     \left(
      g^{a\bar{b}}\, D_a\Omega\, \bar{D}_{\bar{b}}\bar{Z}\,
      -\, g^{b\bar{a}}\, \bar{D}_{\bar{a}}\bar{\Omega}\, D_b Z
     \right)\,
     +\,
     \frac{\Real(e^{-i\alpha}\,\Omega)}{2i|Z|^2}
     \left(
      -\, g^{a\bar{b}}\, D_a Z\, \bar{D}_{\bar{b}}\bar{Z}\,
      +\, g^{b\bar{a}}\, \bar{D}_{\bar{a}}\bar{Z}\, D_b Z
     \right);
  \end{eqnarray*}
  \begin{eqnarray*}
   \mbox{the first summand}
    & = &
       \frac{1}{2|Z|}
       \left(
        -\,i\,g^{a\bar{b}}\, D_a\Omega\, \bar{D}_{\bar{b}}\bar{Z}\,
        +\,i\,g^{b\bar{a}}\, \bar{D}_{\bar{a}}\bar{\Omega}\, D_b Z
       \right)  \\
    & = & \frac{1}{2|Z|}
          \left(
           8\,\vol(Y)\cdot\Gamma^{\vee}\,
           -\, i\,\Omega\bar{Z}\,+\, i\, \bar{\Omega}Z
          \right)\;\;
      =\;\; \frac{4\,\vol(Y)\cdot\Gamma^{\vee}}{|Z|}\,
          +\, \Imaginary\left( e^{-i\alpha}\Omega \right)\,;
  \end{eqnarray*}
  \begin{eqnarray*}
   \mbox{the second summand}
    & = &
       -\,
       \frac{\Real(e^{-i\alpha}\,\Omega)}{|Z|^2}\,
       \Imaginary
       \left(
        g^{a\bar{b}}\, D_a Z\, \bar{D}_{\bar{b}}\bar{Z}
       \right)  \\
    & = &
       -\,
       \frac{\Real(e^{-i\alpha}\,\Omega)}{2|Z|^2}\,
       \left(
        8\,\vol(Y)\cdot
         \int_{Y_{\mbox{\tiny\boldmath $t$}}}
           \Gamma^{\vee}\wedge\Gamma^{\vee}\,
        +\, 2\,\Imaginary|Z|^2
       \right)\;\; =\;\; 0\,, \hspace{4em}
  \end{eqnarray*}
  where the Hodge decomposition identity applied to $\Gamma^{\vee}$
   is used.

 \bigskip

 In full notation,
  one thus obtains an ordinary differential equation (ODE)
  in variable $\mu$:
  $$
   \mu\frac{d}{d\mu}
     \Imaginary\left(e^{-i\alpha(\mu)}\Omega(\mu)\right)\,
            -\, \Imaginary\left( e^{-i\alpha(\mu)}\Omega(\mu) \right)\;
    =\; \frac{4\,\vol(Y)\cdot\Gamma^{\vee}}{|Z(\mu)|}\,.
  $$
 This can be solved
   by the Leibniz' rule and the separation-of-variable technique in ODE
  to give
  $$
   \frac{1}{\mu}\,\Imaginary\left(e^{-i\alpha(\mu)}\Omega(\mu)\right)\;
    =\; \left(-\,4\,\vol(Y)\cdot\tau_{\mu_0}(\mu)\right)\,\Gamma^{\vee}\,
        +\, \frac{1}{\mu_0}\,
             \Imaginary\left(e^{-i\alpha(\mu_0)}\Omega(\mu_0)\right)\,,
  $$
  on $U=U_j$, where
  $$
   \tau_{\mu_0}(\mu)\;
    =\; -\int_{\mu_0}^{\mu}
          \frac{d\mu^{\prime}}{|Z(\mu^{\prime})|\,{\mu^{\prime}}^2}\,.
  $$

 Since
  $$
    \tau_{\mu_0}(\mu_1)\,+\, \tau_{\mu_1}(\mu_2)\;
    =\; \tau_{\mu_0}(\mu_2)
  $$
   and
  the above local solution for
   $\Imaginary\left(e^{-i\alpha(\mu)}\Omega(\mu)\right)$
   from local computations is independent of all the local choices made,
 the expression glues under the transition between local charts
  and remains valid along and throughout an attractor flow trajectory
  in ${\cal M}$.

 This completes the proof.

\end{proof}

\begin{remark}
{\it $[$constant trajectory$]$.} {\rm
 A constant trajectory $\gamma$ of a $\Gamma$-attractor flow
  occurs at $\mbox{\boldmath $t$}\in{\cal M}$
   where $Z(\mbox{\boldmath $t$})\ne 0$ and
         $(D_bZ)(\mbox{\boldmath $t$})
          =(\bar{D}_{\bar{b}}\bar{Z})(\mbox{\boldmath $t$})=0$.
 In this case,
  $\Imaginary(e^{-i\alpha_{\Gamma}(\mu)}\Omega(\mu))$
  takes the constant value
  $-\,\frac{4\,\mbox{\small\it vol}\,(Y)\cdot\Gamma^{\vee}}
           {|Z(\mbox{\small\boldmath $t$})|}\,$
  as $\mu$ runs in ${\Bbb R}_{>0}\,$.
}\end{remark}

\bigskip

\subsection{Morse cobordisms of A-branes on Calabi-Yau 3-folds
            under a reverse split attractor flow
            at a wall of marginal stability.}

We derive first a topological smoothability criterion
 along a $\Gamma$-attractor flow
 for a special Lagrangian submanifold $L$ in the class $\Gamma$,
 with only transverse normal crossing singularities
 in a Calabi-Yau $3$-fold,
 by a fusion of [De3] of Denef and [Joy3: IV and V] of Joyce
and then
 use it as a tool to understand Morse cobordisms of A-branes
 under a (reverse) split attractor flow.
We'll assume in this subsection that
 all the attractor flow trajectories in question are nonconstant.

\bigskip

\begin{flushleft}
{\bf Topological smoothability criterion along attractor flow:
     Denef meeting Joyce.}
\end{flushleft}
The result of Denef [De3], recapped as Proposition~3.1.1 in Sec.~3.1,
 gives us a topologically necessary starting point to consider
 deformations of special Lagrangian submanifolds,
 with a phase lined up with that of central charge,
 on Calabi-Yau $3$-folds along an attractor flow:

\begin{lemma}
{\bf [vanishing of topological obstruction along attractor flow].}
 Let
  \begin{itemize}
   \item[$\cdot$]
    ${\cal F}$ be a domain in ${\cal M}$,

   \item[$\cdot$]
    $\{Y^s:=(Y, J^s, \omega^s, \Omega^s):s\in {\cal F}\}$
     be a family of smooth Calabi-Yau $3$-folds with
     the family of underlying complex manifolds $(Y,J^s)$
      specified by ${\cal F}$,
     the K\"{a}hler classes $[\omega^s]\in H^2(Y,{\Bbb R})$ fixed, and
     the normalization
      $\,\frac{1}{3!}\,\omega^s\wedge\omega^s\wedge\omega^s
       =\frac{i}{2^3}\,\Omega^s\wedge\overline{\Omega^s}\,$;

   \item[$\cdot$]
    $\Gamma\in H_3(Y;{\Bbb Z})$,
    $Z_{\Gamma}$ be the central charge function on ${\cal F}$
     associated to $\Gamma$, defined by
     $Z_{\Gamma}(s)=[\Omega^s]\cdot\Gamma=\int_{\Gamma}\Omega^s\,$;
    $\alpha_{\Gamma}=\Arg(Z_{\Gamma}/|Z_{\Gamma}|)$
     be the phase-angle function on ${\cal F}$
     associated to $Z_{\Gamma}$, defined modulo $2\pi\,$;

   \item[$\cdot$]
    $s_0\in {\cal F}$ with $Z_{\Gamma}(s_0)\ne 0$,
    $f:L\rightarrow Y^{s_0}$ be a special Lagrangian submanifold
     (possibly with singularities)
     with phase $e^{i\alpha_{\Gamma}(s_0)}$ in $Y^{s_0}$
     and with $f_{\ast}([L])=\Gamma\,$;

   \item[$\cdot$]
    $\gamma:I\rightarrow {\cal F}$,
      $I$ an interval in ${\Bbb R}_+$,
     be a trajectory of the attractor flow associated to $Z_{\Gamma}$,
     with $\gamma(\mu_0)=s_0\,$.
  \end{itemize}
 Then
  $$
   [\Imaginary( e^{-i\alpha_{\Gamma}(\mu)}\Omega^{s(\mu)})]
    \cdot f_{\ast}([L])\;=\;0\,,
   \hspace{2em}\mbox{for all $\;\mu\in I$}\,.
  $$
\end{lemma}

\begin{proof}
 This follows immediately from Proposition~3.1.1
 since, under the situation given,
 $$
  [\Gamma^{\vee}]\cdot\Gamma\;
   =\; 0\;
   =\; [\Imaginary( e^{-i\alpha_{\Gamma}(\mu_0)}\Omega^{s(\mu_0)})]
        \cdot f_{\ast}([L])\,.
 $$
\end{proof}

Continuing the situation of Lemma~3.2.1,
assume further that
 $$
  f\,=\, f_1\cup\,\cdots\,\cup f_q\;:\;
    L=L_1\cup\,\cdots\,\cup L_q\; \longrightarrow\; Y^{s_0}
 $$
  is an immersed special Lagrangian submanifolds
  (with phase $e^{i\alpha_{\Gamma}(s_0)}$)
  with only transverse intersections and with each $L_j$ connected.
 (In particular, all $L_j$'s are smooth.)
Here, $\cup$ is a disjoint union.

\begin{remark}
{\it $[$location at wall of marginal stability$]$.}
{\rm
 Let
  $\Gamma_j=f_{\ast}([L_j])={f_j}_{\ast}([L_j])$  and
  $Z_{\Gamma_j}$ be the associated complex central-charge function
   on ${\cal F}$,
   defined by $Z_{\Gamma_j}(s):= [\Omega^s]\cdot \Gamma_j
               =\int_{L_j}f_i^{\ast}\Omega^s$.
 Assume that $Z_{\Gamma_j}\ne 0$, for all $j$,  and
  let $e^{i\alpha_{\Gamma_j}}:= Z_{\Gamma_j}/|Z_{\Gamma_j}|$
    be the associated phase function on ${\cal F}$.
 Then since $f_j=f|_{L_j}$,
  $f_j^{\ast}\Imaginary(e^{-i\alpha_{\Gamma}(s_0)}\Omega^{s_0})=0$
    and
  $f_j^{\ast}\Real(e^{-i\alpha_{\Gamma}(s_0)}\Omega^{s_0})$
   coincides with the volume-form on $L_j$
   induced by the metric $f_j^{\ast}\omega^{s_0}$,
  for all $j$.
 It follows that
  $e^{-i\alpha_{\Gamma}(s_0)}\cdot Z_{\Gamma_j}(s_0)>0$, for all $j$,
  and, hence (in case that $q>1$),
  $s_0$ lies in the stratum of the wall $\Pi^{\MStiny}_{\cal F}$
   of marginal stability in ${\cal F}$ on which
    all the phase angles
     $\alpha_{\Gamma_j}(s_0)$, $j=1,\,\ldots\,,q$, and
     $\alpha_{\Gamma}(s_0)$
    are equal.
}\end{remark}

\begin{assumption}
{$[$connectivity$]$.} {\rm
 Without loss of generality and for the simplicity of presentation,
 one may
 {\it assume that $f(L)$ is connected}
 since otherwise one simply
  applies the discussion below component by component  and then
  adjusts the related constants that appear
   in the estimates/inequalities used
   so that all the corresponding constants become equal and
     work simultaneously for all components.
}\end{assumption}

Denote by $\{y_1,\,\cdots\,, y_n\}\subset Y^{s_0}$
 the set of points in $Y^{s_0}$ at which $f$ meets transversely.
Let $f^{-1}(y_j)=\{p_j^+,p_j^-\}$, where the $\pm$ is determined
 by the angle condition that
  the sum of the characteristic angles\footnote{Here
                                we use the convention in
                                [Joy3: V, Sec.~6.4, Example~6.11],
                                cf.~Appendix A.1.}
   from $f_{\ast}T_{p_j^+}L$ to $f_{\ast}T_{p_j^-}L$
   as a pair of special Lagrangian subspaces
   with phase $e^{i\alpha_{\Gamma}(s_0)}$ in
   $T_{y_j}Y^{s_0}:= (T_{y_j}Y,\,
     J^{s_0}|_{y_j},\,\omega^{s_0}|_{y_j},\,\Omega^{s_0}|_{y_j})$
   is equal to $\pi$.

\begin{definition}
{\bf [dual graph of decomposition].} {\rm
 The {\it dual graph} $\Xi_{f=f_1\cup\cdots\cup f_q}$ associated to
  the decomposition $f:L=L_1\cup\,\cdots\,\cup L_q\rightarrow Y^{s_0}$
  is a directed graph with
   \begin{itemize}
    \item[$\cdot$]
     a {\it vertex} $v_j$ for each $L_j$, $j=1,\,\ldots\,,q$, and

    \item[$\cdot$]
     a {\it directed edge} (i.e.\ {\it arrow}) $e_l$
       for each $y_l$, $l=1,\,\ldots\,,n$,
      with {\it initial end-point} (i.e.\ {\it tail}) $v_j$ and
      {\it terminal end-point} (i.e.\ {\it head}) $v_k$
      if $p_l^+\in L_j$ and $p_l^-\in L_k$.
   \end{itemize}
  The set of vertices (resp.\ edges) of $\Xi_{f=f_1\cup\cdots\cup f_q}$
   will be denoted by $\Xi_{f=f_1\cup\cdots\cup f_q}^{(0)}$
   (resp.\ $\Xi_{f=f_1\cup\cdots\cup f_q}^{(1)}$).
}\end{definition}

\begin{remark}
{$[$transverse special Lagrangian intersection in CY$^{\,3}$:
    direction $=$ sign$]$.} {\rm
 In ${\Bbb C}^3$
  with coordinates $(z^1, z^2, z^3)=(x^1+iy^1, x^2+iy^2, x^3+iy^3)$,
  the standard Hermitian metric
   $ds^2 = dz^1\otimes d\bar{z}^1
           + dz^2\otimes d\bar{z}^2 + dz^3\otimes d\bar{z}^3$  and
  holomorphic $3$-form $\Omega=dz^1\wedge dz^2\wedge dz^3$,
 a pair of special Lagrangian submanifold with respect to
  the standard K\"{a}hler form
   $\omega=-\frac{1}{2}\Imaginary(ds^2)
          =\frac{i}{2}(dz^1\wedge d\bar{z}^1
             + dz^2\wedge d\bar{z}^2 + dz^3\wedge d\bar{z}^3)$ and
  the calibration $\Real\Omega$ can be put into the canonical form
  $$
   \Pi^0\;=\; \{(x^1, x^2, x^3)\,:\, x^1,x^2,x^3\in {\Bbb R}\}\,,
    \hspace{1em}
   \Pi^{(\phi_1,\phi_2,\phi_3)}\;
    =\; \{(e^{i\phi_1}x^1, e^{i\phi_2} x^2, e^{i\phi_3}x^3)\,:\,
          x^1,x^2,x^3\in {\Bbb R}\}
  $$
  under the $\SU(3)$-action on $({\Bbb C}^3, ds^2)$,
  with $0< \phi_1,\,\phi_2,\,\phi_3< \pi$,
        $\phi_1+\phi_2+\phi_3=\pi$ or $2\,\pi$;
 (cf.~Appendix A.1.).
 One can check that,
  with the built-in orientation on ${\Bbb C}^3$ and
   the orientation on $\Pi^0$ and $\Pi^{(\phi_1,\phi_2,\phi_3)}$
    determined by $\Real\Omega$ taken into account,
  the (oriented) intersection
  $$
   \Pi^0\cdot\Pi^{(\phi_1,\phi_2,\phi_3)}\;\;
    (\,=\; -\, \Pi^{(\phi_1,\phi_2,\phi_3)}\cdot\Pi^{0}\,)\;\;
    =\;
    \left\{
     \begin{array}{rcl}
      +1 && \mbox{if\;\;$\phi_1+\phi_2+\phi_3=\pi$}\,, \\[.6ex]
      -1 && \mbox{if\;\;$\phi_1+\phi_2+\phi_3=2\pi$}\,.
     \end{array}
    \right.
  $$
 It follows that in Definition~3.2.4,
   $\Xi_{f=f_1\cup\,\cdots\,\cup f_q}$
    has an arrow $e_l$ from $v_j$ to $v_k$
   if and only if {\it the (oriented) intersection $L_j\cdot L_k$
    takes value $+1$ at $y_l$}; and, hence
 one may equivalently take the latter condition
  as the rule to assign arrows in $\Xi_{f=f_1\cup\,\cdots\,\cup f_q}$.

 In particular, $\Xi_{f=f_1\cup\cdots\cup f_q}$,
   which appears naturally in Joyce' setting as defined above
    via characteristic angles
   (cf.~[Joy3: V.~`graphical method' in the end of Sec.~9.2]),
  {\it is directly related to the quiver underlying
  the effective $d=1$, $N=4$ supersymmetric quantum mechanics
  associated to the D-brane configuration in $Y^{s_0}$
  specified in part by $f$},
   which is more naturally defined via the sign at the intersections;
 cf.~[De4: Sec.~3] and [D-M: Sec's.~4.1 and 4.2].
}\end{remark}

{For} the convenience of the application of Joyce' result
 to the case of attractor flows,
we recall
 [Joy3: V, Sec.~9.3, Theorem~9.8] (cf. Appendix A.1, Theorem~A.1.5)
 in three steps below,
 with notations and some settings adapted to our situation.

\begin{definition}
{\bf [Joyce' criterion: smoothability by deforming complex structure].}
{\rm
 {For} $A_1,\,\cdots\,, A_n>0$,
 let ${\cal G}_{(A_1,\,\cdots\,,\,A_n)}$
  be the set of $(s,t)\in {\cal F}\times (0,1)$ such that
  $$
   f^{\ast}[\Imaginary(e^{-i\alpha_{\Gamma}(s)}\Omega^s)]\cdot [L_j]\;
    =\;
    t^3 \cdot
     \raisebox{-.5em}{$\left(\rule{0em}{2em}\right.$}\!\!
      \sum_{\mbox{\scriptsize
            $\begin{array}{c}
              k\in\{1,\,\ldots\,, n\},\\
               p_k^+ \in L_j \end{array}$}}\!\!A_k \;
      -\;
      \sum_{\mbox{\scriptsize
            $\begin{array}{c}
              k^{\prime}\in\{1,\,\ldots\,, n\},\\
               p_{k^{\prime}}^- \in L_j \end{array}$}}\!\!A_{k^{\prime}}
     \raisebox{-.5em}{$\left.\rule{0em}{2em}\right)$},
  $$
  for $\;j=1,\,\ldots\,, q\,.$
  When ${\cal G}_{(A_1,\,\cdots\,,\,A_n)}\ne\emptyset$ and
   contains a neighborhood of $(s_0,0)$ in ${\cal F}\times (0,1)$,
  we say that the tuple
   $(A_1,\,\cdots\,,A_n)$ {\it satisfies Joyce' criterion}
    ({\it of smoothing the special Lagrangian submanifold-with-singularity
     $f(L)$ with a phase via deforming complex structures}).
}\end{definition}

\begin{definition}
{\bf [admissible subset].} {\rm
 {For} $(A_1,\,\cdots\,,A_n)$ that satisfies the Joyce' criterion above,
  $\epsilon\in(0,1)$, $\kappa>1$, and $C>0$,
 define
 $$
  {\cal G}_{(A_1,\,\cdots\,,\,A_n)}^{\,\epsilon,\kappa, C}\;
   =\; \{\, (s,t)\in {\cal G}_{(A_1,\,\cdots\,,\,A_n)}\:
                 :\: t\in (0,\epsilon]\,,\;
                     |s-s_0|\,\le\, C\,t^{\kappa+3/2}\, \}
 $$
 and call it an {\it admissible subset} in ${\cal F}\times (0,1)$
  ({\it for smoothing $L$}).
}\end{definition}

\noindent
Here,
 a local coordinate chart is introduced in a neighborhood
  of $s_0\in {\cal F}$  and
 $|s-s_0|$ is defined in terms of this chart
  as a subset in some ${\Bbb R}^{m^{\prime}}$.

\begin{theorem}
{\bf [desingularization in a family of Calabi-Yau's].}
{\rm ([Joy3: V.~Sec.~9.3, Theorem~9.8];
      cf.~Appendix A.~1: Theorem~A.1.5 and Observation~A.1.6.)}
 %
 Recall
  the family
   $\{Y^s:=(Y, J^s, \omega^s, \Omega^s):s\in {\cal F}\}$
   of smooth Calabi-Yau $3$-folds  and
  the special Lagrangian immersion
   $$
    f\,=\, f_1\cup\,\cdots\,\cup f_q\;:\;
     L=L_1\cup\,\cdots\,\cup L_q\; \longrightarrow\; Y^{s_0}
   $$
   with phase $e^{i\alpha_{\Gamma}(s_0)}$.
 Continuing the situation under study,
 let $N$ be the (oriented) connected sum of
  (the disjoint union) $L_1\cup\,\cdots\,\cup L_q$ with itself
  at the pairs of points $(p_k^+, p_k^-)$ for $k=1,\,\ldots\,, n\,$.
 Note that under Assumption~3.2.3, $N$ is connected.
 Note also that since $[\omega^s]\in H^2(Y;{\Bbb R})$ is fixed,
  treating $f$ as a map to $Y$ (and hence to all $Y^s$),
  one has
  $f^{\ast}[\omega^s]=f^{\ast}[\omega^{s_0}]=0$ in $H^2(L;{\Bbb R})$
   for all $s\in {\cal F}$.
 Suppose that the $n$-tuple
  $(A_1,\,\cdots\,,A_n)$ satisfies Joyce' criterion in Definition~3.2.6,
 Then,
  there exist
   constants
    $\epsilon\in (0,1)$, $\kappa>1$, and $C>0$, and
   a smooth family of maps
    $$
     \left\{\,
      f^{s,t}\,:\,L^{s,t}:=\tilde{N}^{s,t}\rightarrow Y^s\;
       \left.\rule{0em}{.9em}\right|\; (s,t)\,\in\,
           {\cal G}_{(A_1,\,\cdots\,,\,A_n)}^{\,\epsilon,\kappa, C}
            \,\right\}
    $$
   such that
   \begin{itemize}
    \item[$\cdot$]
     $\tilde{N}^{s,t}$ is a compact smooth manifold diffeomorphic to $N$,
     constructed by gluing a Lawlor neck
     $C^{\,-1/(\kappa+3/2)}\,t\cdot L^{\pm, A_k}$
     into $f(L)$ at $y_k$ for $k=1,\,\ldots\,, n$;

    \item[$\cdot$]
     $f^{s,t}:L^{s,t}\rightarrow Y^s$ is an embedded
       special Lagrangian submanifold with phase
       $e^{i\alpha_{\Gamma}(s)}$;

    \item[$\cdot$]
     in the sense of currents,
     $f^{s,t}\rightarrow f$ as $(s,t)\rightarrow (s_0,0)$.
   \end{itemize}
 Furthermore, $\kappa>1$ can be chosen to be arbitrarily close to $1$.
\end{theorem}

\bigskip

Let us now specialize to what happens along the trajectory
 $\gamma:I\rightarrow {\cal F}$ of the $\Gamma$-attractor flow
 with $\gamma(\mu_0)=s_0$.
{For} simplicity of notations, we will denote
 $e^{i\alpha_{\Gamma}(\gamma(\mu))}$ by $e^{i\alpha_{\Gamma}(\mu)}$ and
 $s=\gamma(\mu)$ by $s=s(\mu)$.
Recall that
 $\Imaginary\left(e^{-i\alpha_{\Gamma}(\mu)}\Omega^{s(\mu)}\right)
   = (-\,4\,\vol(Y^{\!s_0})\cdot\mu\,\tau(\mu))\,\Gamma^{\vee}\,
       +\, \frac{\mu}{\mu_0}\,
             \Imaginary\left(
                 e^{-i\alpha_{\Gamma}(\mu_0)}\Omega^{s(\mu_0)}
                       \right)$
 along $\gamma$,
 where
  $\tau(\mu)
    = -\int_{\mu_0}^{\mu}
         \frac{d\mu^{\prime}}
              {|Z_{\Gamma}(s(\mu^{\prime}))|\,{\mu^{\prime}}^2}$,
  from Proposition~3.1.1.
It follows that
 \begin{eqnarray*}
  \left(
   f^{\ast}
    [\Imaginary(e^{-i\alpha_{\Gamma}(\mu)}\Omega^{s(\mu)})]\cdot[L_j]
  \right)_{j=1}^q
   & =
   & (-\,4\,\vol(Y^{\!s_0})\cdot\mu\,\tau(\mu))\cdot
      \left(\rule{0em}{1em}
        \langle\,\Gamma\,,\,f_{\ast}[L_j]\,\rangle
      \right)_{j=1}^q                                   \\[1.2ex]
  & \in
   & (-\,4\,\vol(Y^{\!s_0})\cdot\mu\,\tau(\mu))\cdot\,{\Bbb Z}^q
 \end{eqnarray*}
since
 $f^{\ast}[\Imaginary(e^{-i\alpha_{\Gamma}(\mu_0)}\Omega^{s_0})]
   \cdot[L_j]=0\,$,
 for $j=1,\,\ldots\,,q\,$.
Here,
 $\langle\,\cdot\,,\,\cdot\,\rangle:
    H_3(Y;{\Bbb Z})\times H_3(Y;{\Bbb Z})\rightarrow {\Bbb Z}$
  is the (symplectic) intersection form on $Y$.
Note that
 $$
  -\,4\,\vol(Y^{\!s_0})\cdot\mu\,\tau(\mu)\;
   \left\{
    \begin{array}{lccc}
     >\;0 && \mbox{for} & \mu\;>\; \mu_0\,, \\[.2ex]
     =\;0 && \mbox{for} & \mu\;=\; \mu_0\,, \\[.2ex]
     <\;0 && \mbox{for} & \mu\;<\; \mu_0\,.
    \end{array}
   \right.
 $$
A comparison of this with Joyce' criterion in Definition~3.2.6
 leads one immediately to the following definition:

\begin{definition}
{\bf [topological criterion of smoothing].}
{\rm
 {For} $A_1,\,\cdots\,, A_n>0$,
 we say that the tuple $(A_1,\,\cdots\,,A_n)$
  {\it satisfies the topological criterion of smoothing}
   ({\it the special Lagrangian submanifold-with-singularity
         $f(L)$ with a phase via deforming complex structures})
 if it satisfies
  $$
   \left(\rule{0em}{1em}
     \langle\,\Gamma\,,\,f_{\ast}[L_j]\,\rangle
   \right)_{j=1}^q\;
   =\;
     \raisebox{-.5em}{$\left(\rule{0em}{2em}\right.$}\!\!
      \sum_{\mbox{\scriptsize
            $\begin{array}{c}
              k\in\{1,\,\ldots\,, n\},\\
               p_k^+ \in L_j \end{array}$}}\!\!A_k \;
      -\;
      \sum_{\mbox{\scriptsize
            $\begin{array}{c}
              k^{\prime}\in\{1,\,\ldots\,, n\},\\
               p_{k^{\prime}}^- \in L_j \end{array}$}}\!\!A_{k^{\prime}}
     \raisebox{-.5em}{$\left.\rule{0em}{2em}\right)_{j=1}^q$}\,.
  $$
}\end{definition}

This means that $(A_1,\,\cdots\,,A_n)$ is a positive solution
 to a system of inhomogeneous linear equations
 whose homogeneous/degree-$1$ part comes solely from the dual graph
  $\Xi_{f=f_1\cup\,\cdots\,\cup f_q}$  and
 whose constant/degree-$0$ part is given by
  $(\,\langle\,\Gamma\,,\,f_{\ast}[L_j]\,\rangle\,)_{j=1}^q\,$.
The following lemma is immediate:

\begin{lemma}
{\bf [topological criterion {\boldmath $\Rightarrow$} Joyce' criterion].}
 If $(A_1,\,\cdots\,,A_n)$ with positive entries
  satisfies the topological criterion of smoothing,
 then it satisfies Joyce' criterion.
\end{lemma}

With the above preparations and
as a consequence of
  Proposition~3.1.1 from Denef  and
  Theorem~A.1.5 from Joyce,
 one can now show that:

\begin{proposition}
{\bf [smoothing along attractor flow].}
 Continuing the situation under study,
 let $(A_1,\,\cdots\,,A_n)$, with positive entries,
  satisfy the topological criterion of smoothing.
 Then,
  there exist
   constants
    $\delta,\,\epsilon \in (0,1)$, $\kappa>1$, and $C>0$, and
   a smooth family of maps
    $$
     \left\{\,
      f^{s,t}\,:\,L^{s,t}:=\tilde{N}^{s,t}\rightarrow Y^s\;
       \left.\rule{0em}{.9em}\right|\; (s,t)\,\in\,
           {\cal G}_{(A_1,\,\cdots\,,\,A_n)}^{\,\epsilon,\kappa, C}
            \,\right\}
    $$
   as in Theorem~A.1.5
   such that
    \begin{itemize}
     \item[$\cdot$]
      the restriction
       $\gamma: (\mu_0\,,\,\mu_0+\delta)\subset I \rightarrow {\cal F}$
       of the $\Gamma$-attractor flow trajectory
       $\gamma:I\rightarrow {\cal F}$ lifts to
      $\;\tilde{\gamma}: (\mu_0\,,\,\mu_0+\delta)\rightarrow
         {\cal G}_{(A_1,\,\cdots\,,\,A_n)}^{\,\epsilon,\kappa, C}\,
         \subset\, {\cal F}\times (0,1)$
       with $\pr_1\circ\tilde{\gamma}
             =\gamma|_{(\mu_0\,,\,\mu_0+\delta)}\,$,
      where $\,\pr_1:{\cal F}\times (0,1)\rightarrow {\cal F}\,$
       is the projection map.
    \end{itemize}
   In other words, Joyce' construction smoothes $f(L)\subset Y^{\!s_0}$
    into $Y^{s(\mu)}$
     along the $\Gamma$-attractor flow trajectory $\gamma$
     with ${\mu\,\raisebox{-.6ex}{$\stackrel{>}{\sim}$}\:\mu_0}$.
\end{proposition}

\begin{proof}
 Since
  $(A_1,\,\cdots\,,A_n)$
   satisfies the topological criterion of smoothing and
  the cubic-root function $(\,\cdot\,)^{1/3}$
   is defined over ${\Bbb R}$,
 the system
  \begin{eqnarray*}
   \left(
    f^{\ast}[\Imaginary(e^{-i\alpha_{\Gamma}(s(\mu))}\Omega^s)]
    \cdot [L_j]
   \right)_{j=1}^q
   & =
   & (-\,4\,\vol(Y^{\!s_0})\cdot\mu\,\tau(\mu))\cdot
      \left(\rule{0em}{1em}
        \langle\,\Gamma\,,\,f_{\ast}[L_j]\,\rangle
      \right)_{j=1}^q \\[1.2ex]
   & =
   & t^3 \cdot
     \raisebox{-.5em}{$\left(\rule{0em}{2em}\right.$}\!\!
      \sum_{\mbox{\scriptsize
            $\begin{array}{c}
              k\in\{1,\,\ldots\,, n\},\\
               p_k^+ \in L_j \end{array}$}}\!\!A_k \;
      -\;
      \sum_{\mbox{\scriptsize
            $\begin{array}{c}
              k^{\prime}\in\{1,\,\ldots\,, n\},\\
               p_{k^{\prime}}^- \in L_j \end{array}$}}\!\!A_{k^{\prime}}
     \raisebox{-.5em}{$\left.\rule{0em}{2em}\right)_{j=1}^q$}
  \end{eqnarray*}
  is always solvable along $\gamma$ for $t\in {\Bbb R}$.
 Furthermore,
 since
  \begin{eqnarray*}
   -\,4\,\vol(Y^{\!s_0})\cdot\mu_0\,\tau(\mu_0)
    & =  & 0 \hspace{7em}\mbox{and}           \\[.6ex]
   \left.\frac{d}{d\mu}\right|_{\mu=\mu_0}
    \left(
     -\,4\,\vol(Y^{\!s_0})\cdot\mu\,\tau(\mu)
    \right)
   & =
   & \left.\frac{d}{d\mu}\right|_{\mu=\mu_0}
        \left(
         4\,\vol(Y^{\!s_0})\cdot\mu\,
         \int_{\mu_0}^{\mu}
          \frac{d\mu^{\prime}}
              {|Z_{\Gamma}(\mu^{\prime})|\,{\mu^{\prime}}^2}
      \right)                                 \\[1.2ex]
   & =
   & \frac{4\,\vol(Y^{\!s_0})}{|Z_{\Gamma}(\mu_0)|\mu_0}\;\;
     >\;\; 0\,,
  \end{eqnarray*}
 the requirement that $t>0$ imposes
  the condition that the system is solvable only
   on the ($\mu>\mu_0$)-side of the flow $\gamma$,
  in which case
   $$
    \mu\;=\;\mu(t) \hspace{3em}\mbox{with}\hspace{1em}
    |\mu-\mu_0|=O(t^3) \hspace{1em}\mbox{and inverse}\hspace{1em}
    t\;=\;t(\mu)
   $$
   for $\,\mu\in (\mu_0,\mu_0+\delta)\subset I\,$
   for some $\,\delta>0\,$.

 The further admissibility condition
  $|s-s_0|\,\le\, C^{\prime}\,t^{\kappa^{\prime}+3/2}$
   with $\kappa^{\prime}>1$ and  $C^{\prime}>0$ in Definition~3.2.7
  translates in the current situation
  to an admissibility condition of the form
  $|\mu-\mu_0|\le C^{\prime\prime}\,t^{\kappa^{\prime\prime}+3/2}$
   with $\kappa^{\prime\prime}>1$ and $C^{\prime\prime}>0$.
 Recall now Observation~A.1.6 that in Theorem~A.1.5,
    $\kappa>1$ can be chosen to be arbitrarily close to $1$ and
    $C>0$ can be adjusted.
 It follows that
 for $\epsilon>0$, $\kappa>1$, and $C>0$ in Theorem~A.1.5
  with $3>\kappa+3/2$, and
  a re-choice of $\delta>0$ to be even smaller if necessary,
 one has
  $$
   (\mu(t), t)\,=\,(\mu,t(\mu))\;\in\;
   {\cal G}_{(A_1,\,\cdots\,,\,A_n)}^{\,\epsilon,\kappa, C}
   \hspace{2em}\mbox{for}\;\mu\in (\mu_0, \mu_0+\delta)\,.
  $$
 The map
  $\;\tilde{\gamma}: (\mu_0\,,\,\mu_0+\delta)\rightarrow
         {\cal G}_{(A_1,\,\cdots\,,\,A_n)}^{\,\epsilon,\kappa, C}\;$
   defined by $\mu \mapsto (\mu,t(\mu))$
  gives then a lifting of the restriction
   $\gamma|_{(\mu_0\,,\,\mu_0+\delta)}$ of $\gamma$.
 This concludes the proof.

\end{proof}

\begin{remark}
{$[$decay of BPS states$]$.}
{\rm
 Recall that
  $A_k>0$, $k=1,\,\ldots\,,n\,$, are the moduli of Lawlor necks
   with fixed asymptotic special Lagrangian planes
  and $t>0$ is the scaling factor,
  both involved in the gluing-and-rectifying construction.
 Their positivity has thus a geometric meaning and, hence,
  the construction of smoothing $f(L)$ is ``directional".
 As special Lagrangian cycles support
   supersymmetric D-brane configurations of A-type,
  reverse of this direction,
     i.e.\ $\mu$ crossing the wall $\Pi^{\MStiny}_{\cal F}$
     from $\mu_0^+\ge\mu_0$ to $\mu_0^-\le\mu_0$,
    in which the underlying special Lagrangian cycles
    morph from a smooth one to one with several components
   indicates then a decay of a BPS state.
 Cf.~Remark~3.2.2.
}\end{remark}

\bigskip
\newpage

\begin{flushleft}
{\bf Morse cobordisms of A-branes under a (reverse) split attractor
     flow:\\ Denef-Joyce meeting Polchinski-Grothendieck.}
\end{flushleft}
Let
 $$
  f\,=\, f_1\cup\,\cdots\,\cup f_q\;:\;
   L=L_1\cup\,\cdots\,\cup L_q\; \longrightarrow\; Y
 $$
 be a special Lagrangian immersion with phase $e^{i\alpha_{\Gamma}}$
 in a Calabi-Yau $3$-fold $Y$,
 where $\Gamma=f_{\ast}[L]\in H_3(Y;{\Bbb Z})$,
 with only isolated transverse intersections
 as in the situation of the previous theme.
Continuing the notations there,
assume that $f$ satisfies the topological criterion of smoothing
  in Definition~3.2.9  and
 let
  $$
   f^{\mu}\; :=\;
    f^{s(\mu), t(\mu)}\; :\; L^{\mu}:=L^{s(\mu),t(\mu)}\;
      \longrightarrow\;  Y^{\mu}:=Y^{s(\mu)}\,,\hspace{1em}
   \mu\;\in\; (\mu_0\,,\, \mu_0+\delta)\; \subset\; {\Bbb R}_{>0}
  $$
  be a smoothing of
   $f=: f^{\mu_0} = f^{s(\mu_0), t(\mu_0)}= f^{s_0,0}$
  along the $\Gamma$-attractor flow $\gamma=\gamma(\mu)$
   with $\gamma(\mu_0)=s_0$,
  following Joyce' construction.
Assume that $\delta<\mu_0\,$.
Let
 $$
  \pi_{(\mu_0-\delta\,,\,\mu_0+\delta)}\; :\;
    X_{(\mu_0-\delta\,,\,\mu_0+\delta)}\; \longrightarrow\;
  (\mu_0-\delta\,,\,\mu_0+\delta)\;\subset\; {\Bbb R}_{>0}
 $$
 be a Morse family, as constructed in the same way
  as in Example~2.2.1(c),
 such that
  $$
   \pi_{(\mu_0-\delta\,,\,\mu_0+\delta)}^{-1}(\mu)\;
    =:\; X_{\mu}\;
    =\;
    \left\{
     \begin{array}{cccl}
      \amalg_{i=1}^q\,L_i
       && \mbox{for}
        & \mu\;\in\; (\mu_0-\delta\,,\,\mu_0)\,, \\[.6ex]
      f(L) && \mbox{for}   & \mu\; =\;\mu_0\,,   \\[.6ex]
      L^{\mu}\simeq N && \mbox{for}
         & \mu\;\in\; (\mu_0\,,\,\mu_0+\delta)\,.
     \end{array}
    \right.
  $$
Take a smaller $\delta>0$ if necessary;
let
 $\Gamma_i:= {f_i}_{\ast}[L_i]\in H_3(Y;{\Bbb Z})$  and
 $\gamma_i:(\mu_0-\delta\,,\,\mu_0]\rightarrow {\cal F}$
  be the $\Gamma_i$-attractor flow with $\gamma_i(\mu_0)=s_0$.
Let
 $$
  p_{(\mu_0-\delta\,,\,\mu_0+\delta)}\;:\;
  Y_{(\mu_0-\delta\,,\,\mu_0+\delta)}\; \longrightarrow\;
    (\mu_0-\delta\,,\,\mu_0+\delta)\,,
 $$
 where
 $$
  p_{(\mu_0-\delta\,,\,\mu_0+\delta)}^{-1}(\mu)\;
   =:\; Y_{\mu}\;
   =\;
   \left\{
    \begin{array}{cccl}
     \amalg_{i=1}^q\,Y_i^{\mu}
      && \mbox{for}
       & \mu\;\in\; (\mu_0-\delta\,,\,\mu_0)\,, \\[.6ex]
     Y=Y^{\mu_0}
      && \mbox{for}   & \mu\; =\;\mu_0\,,       \\[.6ex]
     Y^{\mu}
      && \mbox{for}   & \mu\;\in\; (\mu_0\,,\,\mu_0+\delta)\,,
    \end{array}
   \right.
 $$
 be the universal Calabi-Yau $3$-fold
 over the union of attractor flow trajectories\footnote{{\it
                              Mathematical Convention vs.\
                              String-Theoretical Convention.}
                              This union is called a
                               {\it split attractor flow trajectory}
                               in ${\cal F}$.
                              In mathematical convention,
                               flow follows the direction of
                               increasing $\mu$
                               (i.e.\ the direction of the vector field
                                on ${\cal F}$ that defines the flow)
                              while in string-theoretical convention
                               for the attractor, the flow follows
                               the direction of decreasing $\mu$.
                              In this note, we follow the mathematical
                                convention so far
                               as that is more natural for the purpose
                               to address desingularization,
                                instead of bend-and-break,
                               of a special Lagrangian submanifold
                               with singularities.
                              However, the term `{\it split} attractor flow'
                               in stringy literature,
                               as follows the stringy convention,
                               is fixed to correspond to the direction
                                $\Gamma\rightarrow
                                  \Gamma_1+\,\cdots\,+\Gamma_q$.
                              {For} that reason, we have to call
                               the natural direction in our question,
                                $\Gamma_1+\,\cdots\,+\Gamma_q
                                 \rightarrow\Gamma$,
                               a {\it reverse} split attractor flow.}
  $$
   \left(\bigcup_{i=1}^q
        \gamma_i\left((\mu_0-\delta\,,\,\mu_0]\right)\right)\,
    \bigcup\,\gamma\left([\mu_0\,,\,\mu_0+\delta)\right)\;
   \subset\; {\cal F}\,.
  $$
Take even a smaller $\delta>0$ if necessary,
recall Theorem~A.1.2 in Appendix~A.~1,  and
let
 $$
  f_i^{\mu}\; :\; L_i^{\mu}\; \longrightarrow\; Y_i^{\mu}\,,
  \hspace{1em} \mu\;\in\;(\mu_0-\delta\,,\,\mu_0)
 $$
 be a smooth family of immersed special Lagrangian submanifolds,
 with $L_i^{\mu}=L_i$, from Theorem~A.1.2.
Denote the inclusion $f(L)\hookrightarrow Y$ by $\underline{f}$.
Then,
one can extend the family $f^{\mu}$, $\mu\in [\mu_0,\mu_0+\delta)$,
 to a family of immersed special Lagrangian submanifolds
 over $(\mu_0-\delta\,,\,\mu_0+\delta)\,$:
 $$
  \xymatrix{
   X_{(\mu_0-\delta\,,\,\mu_0+\delta)}
    \ar[rr]^-{f_{(\mu_0-\delta\,,\,\mu_0+\delta)}}
    \ar[dr]_-{\pi_{(\mu_0-\delta\,,\,\mu_0+\delta)}\hspace{2em}}
    && Y_{(\mu_0-\delta\,,\,\mu_0+\delta)}
       \ar[dl]^-{\hspace{2em}p_{(\mu_0-\delta\,,\,\mu_0+\delta)}} \\
   & (\mu_0-\delta\,,\,\mu_0+\delta)
  }
 $$
 with
 $$
  f_{\mu}:= f_{(\mu_0-\delta\,,\,\mu_0+\delta)}|
                       _{\mu\in (\mu_0-\delta\,,\,\mu_0+\delta)}\,
  :\; X_{\mu}\; \longrightarrow\; Y_{\mu}\,,
 $$
 such that
 $$
  f_{\mu}\;=\;
   \left\{
    \begin{array}{cccl}
     \amalg_{i=1}^q f_i^{\mu}
      && \mbox{for} & \mu\;\in\;(\mu_0-\delta\,,\,\mu_0)\,,  \\[.6ex]
     \underline{f}  && \mbox{for} & \mu\;=\;\mu_0\,,         \\[.6ex]
     f^{\mu}        && \mbox{for} & \mu\;\in\;(\mu_0\,,\,\mu_0+\delta)\,.
    \end{array}
   \right.
 $$
Consider the following data

 $$
  \xymatrix{
   (\hat{\cal E}_{(\mu_0-\delta\,,\,\mu_0+\delta)},
    \hat{\nabla}^{(\mu_0-\delta\,,\,\mu_0+\delta)})
   \ar@{.>}[dd]                                       \\ \\
   \hat{X}_{(\mu_0-\delta\,,\,\mu_0+\delta)}
    \ar[dd]^-{(\hat{c}_{(\mu_0-\delta\,,\,\mu_0+\delta)},
               \hat{f}_{(\mu_0-\delta\,,\,\mu_0+\delta)})}
    \ar@/^2ex/[rrrrdd]^-{\hspace{2em}
                         \hat{f}_{(\mu_0-\delta\,,\,\mu_0+\delta)}}
    \ar@/_24ex/[ddddd]^-{\hspace{1em}\mbox{\raisebox{-4em}{\scriptsize
                         $\hat{c}_{(\mu_0-\delta\,,\,\mu_0+\delta)}$ }} }
                                                      \\ \\
   X_{(\mu_0-\delta\,,\,\mu_0+\delta)}
      \times_{(\mu_0-\delta\,,\,\mu_0+\delta)}
      Y_{(\mu_0-\delta\,,\,\mu_0+\delta)}
      \ar[rrrr]_-{pr_2}  \ar[ddd]^-{pr_1}
     &&&& Y_{(\mu_0-\delta\,,\,\mu_0+\delta)}\;,      \\ \\ \\
   X_{(\mu_0-\delta\,,\,\mu_0+\delta)}
  }
 $$

 \bigskip
 \noindent
 similar to Example~2.1.16,
 where
  \begin{itemize}
   \item[$\cdot$]
    all the maps in the commutative diagram are
     maps over $(\mu_0-\delta\,,\,\mu_0+\delta)\,$;

   \item[$\cdot$]
    $\hat{c}_{(\mu_0-\delta\,,\,\mu_0+\delta)}:
           \hat{X}_{(\mu_0-\delta\,,\,\mu_0+\delta)}
           \rightarrow X_{(\mu_0-\delta\,,\,\mu_0+\delta)}$
    is a covering map of finite degree $\hat{d}\,$;

   \item[$\cdot$]
    $\hat{f}_{(\mu_0-\delta\,,\,\mu_0+\delta)}:
           \hat{X}_{(\mu_0-\delta\,,\,\mu_0+\delta)}
           \rightarrow Y_{(\mu_0-\delta\,,\,\mu_0+\delta)}$
    is the composition
    $f_{(\mu_0-\delta\,,\,\mu_0+\delta)}\circ \hat{c}\,$;

   \item[$\cdot$]
    $\pr_1$ and $\pr_2$ are the built-in projection maps
     from the fibered product;

   \item[$\cdot$]
    $(\hat{\cal E}_{(\mu_0-\delta\,,\,\mu_0+\delta)},
      \hat{\nabla}^{(\mu_0-\delta\,,\,\mu_0+\delta)})$
     is a locally free
      ${\cal O}_{\hat{X}_{(\mu_0-\delta\,,\,\mu_0+\delta)},
                                                  {\Bbb C}}$-module
      of finite rank $\hat{r}$\\
      with a flat $U(\hat{r})$-connection.
  \end{itemize}
 Let
  ${\cal E}_{(\mu_0-\delta\,,\,\mu_0+\delta)}
   := \hat{c}_{(\mu_0-\delta\,,\,\mu_0+\delta)\,\ast}\,
      \hat{\cal E}_{(\mu_0-\delta\,,\,\mu_0+\delta)}$,
  which is a locally free
   ${\cal O}_{X_{(\mu_0-\delta\,,\,\mu_0+\delta)},{\Bbb C}}$-module
   of rank $r=\hat{d}\hat{r}$.
 Then, following the construction of Example~2.1.16,
 one obtains a Morse cobordism family of morphisms
 $$
  \varphi_{(\mu_0-\delta\,,\,\mu_0+\delta)}\;
   :\; (X_{(\mu_0-\delta\,,\,\mu_0+\delta)}^{A\!z},
        {\cal E}_{(\mu_0-\delta\,,\,\mu_0+\delta)})
   \; \longrightarrow\; Y_{(\mu_0-\delta\,,\,\mu_0+\delta)}
 $$
 over $(\mu_0-\delta\,,\,\mu_0+\delta)$,
 with a $U(r)$ minimally flat connection-with-singularity
  $\nabla^{(\mu_0-\delta\,,\,\mu_0+\delta)}$ on the surrogate
  $(X_{\varphi_{(\mu_0-\delta\,,\,\mu_0+\delta)}},
    {\cal E}_{\varphi_{(\mu_0-\delta\,,\,\mu_0+\delta)}})$
  associated to $\varphi_{(\mu_0-\delta\,,\,\mu_0+\delta)}$.
 This gives an example of Morse cobordisms of A-branes
  under the (reverse) split attractor flow
  in the sense of Morse family
   of morphisms from Azumaya manifolds with additional data.
 Once having $\varphi_{(\mu_0-\delta\,,\,\mu_0+\delta)}$,
 one can deform
  $\varphi_{(\mu_0-\delta\,,\,\mu_0+\delta)}$ as in Example~2.2.1
  to obtain more Morse cobordisms of A-branes
  under the (reverse) split attractor flow.

\begin{remark}
{$[\,$splitting(/gluing) flat connection
      under (reverse) split attractor flow$\,]$.}\footnote{C.-H.L.\
                              would like to thank Frederik Denef
                              for several discussions on this
                              in spring, 2010.
                             Readers are referred to
                              [Sta], [He], and [Ja]
                              for some of the terminologies
                               used in this theme and
                              for basics on relations between
                               fundamental groups and $3$-manifolds.}
{\rm
 Given
  $$
   \hat{\pi}_{(\mu_0-\delta\,,\,\mu_0+\delta)}\;:\;
    \hat{X}_{(\mu_0-\delta\,,\,\mu_0+\delta)}\;
     \longrightarrow\; (\mu_0-\delta\,,\,\mu_0+\delta)
  $$
  via the composition
  $\pi_{(\mu_0-\delta\,,\,\mu_0+\delta)}
   \circ \hat{c}_{(\mu_0-\delta\,,\,\mu_0+\delta)}$
  in the above discussion,
 let
  $\hat{X}_{\mu} :=
   \hat{\pi}_{(\mu_0-\delta\,,\,\mu_0+\delta)}^{-1}(\mu)$
  be the fiber over $\mu\in (\mu_0-\delta\,,\,\mu_0+\delta)$.
 Then, these fibers are of three homeomorphism types:
  $$
   \hat{X}_{\mu}\; \simeq\;
    \left\{
     \begin{array}{cccl}
      \hat{N}                  && :
       & \mbox{a covering of $N$, the smoothing of $f(L)$}    \\
      &&& \mbox{by a connected sum of $\amalg_{i=1}^qL_i$
                in Theorem~A.1.5}\,, \\[.6ex]
      \hat{f(L)}               && :
       & \mbox{a covering of $f(L)$}\,, \\[.6ex]
      \amalg_{i=1}^q\hat{L}_i  &&:
       & \mbox{$\hat{L}_i$ a covering of $L_i$}\,.
     \end{array}
    \right.
  $$
 It follows that
  \begin{itemize}
   \item[$\cdot$]
    $\hat{f(L)}$ (resp.\ $\hat{N}$) is a (self-)bouquet
    (resp.\ (self-)connected sum) of $\amalg_{i=1}^q\hat{L}_i\,$,

   \item[$\cdot$]
    $\hat{f(L)}$ is obtained from $\hat{N}$ by pinching
     a finite collection of two-collared embedded $S^2$'s\\
     in $\hat{N}\,$.
  \end{itemize}
 Let
  \begin{itemize}
   \item[$\cdot$]
    $\Gamma_{\hat{f(L)}}$ be the dual graph of
    the manifold-with-Morse-type-singularity $\hat{f(L)}$,
    with
     one vertex $v_j$ for each connected component,
      denoted by $\hat{L}_{v_j}$,
      of $\amalg_{i=1}^q\hat{L}_i$  and
     one edge $e_{jj^{\prime}}$ connecting $v_j$ and $v_{j^{\prime}}$
      for each intersection of the irreducible components
      in $\hat{f(L)}$ associated to $L_{v_j}$ and $L_{v_{j^{\prime}}}$
      respectively;

   \item[$\cdot$]
    $\Gamma_{\hat{f(L)}}^{(0)}$
    be the set of vertices of $\Gamma_{\hat{f(L)}}\,$;  and

   \item[$\cdot$]
    $\Gamma_{\hat{f(L)}}^{(1)}$
    be the set of edges of $\Gamma_{\hat{f(L)}}\,$.
  \end{itemize}
 Then, with an implicit base point on each connected component chosen,
  $$
   \pi_1(\hat{N})\;=\;\pi_1(\hat{f(L)})
  $$
  (i.e.\ there is a canonical isomorphism between the two
   under the built-in pinching-map $\hat{N}\rightarrow \hat{f(L)}$
   that takes the base-point of the former to that of the latter)
  fits into an exact sequence of groups
  $$
   \begin{array}{ccccccccc}
    1  & \longrightarrow
       & \pi_1\left(\rule{0em}{1em}\right.
           \bigvee_{v_j\in\Gamma_{\hat{f(L)}}^{(0)}}\hat{L}_{v_j}
             \left.\rule{0em}{1.2em}\right)
       & \longrightarrow  & \pi_1(\hat{f(L)})
       & \longrightarrow  & \pi_1(\Gamma_{\hat{f(L)}})
       & \longrightarrow  & 1\,,                       \\[2.4ex]
    && |\wr                                            \\
    && \ast_{v_j\in\Gamma_{\hat{f(L)}}^{(0)}} \pi_1(\hat{L}_{v_j})
   \end{array}
  $$
  where $\bigvee_{ v_j\in\Gamma_{\hat{f(L)}}^{(0)} } \hat{L}_{v_j}$
   is the bouquet of
   $\{\hat{L}_{v_j}:v_j\in \Gamma_{\hat{f(L)}}^{(0)}\}$
   following a(ny) spanning tree of $\Gamma_{\hat{f(L)}}$.
  Its fundamental group is isomorphic to the free-product
   $\ast_{v_j\in\Gamma_{\hat{f(L)}}^{(0)}} \pi_1(\hat{L}_{v_j})$
   of the groups
   $\pi_1(\hat{L}_{v_j})\,$, $v_j\in \Gamma_{\hat{f(L)}}^{(0)}\,$.
 Let $\Rep(\,\bullet\,,\,U(\hat{r}))$
  be the representation variety of the group $\,\bullet\,$
  on $U(\hat{r})$ (without modding out the $U(\hat{r})$-action
  from the $U(\hat{r})$-action on itself by conjugation).
 Then, the above sequence and isomorphisms induce a morphism
  $$
   \Rep(\pi_1(\hat{N}), U(\hat{r}))\;
   \longrightarrow\;
   \times_{v_j\in\Gamma_{\hat{f(L)}}^{(0)}}\,
                       \Rep(\pi_1(L_{v_j}),U(\hat{r}))
  $$
  via pulling back a representation.
 In general, this map is neither injective nor surjective,
  and can have positive-dimensional fibers.
 This implies that
 in the construction of a Morse cobordism of A-branes,
  \begin{itemize}
   \item[$\cdot$]
    the choice of the isomorphic class of
    $(\hat{\cal E}_{(\mu_0,\mu_0+\delta)},\nabla^{(\mu_0,\mu_0+\delta)})$
    determines the isomorphic class of the whole
    $(\hat{\cal E}_{(\mu_0-\delta,\mu_0+\delta)},
      \nabla^{(\mu_0-\delta,\mu_0+\delta)})\,$;

   \item[$\cdot$]
    but
    the choice of the isomorphic class of
    $(\hat{\cal E}_{(\mu_0-\delta,\mu_0)},\nabla^{(\mu_0-\delta,\mu_0)})$
    {\it does not} determine the isomorphic class of the whole
    $(\hat{\cal E}_{(\mu_0-\delta,\mu_0+\delta)},
      \nabla^{(\mu_0-\delta,\mu_0+\delta)})\,$.
  \end{itemize}
 In particular, for the simplest class of A-branes
  that are realized as embedded special Lagrangian submanifolds
   $(L,V,\nabla)$ with a vector bundle and $U(r)$ flat connection
   on the Calabi-Yau $3$-fold $Y$,
  assume that the $[L]$-attractor flow $\gamma$ bends-and-breaks $L$
   to a union $L_1\cup\cdots\cup L_q$ of embedded special Lagrangian
   submanifolds with only isolated transverse intersections.
 Then, (in the $\mu$ decreasing direction, cf.~footnote~18)
  $(V,\nabla)$ can be driven along and bend-and-break accordingly
  to a collection $(L_i, V_i,\nabla_i)$, $i=1,\,\ldots\,,\,q\,$
  of A-branes of smaller volume,
  each of which may continue to flow along its associated
   $[L_i]$-attractor flow.
 However, in the opposite ($\mu$ increasing) direction,
  when one try to assemble a collection $(L_i,V_i,\nabla_i)$,
   $i=1\,\,\ldots\,,\, q$, in this simplest class
   to a single A-brane $(L,V,\nabla)$ with $[L]=[L_1]+\,\cdots\,+[L_q]$,
  then, even if $L$ can be constructed,
   there remains an ambiguity on the choices of $(V,\nabla)$ on $L$.
}\end{remark}

%
%
%

\newpage


\appendix

\begin{flushleft}
{\bf\large Appendix.}
\end{flushleft}
\setcounter{section}{1}
\subsection{Desingularizations of
    immersed special Lagrangian submanifolds
    with transverse intersections and their moduli space
    \`{a} la Joyce.}
Calibrated geometry background and results from [Joy3]
 that are needed for the current work are collected here.
Joyce' results work for the more general almost Calabi-Yau $m$-folds
 with $m>2$.
Here, we only state his results in the special case:
 Calabi-Yau $m$-folds, with $m>2$,
 (and with a slight modification of notations to fit
  the main contents of the notes).
See also [Bu], [Ch], [Lee], and [Mar] for related study/results.

\bigskip

\begin{flushleft}
{\bf Generalization of McLean [McL] to a family.}
\end{flushleft}
\begin{definition}
{\bf [family moduli space].}
{\rm ([Joy3: II.~Sec.~2.3, Definition~2.12].)}
{\rm
 Let
  $\{(M,J^s,\omega^s,\Omega^s):s\in {\cal F}\}$
    be a smooth family of deformations of a Calabi-Yau $m$-fold
    $(M,J.\omega,\Omega)$,
   with base space ${\cal F}\subset {\Bbb R}^d$,
   and
  $N$ be a compact special Lagrangian $m$-manifold
   in $(M,J,\omega,\Omega)$.
 Define the
  {\it moduli space ${\cal M}_N^{\cal F}$ of deformations of $N$
       in the family ${\cal F}$}
  to be the set of pairs $(s,\hat{N})$ for which
   $s\in {\cal F}$  and
   $\hat{N}$ is a compact special Lagrangian $m$-manifold
    in $(M,J^s.\omega^s,\Omega^s)$ which is diffeomorphic to $N$
    and isotopic to $N$ in $M$.
 Define a {\it projection}
  $\pi^{\cal F}:{\cal M}^{\cal F}_N\rightarrow {\cal F}$
  by $\pi^{\cal F}(s,\hat{N})=s$.
 Then
  ${\cal M}^{\cal F}_N$ has a natural topology and
  $\pi^{\cal F}$ is continuous.
}\end{definition}

\begin{theorem}
{\bf [deformation in family].}
{\rm ([Joy3: II.~Sec.~2.3, Theorem~2.13] and
      [Mar: Sec.~3.2, Theorem~3.21].)}
 Let
  $\{(M,J^s,\omega^s,\Omega^s):s\in {\cal F}\}$
    be a smooth family of deformations of a Calabi-Yau $m$-fold
    $(M,J.\omega,\Omega)$,
   with base space ${\cal F}\subset {\Bbb R}^d$.
 Suppose
  $N$ is a compact special Lagrangian $m$-manifold
   in $(M,J,\omega,\Omega)$
   with $[\omega^s|_N]=0$ in $H^2(N;{\Bbb R})$ and
        $[\Imaginary\Omega^s|_N]=0$ in $H^m(N;{\Bbb R})$
    for all $s\in {\cal F}$.
 Let
  ${\cal M}_N^{\cal F}$ be the moduli space ${\cal M}_N^{\cal F}$
   of deformations of $N$ in the family ${\cal F}$  and
  $\pi^{\cal F}:{\cal M}^{\cal F}_N\rightarrow {\cal F}$
   the natural projection.

 Then
  ${\cal M}_N^{\cal F}$ is a smooth manifold of dimension
   $d+b^1(N)$ and
  $\pi^{\cal F}:{\cal M}^{\cal F}_N\rightarrow {\cal F}$
   is a smooth submersion.
 {For} small $s\in {\cal F}$,
  the moduli space ${\cal M}_N^s:=(\pi^{\cal F})^{-1}(s)$
  of deformations of $N$ in $(M,J^s,\omega^s,\Omega^s)$
  is a nonempty smooth manifold of dimension $b^1(N)$,
  with ${\cal M}_N^0={\cal M}_N$.
\end{theorem}

\bigskip

\begin{flushleft}
{\bf Transverse pair $(\Pi^+,\Pi^-)$ of special Lagrangian $m$-planes
     in ${\Bbb C}^m$.}
\end{flushleft}
([Joy3: V: Sec.~9.1]; see also [Bu: Sec.~1] and [Lee: Sec.~2].)
Let $(\Pi^+,\Pi^-)$
 be a pair of special Lagrangian $m$-planes $\simeq {\Bbb R}^m$
 in ${\Bbb C}^m$
 that intersect transversely, i.e.\ $\Pi^+\cap \Pi^-=\{0\}$.
Then, there exists $B\in \SU(m)$
      and $\phi_1\,,\,\cdots\,,\,\phi_m\in (0,\pi)$
 such that
 $$
  B(\Pi^+)\;=\; \Pi^0
   \hspace{2em}\mbox{and}\hspace{2em}
  B(\Pi^-)\;=\; \Pi^{\mbox{\scriptsize\boldmath $\phi$}}\,,
 $$
 where
 $$
  \Pi^0\;=\;\{(x_1\,,\,\cdots\,,\,x_m)\,:\, x_j\in {\Bbb R}\}
   \hspace{2em}\mbox{and}\hspace{2em}
  \Pi^{\mbox{\scriptsize\boldmath $\phi$}}\;
   =\; \{(e^{i\phi_1}x_1\,,\,\cdots\,,\,e^{i\phi_m}x_m)\,:\,
                                            x_j\in {\Bbb R}\}\,.
 $$
Furthermore,
 $\phi_1\,,\,\cdots\,,\,\phi_m$ are unique up to order
  --  and hence become unique under the assumption
      $\,0<\phi_1\le\cdots\le\phi_m<\pi\,$ --  and
 $\,\phi_1+\cdots+\phi_m = k\pi\,$
  for some $k\in\{1\,,\,\ldots\,,\,m-1\}\,$.

\begin{definition}
{\bf [characteristic angles and type].} {\rm
 The unique angles $0<\phi_1\le\cdots\le\phi_m<\pi$
  is called the {\it characteristic angles} of the pair $(\Pi^+,\Pi^-)$
  and
 $k$ is called the {\it type} of the pair $(\Pi^+,\Pi^-)$
  at their intersection $0$.
}\end{definition}

\noindent Note that if the characteristic angles and type of
$(\Pi^+,\Pi^-)$ at $0$
 are $0<\phi_1\le\cdots\le\phi_m<\pi$ and $k$ respectively,
then the characteristic angles and type of $(\Pi^-,\Pi^+)$ at $0$
 are $0<\pi-\phi_m\le\cdots\le\pi-\phi_1<\pi$ and $m-k$ respectively.

\bigskip

\begin{flushleft}
{\bf Lawlor necks $L^{\mbox{\scriptsize\boldmath $\phi$},A}$
     in ${\Bbb C}^m$.}
\end{flushleft}
([Harv: pp.~139-144] and [La];
 cf.\ [Joy3: V.~Example~6.11] and also [Bu], [Lee], and [Mar].)

\bigskip

\noindent
$(a)$ {\it Lawlor necks
      $L^{\mbox{\scriptsize\boldmath $\phi$},A}$.}\hspace{1em}
Let $m>2$ and $a_1,\,\ldots\,,\,a_m>0$.
Define polynomial $P$ by
 $$
  P(x)\;=\;
   \frac{(1+a_1x^2)\,\cdots\,(1+a_mx^2)-1}{x^2}
 $$
and real numbers $\phi_1,\,\ldots\,,\phi_m$ and $A$ by
 $$
  \phi_k\;
  =\; a_k\int_{-\infty}^{\infty}\frac{dx}{(1+a_k)\sqrt{P(x)}}
  \hspace{1em}\mbox{and}\hspace{1em}
  A\;=\; \omega_m(a_1\,\cdots\,a_m)^{-\frac{1}{2}}\,,
 $$
  where $\omega_m$ is the volume of the unit sphere in ${\Bbb R}^m$.
Note that
 $$
  \phi_k\in (0,\pi)
   \hspace{1em}\mbox{with}\hspace{1em}
  \phi_1+\,\cdots\,+\phi_m\;=\; \pi\,,
 \hspace{2em}
 A>0\,,
 $$
 and that the correspondence
  $(a_1,\,\ldots\,, a_m)\mapsto (\phi_1,\,\ldots\,, \phi_m, A)$,
  with the conditions above, is one-to-one.
Define functions
 $z_k:{\Bbb R}\rightarrow {\Bbb C}$, $k=1,\,\ldots\,m\,$, by
 $$
  z_k(y)\; =\; e^{i\psi_k(y)}\sqrt{a_k^{-1}+y^2}\,,
  \hspace{1em}\mbox{where}\hspace{1em}
  \psi_k(y)\; =\;
   a_k\int_{-\infty}^{y}\frac{dx}{(1+a_kx^2)\sqrt{P(x)}}\,.
 $$
Write {\boldmath $\phi$}$=(\phi_1,\,\ldots,\,\phi_m)$  and
 define a submanifold $L^{\mbox{\scriptsize\boldmath $\phi$},A}$
 in ${\Bbb C}^m$ by
 $$
  L^{\mbox{\scriptsize\boldmath $\phi$},A}\;  =\;
  \{\, (z_1(y)x_1,\,\ldots\,,z_m(y)x_m)\,:\,
   y\in{\Bbb R},\, x_k\in{\Bbb R},\, x_1^2+\,\cdots\,x_m^2=1\, \}\,.
 $$
Then {\it
 $L^{\mbox{\scriptsize\boldmath $\phi$},A}$
 is a closed embedded special Lagrangian submanifold in ${\Bbb C}^m$
 that is diffeomorphic to $S^{m-1}\times {\Bbb R}$.}
It intersects each $\Pi^{\mbox{\scriptsize\boldmath $\psi$}(y)}$,
 $y\in {\Bbb R}$, along an ellipsoid,
 where {\boldmath $\psi$}$(y)=(\psi_1(y),\,\cdots\,\psi_m(y))$.
In terms of [Joy3: I.~Definition~7.1],
 $L^{\mbox{\scriptsize\boldmath $\phi$},A}$
  is {\it asymptotically conical,
  with rate $2-m$ and
  cone the type-$1$ transverse pair
  $\Pi^0\cup \Pi^{\mbox{\scriptsize\boldmath $\phi$}}$
  of special Lagrangian planes}.
Furthermore, for $t>0$,
 let $t:{\Bbb C}^m\rightarrow {\Bbb C}^m$
  be the associated dilation map (i.e.\ multiplication by $t$).
Then
 $t\cdot L^{\mbox{\scriptsize\boldmath $\phi$},A}
  = L^{\mbox{\scriptsize\boldmath $\phi$},t^mA}$.

%
%
%
%
%
%

\bigskip

\begin{flushleft}
{\bf Desingularizations of immersed special Lagrangian submanifolds
     with transverse intersections.}
\end{flushleft}
\begin{theorem}
{\bf [desingularization in a Calabi-Yau].}
{\rm ([Joy3: V.~Sec.~9.2, Theorem~9.7].)}
 Let
  \begin{itemize}
   \item[$\cdot$]
    $(M,J,\omega,\Omega)$ be a Calabi-Yau $m$-fold with $m>2$,

   \item[$\cdot$]
    $\iota:X\rightarrow M$
     be a compact, immersed special Lagrangian $m$-manifold in $M$,

   \item[$\cdot$]
    $x_1,\,\ldots\,, x_n\in M$ be transverse self-intersection
      points of $X$ with type $1$,

   \item[$\cdot$]
    $x_i^{\pm}\in X$ be the pair of points in $\iota^{-1}(x_i)$
     such that the type of the pair
     $(\iota_{\ast}T_{x_i^+}X,\iota_{\ast}T_{x_i^-}X)$ at
      $0\in T_{x_i}M$ is $1$,
    and

   \item[$\cdot$]
    $X_1,\,\ldots\,, X_q$, where $q=b^0(X)$,
    be the connected components of $X$.
  \end{itemize}
 Suppose $A_1,\,\ldots\,, A_n >0$ satisfy
  $$
   \sum_{\mbox{\scriptsize
         $\begin{array}{c}
           i\in\{1,\,\ldots\,, n\},\\
            x_i^+ \in X_k \end{array}$}} A_i \;
    -\,
   \sum_{\mbox{\scriptsize
         $\begin{array}{c}
           i\in\{1,\,\ldots\,, n\},\\
            x_i^- \in X_k \end{array}$}} A_i\;\;
   =\;\; 0
   \hspace{3em}\mbox{for all $\;k=1,\,\ldots\,, q\,.$}
  $$
 Let $N$ be the (oriented multiple self-)connected sum of $X$
  at the pairs of points $\{(x_i^+,x_i^-)\}_{i=1}^n$.
 Suppose that $N$ is connected.
 Then
  \begin{itemize}
   \item[$\cdot$]
    there exist
     an $\epsilon>0$ and
     a smooth family
      $\{\iota^t:\tilde{N}^t\rightarrow M\,|\, t\in (0,\epsilon]\}$
       of compact, immersed special Lagrangian $m$-manifolds
       in  $(M,J,\omega,\Omega)$,
      with $\tilde{N}^t$ diffeomorphic to $N$,
     such that $\iota^t$ is constructed by gluing a Lawlor neck
      $t\cdot L^{\pm,A_i}\,(=L^{\pm,t^mA_i})$ into $\iota$ at $x_i$
      for $i=1\,,\,\ldots\,,\,n\,$.
  \end{itemize}
 In the sense of currents, $\iota^t\rightarrow \iota$ as $t\rightarrow 0$.
 If $x_1\,,\,\cdots\,,\,x_n$ are the only self-intersection points
   of $\iota$,
  then $\iota^t$ is an embedding.
\end{theorem}

In general, it can happen that
 there is no smoothing for a compact immersed special
  Lagrangian submanifold at its transverse self-intersection point
  for a fixed Calabi-Yau manifold.
In this case, Joyce proves another result:

\begin{theorem}
{\bf [desingularization in a family of Calabi-Yau's].}
{\rm ([Joy3: V.~Sec.~9.3, Theorem~9.8].)}
 Let
  \begin{itemize}
   \item[$\cdot$]
    $(M,J,\omega,\Omega)$ be a Calabi-Yau $m$-fold with $m>2$,

   \item[$\cdot$]
    $\iota:X\rightarrow M$
     be a compact, immersed special Lagrangian $m$-manifold in $M$,

   \item[$\cdot$]
    $x_1,\,\ldots\,, x_n\in M$ be transverse self-intersection
      points of $X$ with type $1$,

   \item[$\cdot$]
    $x_i^{\pm}\in X$ be the pair of points in $\iota^{-1}(x_i)$
     such that the type of the pair
     $(\iota_{\ast}T_{x_i^+}X,\iota_{\ast}T_{x_i^-}X)$ at
      $0\in T_{x_i}M$ is $1$,
    and

   \item[$\cdot$]
    $X_1,\,\ldots\,, X_q$, where $q=b^0(X)$,
    be the connected components of $X$.
  \end{itemize}
 Let $N$ be the (oriented multiple self-)connected sum of $X$
  at the pairs of points $\{(x_i^+,x_i^-)\}_{i=1}^n$.
 Suppose that $N$ is connected.

 Suppose $\{(M,J^s,\omega^s,\Omega^s)\,:\, s\in {\cal F}\}$
  is a smooth family of deformations of $(M,J,\omega,\Omega)$
  with $\iota^{\ast}([\omega^s])=0$ in $H^2(X;{\Bbb R})$
  for all $s\in {\cal F}$.
 Let $A_1,\,\cdots\,, A_n>0$.
 Define ${\cal G}\subset {\cal F}\times (0,1)$ to be the subset of
  $(s,t)\in {\cal F}\times (0,1)$ with
  $$
   [\Imaginary\Omega^s]\cdot [X_k]\; =\;
    t^m \cdot
     \raisebox{-.5em}{$\left(\rule{0em}{2em}\right.$}
      \sum_{\mbox{\scriptsize
            $\begin{array}{c}
              i\in\{1,\,\ldots\,, n\},\\
               x_i^+ \in X_k \end{array}$}} A_i \;
      -\;
      \sum_{\mbox{\scriptsize
            $\begin{array}{c}
              i\in\{1,\,\ldots\,, n\},\\
               x_i^- \in X_k \end{array}$}} A_i
     \;\raisebox{-.5em}{$\left.\rule{0em}{2em}\right)$}
   \hspace{3em}\mbox{for all $\;k=1,\,\ldots\,, q\,.$}
  $$
 Then there exist
  $\epsilon\in (0,1)$, $\kappa>1$  and
  a smooth family
   $$
    \left\{\,
     \iota^{s,t}:\tilde{N}^{s,t}\rightarrow M   \,|\,
       (s,t)\in {\cal G}\,,\;   t\in (0,\epsilon]\,,\;
       |s|\le t^{\kappa+m/2}
    \,\right\}
   $$
  such that
  \begin{itemize}
   \item[$\cdot$]
    $\iota^{s,t}:\tilde{N}^{s,t}
                 \rightarrow (M, J^s,\omega^s,\Omega^s)$
     is a compact , nonsingular special Lagrangian
     $m$-manifold in $(M, J^s,\omega^s,\Omega^s)$,
     with $\tilde{N}^{s,t}$ diffeomorphic to $N$,
    constructed by gluing a Lawlor neck
     $t\cdot L^{\pm, A_i}\, (= L^{\pm, t^mA_i})$
     into $\iota$ at $x_i$ for $i=1,\,\ldots\,, n$.
  \end{itemize}
  In the sense of currents,
   $\iota^{s,t}\rightarrow \iota$ as $s, t\rightarrow 0$.
  If $x_1,\,\cdots\,,\, x_n$
    are the only self-intersection points of $\iota$,
   then $\iota^{s,t}$ is embedded.
\end{theorem}

The following observation is needed to apply Joyce' theorem
 to our situation:

\begin{observation}
{\bf [restriction $|s|\le t^{\kappa+m/2}\,$].}
{\rm (Cf.\ [Joy3: IV.\ Remark after Theorem~7.9].)}
 In Theorem~A.1.5, the restriction $|s|\le t^{\kappa+m/2}$
  can be replaced by any restriction of the form:
  $$
   |s|\; \le\; C\,t^{\kappa+m/2}
    \hspace{2em}\mbox{for some constant $\,C>0$}
  $$
  with the glued Lawlor neck $t\cdot L^{\pm,A_i}=L^{\pm, t^mA_i}$
   replaced by $C^{\,-1/(\kappa+m/2)}\,t\cdot L^{\pm, A_i}$
   in the construction.
 Furthermore, $\kappa>1$ can be chosen to be arbitrarily close to $1$.
\end{observation}

\bigskip

\begin{proof}
 The first statement follows from a reparameterization of the parameter
  space with coordinates $(s,t)$ in the procedure of constructing
  a family of special Lagrangian submanifolds in Calabi-Yau manifolds,
  parameterized by $(s,t)$.
 {For} the second statement,
  recall (cf.~[Joy3: IV, proof of Theorem~7.9])
  that the constraint $|s|\le t^{\kappa+m/2}$ with $\kappa>1$
   becomes part of a sufficient condition for bounds\footnote{Here,
                 $\varepsilon^{s,t}$ is a smooth function on
                   a Lagrangian submanifold $N^{s,t}$ constructed
                   via gluing Lawlor necks to $X$
                   in ways parameterized by $(s,t)$ .
                  It measures how close $N^{s,t}$ is
                   to being special Lagrangian.
                 $W^{s,t}\simeq {\Bbb R}^q$
                   is a finite dimensional subspace of
                   $C^{\infty}(N^{s,t})$
                   that approximates the ${\Bbb R}$-span of
                   eigenfunctions with small eigenvalues
                   of a second-order elliptic operator
                   arising from the special Lagrangian condition, and
                 $\pi_{W^{s,t}}:L^2(N^{s,t})\rightarrow W^{s,t}$
                   is the projection onto using the $L^2$-inner product.
                 As we won't need the precise expression of these
                   for this note  and
                  to give a complete definition requires
                   three pages of related definitions,
                  we refer readers directly to
                   [Joy3: IV, Sec.~7.1, Definitions~7.1 and 7.2]
                   (resp.\ [Joy3: IV, Sec.~7.1, Definition~7.3];
                           [Joy3: III, Sec.~5.2, Theorem~5.3];
                           [Joy3: I, Sec.~2.2])
                   for details of $\varepsilon^{s,t}$
                   (resp.\
                    $W^{s,t}$ and $\pi_{W^{s,t}}$;
                    constant $A_2$;
                    the various Banach/Sobolev spaces $C^k$, $L^p$, and
                     norms $\|\cdot\|_{C^k}$, $\|\cdot\|_{L^p}$).}
    \begin{eqnarray*}
    & \|\varepsilon^{s,t}\|_{L^{2m/(m+2)}}\;
       \le\; A_2\,t^{\kappa+m/2}\,,\hspace{1em}
      \|\varepsilon^{s,t}\|_{C^0}\;
       \le\; A_2\,t^{\kappa-1}\,,\hspace{1em}
      \|d\,\varepsilon^{s,t}\|_{L^{2m}}\;
       \le\; A_2\,t^{\kappa-3/2}\,, & \\[.6ex]
    &  \|\pi_{W^{s,t}}(\varepsilon^{s,t})\|_{L^1}\;
        \le\; A_2\,t^{\kappa+m-1} &
    \end{eqnarray*}
   to hold
  if $\kappa>1$ is a solution to the following system of
     inequalities\footnote{Here,
                       specified to the current situation we need,
                       $\mu_i$
                        (resp.\ $\lambda_i$) is the rate
                        of the conical special Lagrangian submanifold $X$
                        at $x_i$
                        (resp.\ Lawlor neck $L^{\pm, A_i}$)  and
                        $\tau$ is a parameter that governs in part
                         how the Lawlor necks $L^{\pm, A_i}$ are glued
                         into $X^{\prime}:=X-\{x_1,\,\cdots,\,x_n\}$
                         in the construction of $N^{s,t}$.
                       See [Joy3: I, Sec.~3.3, Definition~3.6]
                        (resp.\
                         [Joy3: I, Sec.~7, Definition~7.1];
                         [Joy3: IV, Sec.~7.1, Definition~7.2])
                        for details.}
     ([Joy3: IV: Sec.~7.3, Eq'ns (99), (100), (101), (102)]):
   $$
    \begin{array}{rcl}
     -\tau(1+m/2)+\tau(\mu_i-2)\,\ge\,\kappa+m/2\,,\hspace{.7em}
      && \tau(1+m/2)+(1-\tau)(2-\lambda_i)\,\ge\,\kappa+m/2\,,\\[.6ex]
     \tau(\mu_i-2)\,\ge\,\kappa-1\,,\hspace{2.1em}
      &&\hspace{6.2em}
        (1-\tau)(2-\lambda_i)\,\ge\,\kappa-1\,,\\[.6ex]
     -\tau/2+\tau(\mu_i-2)\,\ge\,\kappa-3/2\,,\hspace{1.1em}
      &&\hspace{2.6em}
        -\tau/2+(1-\tau)(2-\lambda_i)\,\ge\,\kappa-3/2\,,\\[.6ex]
     (m+1)\tau\,\ge\,\kappa+m-1\,,
      && \hspace{6em}i\,=\,1\,,\,\ldots\,,\,n\,,
    \end{array}
   $$
   where
    $\mu_i\in (2,3)$ can be chosen to be arbitrarily close to $2$
     (cf.~[Joy3: I, Sec.~5.3, Theorem~5.5]),
    $\lambda_i=2-m$ in the current situation, and
    $0< \max_{i=1,\ldots,n}
        \{ \frac{m}{m+1}\,,\,\frac{m+2}{2\mu_i+m-2} \}
        <\tau < 1$.

 Since $\mu_i$ can be chosen to be arbitrarily close to $2$,
  $\tau$ can be chosen to be arbitrarily close to $1$ as well.
 Let $\tau=1-\delta_0$, $\mu_i=2+\delta_i$, and  note that
  $\lambda_i<\frac{1}{2}(2-m)$.
 Then, the above system of inequalities has nonempty solution
   for $\kappa>1$
  if and only if
   $\delta_0$, $\delta_i$, $i=1,\,\ldots\,,\,n$, satisfy inequalities
   $$
    \delta_i\;>\; (1+\frac{m}{2})\,\frac{\delta_0}{1-\delta_0}\,,
     \hspace{1em}\mbox{and}\hspace{1em}
    0\;<\; \delta_0\;<\; \frac{1}{m+1}\,.
   $$
 Note that
  the latter system does have nonempty positive solutions
   for $\delta_0,\,\delta_1,\,\ldots\,,\delta_n$  and that
  these solutions can be chosen to be arbitrarily close to $0$.
 {For} any such solution $(\delta_0,\,\delta_1,\,\ldots\,,\delta_n)$,
  any $\kappa$ that satisfies
  $$
   1\;<\;\kappa\;
      \le\; 1\, +\,
            \min_{i=1,\,\ldots\,,\,n}\left\{
            (1-\delta_0)\delta_i-\delta_0(1+\frac{m}{2})\,,\,
            \delta_0(\frac{m}{2}-1)\,,\,
            1-\delta_0(m+1)         \right\}
  $$
  lies in the solution set of the above system of inequality
  for $\kappa$.
 In particular, $\kappa>1$ can be chosen to be arbitrarily close to $1$.

\end{proof}

%
%
%
%
%
%
%
%
%
%

\bigskip

\begin{flushleft}
{\bf The immersed version of Joyce' results.}
\end{flushleft}
([Joy3: I.~Sec.~1 and II.~Sec.~8.1].)
Nearly all the results of Joyce [Joy3] generalize immediately to
 immersed special Lagrangian submanifolds with conical singularities,
 with only superficial change,
 due to the following version of Lagrangian Neighborhood Theorem:

\bigskip

\begin{theorem}
{\bf [immersed Lagrangian].}
 Let $(M,\omega)$ be a symplectic manifold and
  $f:L\rightarrow M$ be a compact immersed Lagrangian submanifold.
 Then there exist
  a neighborhood $U\subset T^{\ast}L$ of the zero-section  and
  an immersion $\phi:U\rightarrow M$
  such that
   $\phi|_L=f$ and
   $\phi^{\ast}\omega=\omega_{\scriptsizecan}$,
    where $\omega_{\scriptsizecan}$
     is the canonical symplectic form on $T^{\ast}L$.
\end{theorem}

\newpage
\baselineskip 13pt

\end{document}